\documentclass[12pt]{amsart}

\usepackage[numbers,sort]{natbib}
\usepackage{amsmath}
\usepackage[psamsfonts]{amssymb}
\usepackage[all]{xy}
\usepackage{mathtools}
\usepackage{graphicx}
\usepackage[bitstream-charter]{mathdesign}
\usepackage{hyperref}
\usepackage{color}
\usepackage{dsfont}
\usepackage{accents}
\usepackage{soul}

\let\oldproofname=\proofname
\renewcommand{\proofname}{\rm\bf{\oldproofname}}

\usepackage[symbol]{footmisc}

\newtheorem{theorem}{Theorem}[section]
\newtheorem{prop}[theorem]{Proposition}
\newtheorem{lemma}[theorem]{Lemma}
\newtheorem{cor}[theorem]{Corollary}

\newtheorem{conj}[theorem]{Conjecture}

\newtheorem{ex}[theorem]{Example}

\theoremstyle{definition}
\newtheorem{dfn}[theorem]{Definition}
\theoremstyle{remark}
\newtheorem{remark}[theorem]{Remark}

\def\B{\mathcal{B}}
\def\R{\mathbb{R}}
\def\P{\mathcal{P}}
\def\Z{\mathbb{Z}}
\def\N{\mathbb{N}}

\def\U{\mathcal {U}}
\def\cF{\mathcal{F}}
\def\vol{\text{Vol}}
\def\supp{\text{supp}}

\def\S{\mathcal{S}}

\DeclareMathOperator{\coker}{coker}

\def\eu{\text{eu}}

\newcommand{\Vol}{\operatorname{Vol}}
\newcommand{\im}{\operatorname{im}}
\newcommand{\interior}{\operatorname{int}}
\newcommand{\colim}{\operatorname{colim}}

\newcommand{\Cnkp}{C_{n,k,p}}

\def\mrm#1{{\mathrm{#1}}}
\def\cl#1{{\mathcal{#1}}}

\def\bb#1{{\mathbb{#1}}}

\newcommand{\ol}[1]{\overline{#1}}

\newcommand{\la}{\lambda}
\newcommand{\al}{\alpha}
\newcommand{\om}{\omega}
\newcommand{\eps}{\epsilon}
\newcommand{\id}{\mathrm{id}}

\newcommand{\de}{\delta}
\newcommand{\fin}{\mrm{fin}}

\newcommand{\bs}{\bigskip}

\newcommand{\grad}{\nabla}
\newcommand{\del}{\partial}
\renewcommand{\div}{\mrm{div}}

\newcommand{\bK}{{\mathbb{K}}}

\renewcommand{\odot}[1]{{\accentset{\circ}{#1}}}
\newcommand{\til}[1]{\widetilde{#1}}

\begin{document}

\title{Coarse nodal count and topological persistence}

\author{Lev Buhovsky$^1$, Jordan Payette$^2$, Iosif Polterovich$^3$, Leonid Polterovich$^4$,
Egor Shelukhin$^5$ and Vuka\v{s}in Stojisavljevi\'{c}$^6$}
\date{}

\footnotetext[1]{Partially supported by ERC Starting Grant 757585 and ISF grant 2026/17.}
\footnotetext[2]{Partially supported by ERC Starting Grant 757585 and FRQNT postdoctoral scholarship.}
\footnotetext[3]{Partially supported by NSERC and FRQNT.}
\footnotetext[4]{Partially supported by the Israel Science Foundation
grant 1102/20.}
\footnotetext[5]{Partially supported by NSERC, Fondation Courtois, Alfred P. Sloan Foundation.}
\footnotetext[6]{Partially supported by ISF grant 667/18, ISF grant 1102/20, CRM-ISM postdoctoral fellowship and ERC Starting Grant 851701.}

\maketitle

\begin{abstract}
	Courant's theorem implies that the number of nodal domains of a Laplace eigenfunction is controlled by the corresponding eigenvalue. Over the years, there have been various attempts to find an appropriate generalization of this statement in different directions. We propose a new take on this problem using ideas from topological data analysis. We show  that if one counts the nodal domains in a coarse way, basically ignoring small oscillations, Courant's theorem extends to linear combinations of eigenfunctions, to their products, to other operators, and to higher topological invariants of nodal sets. We also obtain a coarse version of the Bézout estimate for common zeros of linear combinations of eigenfunctions.  We show that our results are essentially sharp and that the coarse count is necessary, since these extensions fail in general for the standard count. Our approach combines multiscale polynomial approximation in Sobolev spaces  with new results in the theory of persistence modules and barcodes.
\end{abstract}

\setcounter{tocdepth}{1}
\tableofcontents

\section{Introduction and main results}

\subsection{Measuring oscillations}
\label{subsec:osc}
The present paper focuses on the interplay between topology and analysis
of smooth functions, with links to spectral and algebraic geometry. The topological function theory  deals
with invariants of functions under diffeomorphisms and, roughly speaking, enables one
to study oscillations of functions by looking at the topology  of its sublevel sets. This theory
is based on persistence modules and barcodes, a mathematical apparatus originated in topological
data analysis. On the analysis side, we consider measurements of functions based on the Sobolev scale,
often in the context of eigenfunctions of elliptic operators, as well as  their linear combinations and products.

Let $M$ be a smooth compact connected $n$-dimensional Riemannian manifold, possibly with a non-empty boundary, and let $E \to M$ be a rank $l$ real vector bundle over $M$. Given a section $s: M \to E$, we introduce its zero (or nodal) set $Z_s = \{s=0\}$, and denote by  $z_r(s)=\dim H_r(Z_s)$  and $m_r(s)=\dim H_r(M \setminus Z_s)$  the Betti  numbers of the zero set and its complement, respectively. Here and further on, $H_r(X)$ stands for the $r$-th singular homology group of a subset $X \subset M$
with coefficients in a field. 

 The cases of particular importance are $l=1$, when $Z_s$ is generically a hypersurface in $M$ and the  connected components of $M \setminus Z_s$ are called {\it nodal domains}, and also $l=n$ when generically $Z_s$  is a finite set.  The traditional objects of study are  the count of nodal domains $m_0(s)$  and 
the count of 
zeros $z_0(s)$. 

Let us introduce a {\it coarse} version of Betti numbers, called the {\it  persistent  Betti numbers},  as follows.  Let us fix a Riemannian metric
on $M$ and an inner product on $E$. For a smooth section $s: M \to E$ and a number $\delta >0$, put
\begin{equation}\label{eq-m-vsp}
 m_r(s,\delta) = \dim  \mrm{Im}\left(H_r(\{|s|>\delta\}) \to H_r(M \setminus Z_s)\right)\;,
\end{equation}
  and
  \begin{equation}  \label{eq-z-vsp}
  z_r(s,\delta) = \dim \mrm{Im}(H_r(Z_s) \to H_r(\{|s| < \delta\}))\;.
  \end{equation}
In Section \ref{subsec:persbounds} we restate these definitions in the language of the theory of persistence modules.

As an illustration, assume that $E= M \times \R$, so that sections of $E$ are functions
on $M$. Then, given a function $f$,  $m_0(s,\delta)$ is the number of  {\it ``$\delta$-deep"} nodal domains $U$, i.e. such that $\max_U  |f| > \delta$, while other domains are discarded as a topological noise. This approach goes back to \cite{polterovich2007nodal} and has been further developed in \cite{PPS19},
see Section \ref{subsec:CCB} for a discussion.

Assume now that $l=n$, and $s$ is a generic section of $E$ with a finite number of zeros.
Then $z_0(s,\delta)$ counts  only those connected components of $\{|s| < \delta\}$
which contain zeros of $s$. Other connected components are discarded as topological noise.

Let  $||s||_{W^{k,p}}$, $k \in \mathbb{N}$, $p\ge 1$,  be the  Sobolev norm of $s$, see Subsection \ref{subsec:sobolev}  for a precise definition. Recall that this norm is controlled by the $L^p$ norms of the derivatives of $s$ up to the order $k$.
Our first main result is as follows.
\begin{theorem}\label{thm: main 2-vsp}
Let $E$ be a vector bundle with an inner product over a Riemannian manifold $M$ of dimension $n.$ Fix integers $k > n/p$, $0\le r < n,$ and suppose that $s \in W^{k,p}(M;E).$ Then for any $\delta >0$,
\begin{equation}\label{eq-cour-vsp}
m_r(s,\delta) \leq C_1{\delta^{-n/k}} ||s||_{W^{k,p}}^{n/k}+C_2\;,
\end{equation}
and
\begin{equation}\label{eq-bezout-vsp}
z_r(s,\delta) \leq C_1{\delta^{-n/k}} ||s||_{W^{k,p}}^{n/k}+C_2\;,
\end{equation}
where the constant $C_1$ depends only on $M,E,k,p$ and $C_2 = \dim H_r(M).$ 
\end{theorem}
It should be emphasized that this theorem is new and meaningful already for the case when $r=0$, $E= M \times \R$  and the sections are simply functions on $M$.  Moreover,  the result does not hold if  the persistent  Betti numbers are  replaced  by the usual Betti numbers,
 and the powers of $\|s\|$  and $\delta$ in formulas \eqref{eq-cour-vsp} and \eqref{eq-bezout-vsp} are sharp, see Subsection \ref{subsec:sharp} for details. 

A few more remarks are in order. The assumption $k-n/p >0$ guarantees that $s$ is continuous;
otherwise, our  topological considerations are not feasible. The formulation above involving
persistent Betti numbers is not yet an ultimate one:  we shall generalize this result by using
the language of persistence barcodes, see Theorem \ref{thm: main 2} below. In view of Lemma \ref{lma: replace C2} it is sufficient to prove Theorem \ref{thm: main 2}, which implies Theorem \ref{thm: main 2-vsp}, in a weaker form where $C_2$ depends on $M,E,k,p$ like $C_1.$

The first estimates on the magnitude of the oscillations of a smooth function $f$ in terms of the uniform norm of its higher derivatives were obtained by Yomdin \cite{yomdin1985global} (we refer also to \cite{Kronrod, Vitushkin, Ivanov} for earlier related results). Constraints similar to \eqref{eq-cour-vsp}, stated in the language of persistence barcodes are known for $p=\infty$ and $k=1$ \cite{CSEHM} and, in the case of surfaces for $p=k=2$
\cite{PPS19} (see also \cite{polterovich2007nodal} for other related estimates).

Our approach to Theorem \ref{thm: main 2-vsp} combines the theory of persistence modules and barcodes with a multi-scale version of Yomdin's method  based on polynomial approximation of sections on small cubes.  Furthermore, we   obtain bounds on the topology of the nodal sets of these approximations using tools from algebraic geometry, and glue together the data  on different cubes using the Mayer-Vietoris sequence.

As an application of Theorem \ref{thm: main 2-vsp} we present a coarse version of Courant's nodal domain theorem \cite{PPS19, polterovich2007nodal}. We discuss new instances of the coarse Courant theorem in Section \ref{subsec-CC}, in particular, for products of linear combinations of eigenfunctions.  We also present novel applications to a coarse version of
B\'ezout's theorem (Section \ref{subsec-Bez}), which is related to the coarse Courant theorem for products via the Mayer-Vietoris sequence, see Section \ref{subsec:CCB}.  

In a way, these  results provide an answer to a problem posed by  V.~Arnold in 2003  on extending Courant's theorem to {\it ``...the case of systems of equations, describing oscillations of the sections of fibrations whose fiber has dimension $>~1$''} \cite[Problem 2003-10]{Arnold2005}.  Moreover, as shown in Proposition \ref{prop: no}, the coarse approach is essential for such an extension. 

\subsection{Coarse Courant theorem} \label{subsec-CC}
Consider the  following motivating example.  Let  $\Delta f = - \div (\grad f)$ be the Laplace-Beltrami operator associated
to a Riemannian metric on a closed manifold $M$ of dimension $n$. 
It is well-known that the eigenvalues $\lambda_j$ are non-negative. Let us arrange them  in the non-decreasing order with account of multiplicities, and define the counting function $N(\la) =  \# \{ \la_j \leq \la \}$.
The counting function satisfies the Weyl law which implies $N(\la)=O(\la^{n/2})$. Let $f_j$ with $\Delta f_j = \lambda_j f_j$ be any  sequence of eigenfunctions
normalized by the $L_2$-norm,  $\int_M f_j^2 d\text{Vol}=1$.
Courant's  nodal domain theorem states that $m_0(f_j) \leq j$, and combined with the Weyl law it yields
\begin{equation}\label{eq-cour}
m_0(f_j)=O\left(\lambda_j^{n/2}\right).
\end{equation}

Our main finding is that if one replaces the Betti
numbers by their  persistent counterparts, estimate \eqref{eq-cour} can be extended
in several directions:
\begin{itemize}
\item to linear combinations of eigenfunctions, as opposed to single eigenfunctions;
\item to products of linear combinations of eigenfunctions;
\item to persistent Betti numbers in arbitrary degree instead of degree zero;
\item to arbitrary elliptic operator on sections of a vector
bundle instead of the Laplace-Beltrami operator on functions.
\end{itemize}
It should be mentioned that none of these generalizations are possible with the usual Betti numbers, see Proposition \ref{prop: no} below. At the same time, results of this kind are known to hold   for random linear combinations  of eigensections of elliptic operators, see \cite{gaywel17}.

Throughout this section, let $M$ be a  compact  Riemannian manifold of dimension $n$ and let $D$ be a non-negative self-adjoint elliptic pseudo-differential operator of order $q>0$ on the sections of a vector bundle $E$ over $M$ with an inner product. If  $\partial M \neq 0$,  we assume that $D$ is a differential operator of even order $q=2q'$ satisfying Dirichlet boundary conditions (i.e. all the derivatives up to the order $q'-1$ vanish at the boundary).

Let $\mathcal{F}_\la$  denote the subspace spanned by all eigensections of $D$ with eigenvalues $\leq \la$.

\begin{theorem}[coarse Courant] \label{thm: coarse Courant}
Let  $0 \leq r < n$ and $k > n/2$ be integer numbers. 
Then for any $\delta>0$ and any $s \in \mathcal{F}_\la$ with $||s||_{L^2} =1$, 
 \[ m_r(s,\delta) \leq \frac{C_1}{\delta^{n/k}}(\la+1)^{\frac{n}{q}}+C_2,\]
 \[ z_r(s,\delta) \leq \frac{C_1}{\delta^{n/k}}(\la+1)^{\frac{n}{q}}+C_2,\]
 where the constant $C_1$ depends only on $M,E,D,k$  and $C_2 = \dim H_r(M).$
\end{theorem}

\begin{remark}
We note that in the case $0 < \delta \leq 1,$ Theorem \ref{thm: coarse Courant} and Theorems \ref{cor: product}, \ref{cor: Bezout}, \ref{thm: critical points} and \ref{thm: barcode conj delta} below hold for arbitrary positive $k.$	
\end{remark}

We note that since  Theorem \ref{thm: coarse Courant} applies to pseudo-differential operators, it gives a partial answer to a question on a Courant-type bound for the number of nodal domains of the Dirichlet-to-Neumann operator \cite[Open problem 9]{GP2017}, see also  \cite{HasSher2021}. 
   
Another result where a similar bound  holds  concerns the products of linear combinations of eigenfunctions.
\begin{theorem}[coarse Courant for products]\label{cor: product}
Let $E=M\times \mathbb{R}$ and  $f_1,\ldots,f_l \in  \mathcal{F}_\lambda$,  $l \ge 1$,  be $L^2$-normalized linear combinations of eigenfunctions:
 $||f_j||_{L^2} =1$, $j=1,\dots, l$. Set  $f = f_1\cdot \ldots \cdot f_l$, and let  $0 \leq r < n$ be an integer. Then for every $\varepsilon > 0$ there exists an integer $k_0 > n/2$ such that for any  $\delta > 0$ and $k \geq k_0$,
  \[ m_r(f,\delta) \leq \frac{C_1}{\delta^{n/k}} (\la+1)^{\frac{n}{q} + \varepsilon} +C_2,\] \[z_r(f,\delta) \leq \frac{C_1}{\delta^{n/k}} (\la+1)^{\frac{n}{q} + \varepsilon} +C_2,\]
where the constant $C_1$ depends only on $M,D,l,k, \varepsilon$ and $C_2 = \dim H_r(M).$ The integer $k_0$ depends only on $n,q,l,\varepsilon.$ 

\end{theorem}
Theorem \ref{cor: product}  is a consequence of Theorem \ref{thm: main 2} and Proposition \ref{prop: Sobolev of eigenchunk} for functions together with an estimate of the Sobolev $W^{k,2}$ norm of products for $k>n/2$, known as the fractional Leibniz rule in Sobolev spaces (see \cite{Grafakos,Naibo}). With slightly less optimal constants, it can also be proved using the Sobolev trace theorem \cite[p. 121]{EgorovShubin}, see Remark \ref{rmk: trace}.

Up to  $\varepsilon>0$, the exponent in the estimates above can not be improved.  This can be easily seen by considering a  product of eigenfunctions $\sin jx$ and $\sin j y$ on a flat $2$-torus as $j \to \infty$. 

Note that if $f_j \in \mathcal{F}_{\lambda_j}$, $j=1,\dots l$, the above estimates are given in terms of $\lambda=\max_j \lambda_j$. In particular,  they are accurate provided  $\lambda_j$ are comparable to $\lambda$ for all $j$, i.e., there exists a constant $C>0$ such that $1/C \le \lambda_j/\lambda  \le C$. However, for arbitrary $\lambda_j$ these bounds are not sharp.  Theorem \ref{thm: gen courant}  proved in   Section \ref{subsec: prod proof} gives a somewhat more refined version of Theorem \ref{cor: product}, capturing the contributions of the individual $\lambda_j$, albeit still in a non-sharp manner.

\subsection{Coarse B\'ezout theorem}\label{subsec-Bez}
Loosely speaking,  eigenfunctions of the Laplace-Beltrami operator with the eigenvalue $\lambda$ are expected to share some common features with polynomials of degree $\sqrt{\lambda}$ when $\lambda$ is sufficiently large \cite{donnelly1988nodal}. To illustrate this principle,
consider the sphere $S^n$ equipped with the standard spherical
metric. Harmonic homogeneous polynomials of degree $d$ on $\R^{n+1}$ correspond to eigenfunctions
of the Laplace-Beltrami operator with the eigenvalue $d(d+n-1)$. Given eigenfunctions
$f_1, \dots, f_n$ on $S^n$ with the eigenvalues $\lambda_1, \dots, \lambda_n$, the number of
common zeros generically does not exceed $\text{const} \cdot \sqrt{\lambda_1 \cdots \lambda_n}$.
This follows from the standard B\'ezout theorem. Furthermore, it was proved in  \cite{gichev2009some} that the expectation (in a natural probabilistic setting) of the number of common zeros equals $2n^{-n/2}\sqrt{\lambda_1 \cdots \lambda_n}$. Similar bounds for certain homogeneous Riemannian manifolds have been also obtained in \cite{AkhKaz16, AkhKaz17}.

Below we promote another informal principle stating that {\it persistent} topological
characteristics of eigenfunctions are similar to those predicted by algebraic geometry,
where, again, the degrees correspond to the square root of the eigenvalue. For instance,
we prove the following coarse version of B\'ezout's theorem, as an application of Theorem \ref{thm: coarse Courant}.

\begin{theorem} [coarse B\'ezout] \label{cor: Bezout} 
Let, as before, $E=M\times \mathbb{R}$,  $f_1,\dots,f_n \in \mathcal{F}_\la$, $||f_j||_{L^2} = 1$, $j=1,\dots , n$.
Consider  $s = (f_1,\ldots, f_n)$ as a section of the trivial bundle $M \times \R^n$  with the standard metric, and let $k > n/2$ be an integer.
Then for any $\delta>0$,
\[ z_0(s,\delta) \leq \frac{C_1}{\delta^{n/k}}  (\la+1)^{\frac{n}{q}}+1,\]
where the constant $C_1$ depends only on $M,D,k.$ 
\end{theorem}

Note that Theorem \ref{cor: Bezout} agrees with the B\'ezout estimate for Laplace eigenfunctions on the round sphere corresponding to the same eigenvalue $\lambda$. As in the case of the coarse Courant theorem for products, if $f_j \in \mathcal{F}_{\lambda_j}$, the estimate above is sharp provided   all $\lambda_j$ are comparable to $\lambda$. A more general version of the coarse B\'ezout theorem capturing the contributions of different $\lambda_j$ is presented in Theorem \ref{thm: Bezout general}. In fact, it is tempting to make the following
\begin{conj}\label{conj: Bezout} Let $f_j \in \mathcal{F}_{\lambda_j}$  and $s=(f_1, \dots, f_n)$ be as in Theorem \ref{cor: Bezout}. Then
\[ z_0(s,\delta) \leq \frac{C_1}{\delta^{n/k}}   \left((\la_1+1) \cdot \ldots \cdot (\la_n+1)\right)^{\frac{1}{q}} +1,\]
where the constant $C_1$ depends only on  $M,D,k.$ 
\end{conj}

Theorem \ref{cor: Bezout} holds for $z_r$ for all $0 \leq r < n$ and Conjecture \ref{conj: Bezout} makes sense in this case. However, the geometrically significant value of $r$ is $r = 0.$ 

The conjectured bound, if true,  would be sharp.
However, our methods appear to be insufficient to prove it, essentially because of the condition $k>n/2,$ see Theorem \ref{thm: Bezout general}.

Another result in a similar spirit  provides an estimate for  the coarse count of critical points of a linear combination of eigenfunctions. Note that the critical point of a smooth function $f$ on $M$ is a zero of its differential $df$ which is a section of the cotangent bundle $T^*M$ of $M.$

\begin{theorem}\label{thm: critical points}
	Let $E = T^*M$ with metric induced from $M$ and $s = df$ where $f \in \cl{F}_{\la}$ for the Laplace-Beltrami operator $\Delta$. Let $k > n/2$ be an integer. Then for any $\delta>0$,
	\[ z_0(s,\delta) \leq \frac{C_1}{\delta^{n/k}}  (\la+1)^{\frac{n}{2}}+1,\]
	where the constants $C_1$ depends only on  $M,D,k.$
\end{theorem}

This is an immediate consequence of Theorem \ref{thm: coarse Courant} applied to the Hodge-Laplacian $D$ on $E.$ Note that if $f \in \cl F_{\la}$ for $\Delta$ then $df \in \cl F_{\la}$ for $D$ since $Dd = d\Delta$ on smooth functions.

The upper bound asymptotically agrees with the estimate of Nicolaescu \cite{Nicolaescu} on average. Furthermore, the coarse count is necessary, since the example of Buhovsky-Logunov-Sodin \cite{BLS} has infinitely many critical points.

\subsection{Courant and B\'ezout: discussion}
\label{subsec:CCB}
The search for the analogue of Courant's theorem for linear combinations of Laplace eigenfunctions has a long history. A direct generalization of Courant's theorem to linear combinations of eigenfunctions is often referred to as the Courant--Herrmann conjecture \cite{GladZhu03} or the Extended Courant Property \cite{BerHel18}. For the one-dimensional Sturm--Liouville problem this result was proved by Sturm in 1836, see \cite{Arnold11} and \cite{BerHel20} for a fascinating historical discussion and another proof based on the ideas of Gelfand.  In higher dimensions, the Extended Courant Property does not hold in general
\cite{Arnold73,Arnold2004, Viro79} and various counterexamples have been found.  Moreover, as was shown in \cite{BLS}, there exist Riemannian metrics on a $2$-torus such that linear combinations of Laplace eigenfunctions have {\it infinitely many} nodal domains, and hence there is no hope for even a weaker analogue of Courant's theorem. Further examples  of this kind were constructed in \cite{BerHelChar20}.

Theorem \ref{thm: coarse Courant} follows a different  approach to  find an extention of  Courant's theorem.  It  was originally proposed in \cite{polterovich2007nodal} for Laplace eigenfunctions on surfaces,  and has been further developed using the language of persistent barcodes in \cite{PPS19}. The idea is to count only {\it ``deep''} nodal domains, i.e. nodal domains in which the absolute value of an eigenfunction reaches a certain threshold. In Theorem \ref{thm: coarse Courant}, this threshold  is given by $\delta>0$. Note that this  {\it coarse} nodal count  is physically meaningful, as very small oscillations are often difficult to detect, both experimentally and numerically.  Moreover, as was mentioned above,  the coarse nodal count extends not only to linear combinations of Laplace eigenfunctions, but also to  eigenfunctions of higher order operators. In particular, as was observed in \cite{polterovich2007nodal}, the coarse Courant theorem holds for eigenfunctions of a vibrating clamped plate. Note that in this case there is no usual Courant's theorem. On the contrary, for  planar domains with corners  having angles that are not too obtuse, it is expected that eigenfunctions have infinitely many nodal domains, see \cite[Section 2.5]{Davies97} and references therein. While the results of \cite{polterovich2007nodal, PPS19} were obtained only in dimension two, Theorem \ref{thm: coarse Courant} holds in arbitrary dimension.  In particular, it provides a positive answer to
Conjecture 1.4.7 posed in \cite{PPS19}.

Significantly less is known  about the analogues of  Courant's theorem for products of eigenfunctions. Some partial results in this direction have beeen obtained in \cite{Arnold05} and \cite{polt08}.   Interestingly enough, this subject is closely related to an analogue of B\'ezout's theorem for nodal sets discussed above. In fact, Theorems \ref{cor: product}  and \ref{cor: Bezout}
can be viewed as different facets of the same phenomenon. We illustrate this link in the following situation. Let  $Z_1,Z_2$  be the nodal sets of Laplace-Beltrami eigenfunctions $f_1$,$f_2$, respectively. The nodal set of the product $f_1f_2$ is the union $Z_1 \cup Z_2$, while B\'ezout's theorem deals with the intersection $Z_1 \cap Z_2$. By the Mayer-Vietoris exact sequence we have
$$H_{r+1}(Z_1) \oplus H_{r+1}(Z_2) \to H_{r+1}(Z_1 \cup Z_2) \to H_r(Z_1 \cap Z_2) \to H_r(Z_1) \oplus H_r(Z_2)\;. $$
Applying the rank-nullity theorem to the second and the third arrows, this readily yields
$$| \dim H_{r+1}(Z_1 \cup Z_2) - \dim H_r(Z_1 \cap Z_2)| \leq $$
$$ \dim H_{r+1}(Z_1)+  \dim H_{r+1}(Z_2)
+ \dim H_{r}(Z_1)+  \dim H_{r}(Z_2)\;.$$
While in general this inequality is not sharp, its coarse version developed below in Section \ref{sec: MV bezout prod} provides  a satisfactory link between 
the coarse Courant for products (Theorem \ref{cor: product}) and the coarse B\'ezout (Theorem \ref{cor: Bezout}) as the eigenvalues tend to infinity.  In particular, one can recover the asymptotics in 
the coarse Courant for products  using the coarse B\'ezout and  the coarse Courant for individual eigenfunctions  (Theorem \ref{thm: coarse Courant}), which is 
applied to estimate the coarse Betti numbers of $Z_1,Z_2$, see \eqref{eq-improved-vsp}. In this way the Mayer-Vietoris sequence brings together our main applications. 

\subsection{Optimality of the main results}
\label{subsec:sharp}
 The following simple example shows that  the powers of $\|s\|$  and $\delta$ in formulas \eqref{eq-cour-vsp} and \eqref{eq-bezout-vsp} are sharp.
\begin{ex}

Let $n=1$ and assume that the sections $s$ are  functions over an interval $[0,2\pi]$. Then there exists a constant $C$ such that for any $0<\delta<1$,
$$
m_0(\sin j x, \delta) \ge C j
$$
as $j \to \infty$, while $\|\sin jx \|_{W^{k,p}}^{1/k} = O(j)$ for any $k,p \ge 1$. Similar inequalities hold also for $z_0(\sin jx,\delta)$.

To show that the power of $\delta$ is sharp we first note that elementary rescaling yields $m_r(ts,t\delta)=m_r(s,\delta)$ and  $z_r(ts,t\delta)=z_r(s,\delta)$.
Hence, the right-hand side of the inequalities  \eqref{eq-cour-vsp} and \eqref{eq-bezout-vsp}  must depend only on the ratio between the norm of $s$ 
and~$\delta$.
\end{ex}
It is also instructive to consider 
\begin{ex}
Set 
 $$s_{\alpha,\beta}(x)=x^\alpha \sin \left(x^{-\beta}\right)$$
for some $\alpha,\beta >0$.  Note that if $\alpha=k (\beta+1)$, then $s_{\alpha,\beta} \in W^{k,p} ((0,2\pi))$ for any $k,p \ge 1$
Moreover, it is easy to check that there exists a constant $C>0$ such that
$$
m_0(s_{\alpha,\beta},\delta)\ge C \delta^{-\beta/\alpha}.
$$
as $\delta \to 0$. At the same time, \eqref{eq-cour-vsp} yields $m_0(s_{\alpha,\beta}, \delta) = O\left(\delta^{-\frac{1}{k}}\right)$, and this bound is saturated in the limit as $\beta \to \infty$. Similar estimates hold also for $z_0(s_{\alpha,\beta},\delta)$.
\end{ex}

In fact, a considerably more general sharpness result holds. It shows that the upper bound of Theorem \ref{thm: coarse Courant} is essentially sharp, at least as far as the power of $(\la+1)$ is concerned. 

\begin{theorem}\label{thm-sharp mini}
Let $(M,g)$ be a closed Riemannian manifold and $D=\Delta$ the Laplace-Beltrami operator on functions. There exists $c = c(M,g) > 0$ such that for every $ \delta > 0 $ one can find $ f \in \cF_\lambda $, $ \| f \|_{L^2} = 1 $, for which we have
	\begin{equation} \label{eq:sthm-statement}
		m_0(f,\delta)   \geqslant c \frac{(\lambda+1)^{n/2}}{\max(1,\delta^2)} - 1.
	\end{equation}
	The same lower bound  holds also for $z_0(f,\delta)$.
\end{theorem}
The proof of Theorem \ref{thm-sharp mini} is presented in Section \ref{sec: sharpness}. Note that Theorem \ref{thm-sharp mini} is consistent with the asymptotically sharp $L^{\infty}$ bound \[ ||f||_{L^{\infty}} \leq C (\la+1)^{n/4}\] on $ f \in \cF_\lambda $, $ \| f \|_{L^2} = 1$, which is a consequence of the local Weyl law \cite{Hormander-spectral}, see also \cite[Proposition IV.1]{Chavel-eigen}. Indeed, in view of this bound, if $\delta > C(\la+1)^{n/4},$ then $m_0(f,\delta) = 0$. At the same time,  inspecting  the proof of Theorem \ref{thm-sharp mini}, one can check that in this case $c/C^2 \le  1$, and hence  the right-hand side in \eqref{eq:sthm-statement} is non-positive. We refer also to Remark \ref{rmk: sharpness} for further discussion on Theorem \ref{thm-sharp mini} in relation to  sharpness of our main results.

The coarse Courant theorem gives rise to a natural  question on whether its  {\it non-coarse} analogue holds.
In particular, does  a bound of the form 
\begin{equation}\label{eq: Courant const} 
m_r(f)=O\left(F(\la)\right), 
\end{equation} 
 where  $F$ is some positive function,  hold on an arbitrary compact  Riemannian manifold $M$, provided
\begin{itemize}
\item  $r=0$  and  $f = \sum_{j=1}^i a_j f_j$, where  $f_j$ are Laplace eigenfunctions on $M$ with eigenvalues $\la_j \leq \la=\la_i$;  

\smallskip

\item $r\ge 0$  is  arbitrary and $f$ is a Laplace eigenfunction on $M$ with eigenvalue $\lambda$;

\smallskip

\item $r=0$  and  $f$ is an 
eigenfunction of an arbitrary elliptic operator $D$ on $M$ with eigenvalue $\lambda$.
\end{itemize}
Using results of \cite{BLS} we show  that in general the answer to all these questions is ``no". In what follows $T^n$ denotes an $n$-dimensional torus.

\begin{prop}\label{prop: no}
The following assertions hold:

\smallskip

\noindent (i) There exists a Riemannian metric $g_{BLS}$ on a  $T^2$ admitting a  sequence  $f_{i_j}$ of Laplace eigenfunctions corresponding to eigenvalues 
 $\la_{i_j} \to \infty$  as $j\to \infty$, such that  \[m_0(f_{i_j}-c_{i_j}) = +\infty\] for some constants $c_{i_j}$ for all $j \ge 1$. 

\smallskip

\noindent (ii) 
For $T^4$ endowed with $g_{BLS} \oplus g_{BLS},$ the eigenfunctions 
$u_{i_j} = f_{i_j} \oplus - f_{i_j}$ satisfy \[m_1(u_{i_j}) = +\infty\] for all $j \ge 1.$ 

\smallskip

\noindent (iii)  Let  $g_{BLS} \oplus g_{st}$ be the Riemannian metric on a $T^3$, where $g_{st}$ is the standard  metric on a unit circle.
Then the  eigenfunctions $h_j = f(x) \sin(jy)$ of the non-negative fourth order elliptic operator $D = \Delta^2 - \la \Delta_x+\la^2/4$ with eigenvalues $\la_j = j^4+\la^2/4$, where $f=f_{i_1}-c_{i_1}, \la=\la_{i_1}$ satisfy \[m_0(h_j) = +\infty\] for all $j \ge 1.$ 
\end{prop}

Proposition \ref{prop: no} confirms  the intuition that the Courant-type bound \eqref{eq-cour} is  rather special for the nodal domain count of Laplace eigenfunctions.  For $r=0$ it also holds for  some closely related operators, like the Schr\"{o}dinger  operator, or certain linear combinations of its powers. 
However,  in the pseudo-differential setting,  the nodal domain count can be infinite even for operators of order two. 
Indeed, let $A=\sqrt{D+I},$  where $D$ is the operator defined in (iii) and $I$ is the identity operator.  
By \cite{Seeley}, $A$ is a pseudo-differential operator of order two of the form $A=\Delta+P,$ where $P$ is of order at most one,  and $h_j$ are  eigenfunctions of $A$. 

As follows from (ii),  even in the case of Laplace eigenfunctions, estimate  \eqref{eq: Courant const}  can  not hold in general for  higher Betti numbers.  
It would be interesting to understand
whether  \eqref{eq: Courant const}  for $r>0$ holds for real-analytic  Riemannian metrics (note that the metric $g_{BLS}$ that was constructed in \cite{BLS}  is smooth but {\it not} real-analytic). Some related results in this direction have been obtained in \cite{LinLiu}. Using Milnor's theorem on the zero sets of real polynomials \cite{Milnor}, one can show that an analogue of \eqref{eq-cour} for  higher Betti numbers holds 
for the nodal sets of eigenfunctions on spheres and flat tori \cite{Non19}.

Finally, let us note that while the counterexamples  in Proposition \ref{prop: no} are presented for the Betti numbers $m_r$ of the complement to the nodal set,
it should not be hard to obtain similar results for the Betti numbers $z_r$ of the zero set.

\subsection{Bounds on persistence barcodes}\label{subsec:persbounds}
Recall that for a Morse function $f:M \to \R$ on a compact manifold and a coefficient field $\bb
K$, its {\em barcode} is a finite multiset $\cl B(f;\bb K)$ of intervals
with multiplicities $(I_j,m_j),$ where $m_j \in \N$ and $I_j $ is finite, that is of the form
$[a_j,b_j)$ or infinite, that is of the form $[c_j,\infty).$ The number of infinite bars is
equal to the total Betti number $b(M;\bb K) = \dim H(M;\bb K).$

This barcode is obtained algebraically from the {\em persistence module} $V(f)$
consisting of vector spaces $V(f)_t = H(\{f \leq t\}; \bb K)$ parametrized by $t\in
\R$ and structure maps $\pi_{s,t}: V(f)_s \to V(f)_t$ induced by the inclusions
$\{f \leq s \} \hookrightarrow \{f \leq t \}$ for $s\leq t.$ These maps satisfy the structure
relations of a persistence module: $\pi_{s,s} =\id_{V(f)_s}$ for all $s$ and
$\pi_{s_2,s_3} \circ \pi_{s_1,s_2} = \pi_{s_1,s_3}$ for all $s_1 \leq s_2 \leq s_3.$
We refer to \cite{PPS19} for first applications of
persistent homology to spectral theory, and to Section \ref{Section_Prelims} below for further preliminaries
and references. 

Recall that the length of a finite bar $[a,b)$ is $b-a$ and the length of an infinite bar
$[c,\infty)$ is $+\infty.$ We require the following number: $\cl N_{\delta}(f)$ is the
number of bars of length $>\delta$ in the barcode $\cl B(f).$ As we shall see below,
this quantity is well defined for continuous (not necessarily smooth) functions.  
With these preparations in mind we state our main technical result.

\begin{theorem}\label{thm: main 2}
	Let $E$ be a vector bundle over $M$ with an inner product. Suppose that $s \in W^{k,p}(M;E)$
	and $k-n/p > 0.$ Then $|s|$ being continuous, $\cl B(|s|)$ is well-defined and for all
	$\delta > 0,$ \[ \cl N_{\delta}(|s|) \leq \frac{C_1}{\delta^{n/k}}
	||s||_{W^{k,p}}^{n/k}+C_2,\]
	where the constant $C_1$ depends only on $M,E,k,p$ and $C_2 = \dim H_*(M).$ 
\end{theorem}
\begin{remark}\label{rmk: negative s}
The same result holds with $|s|$ replaced by $-|s|$ on the left hand side (see Remark \ref{rmk: negative s detail}). This is particularly relevant in the case of manifolds with boundary (see Remark \ref{rmk: no 113}).
\end{remark}
\begin{remark}\label{rmk: 1.12 for r}
A similar result holds for $\cl N_{r,\delta}(|s|),$ where we consider the barcode in degree $r$ only. In this case $C_2 = \dim H_r(M).$ A similar bound with $C_2 = 0$ holds for the count $\cl N^{\mrm{fin}}_{\delta}(|s|)$ of only the finite bars of length $>\delta.$
\end{remark}

This result yields  Conjecture 1.4.7  and a particular case of Conjecture 1.4.8 from  \cite{PPS19} (for $n=2,$  both conjectures were
proved in \cite{PPS19}.)  Originally these conjectures have been  formulated for the Laplacian, but we prove them below in greater generality.
Let $E$ be a vector bundle with inner product on a closed Riemannian manifold $M$ of dimension
$n$ and let $D$ be a non-negative elliptic self-adjoint differential operator of order $q$ on the sections
of $E.$  Recall that  $\mathcal{F}_\la$  denotes the subspace spanned by all eigensections with eigenvalues $\leq \la$.

\begin{theorem}\label{thm: barcode conj delta}
	Let $s \in \cl{F}_{\la}$ with $||s||_{L^2} = 1.$ 
	Then for all $\delta > 0$ and integer $k>n/2,$ \[\cl{N}_{\delta}(|s|) \leq \frac{C_1}{\delta^{n/k}}
	(\la+1)^{n/q}+C_2\] where $C_1$ depends only on $M,E,D,k$ and $C_2 = \dim H_*(M).$ 
\end{theorem}
Note that this result is essentially sharp, as follows from Theorem \ref{thm-sharp mini}.
Theorem \ref{thm: barcode conj delta} has applications to approximation theory, which we will not discuss here, referring the reader to \cite{PPS19,PRSZ} for a detailed discussion in the case of surfaces. We present another application to  Conjecture 1.4.8 from \cite{PPS19}.

Recall that for a barcode $\cl{B}(f)$ of a function $f$ on a closed manifold $M,$ $|\cl{B}(f)|$
denotes the sum of the lengths of the finite bars of $\cl{B}(f)$ plus the sum of the differences
$\max(f) - c_j$ for $1 \leq j \leq \dim H_*(M),$ where $c_j$ are the starting points of the
infinite bars in $\cl{B}(f).$ Note that $\max(f)$ is itself the maximal such starting point.

\begin{theorem}\label{thm: barcode conj norm}
	Suppose $n = \dim M \geq 3.$ Let $s \in \cl{F}_{\la}$ with $||s||_{L^2} = 1.$ Then \[
	|\cl{B}(|s|)| \leq C(\la+1)^{n/q}\] where $C$ depends only on $M,E,D.$ 
\end{theorem}

The condition $n \geq 3$ is technical and comes from being able to choose an {\em integer} $k$ with $n>k>n/2.$ 

\begin{remark}
\label{rem:fractional}
It should  not be hard to extend Theorems \ref{thm: main 2} and \ref{thm: barcode conj delta}  in the spirit of \cite[Proposition 6.1]{DS80}  to fractional Sobolev spaces (cf. \cite{triebel}) with arbitrary real parameter $k>n/p.$  Such an extension would remove the technical condition $n \geq 3$ in Theorem \ref{thm: barcode conj norm}, see also Remark \ref{rmk: sharpness}. 
\end{remark}

\begin{remark}\label{rmk: Lp barcode}
We can prove an analogue of Theorem \ref{thm: barcode conj norm} for the $L^p$ norm of the barcodes by essentially the same argument. The $L^p$ norm of the barcode of $|s|$ is defined for $p\geq 1$ as the expression \[ |\cl{B}(|s|)|_p = \left(\sum \beta_i(|s|)^p + \sum (\max(f)-c_j)^p  \right)^{1/p},\] where $\beta_i(|s|)$ are the lengths of the finite bars in the barcode, arranged in decreasing order (see \cite{CSEHM} for a similar definition). We can prove that for all $p \geq 1,$ $s \in \cl{F}_{\la},$ $||s||_{L^2} = 1,$ \[ |\cl{B}(|s|)|_p \leq C(\la+1)^{n/q},\] where $C$ depends on $M,E,D, p.$ Moreover, for $p \in [1,2)$ we can improve the power of $\la+1$ to $(\la+1)^{n/pq}$ and for $p \geq 2,$ we can improve it to $(\la+1)^{k_1/q},$ for every $n/2 < k_1 < n.$ 
We refer to Remark \ref{rmk: Lp barcode proof} for a few details of this generalization.
\end{remark}

\begin{remark}
	Let $M$ be a closed $n$-dimensional Riemannian manifold, and let $T^*M$ be  
		its cotangent bundle equipped with the associated (Sasaki) metric. Given a smooth function on $f$, consider the graph
		of its differential,  $\text{graph}(df) \subset T^*M$.  Note that it is Lagrangian with respect to the canonical symplectic
		form on $T^*M$.  A recent paper \cite{CGG},
		which relates the Floer-homological bar counting function of Lagrangian submanifolds
		with the topological entropy of symplectic maps, yields an interesting result
		in our context. Namely,  the arguments in \cite[Section 5]{CGG} imply that for all $\delta >0$
		\begin{equation}\label{eq-CGG-vsp}
			\cl N_\delta(f) \leq C(\delta)\; \text{Volume}_n(\text{graph}(df))  \;,
		\end{equation}
		where $C(\delta)$ is a positive constant depending on $\delta$ and the metric.
		For instance, if $M$ is the standard Euclidean torus, this reads
		$$\cl N_\delta(f) \leq C(\delta) \int_{M} \sqrt{\det \left(I + (\text{Hess} \,\, f)^2\right)} d\text{Vol}\;,$$
where $\text{Hess}\, \, f$ denotes the Hessian of $f$ and $I$ is the identity matrix.
		Inequality \eqref{eq-CGG-vsp} is neither stronger, nor weaker than the one provided by our main theorem.
At the same time,  in terms of Sobolev norms, it yields
		$$\cl N_\delta(f) \leq C(\delta) ||f||^n_{W^{2,n}}+C'\;,$$
		while we get a stronger estimate
		$$\cl N_\delta(f) \leq C_1(\delta) ||f||^{n/2}_{W^{2,n}}+C_1'\;.$$
		It should be mentioned also that for $n=2$, i.e., when $M$ is a surface,
		the approach of \cite{polterovich2007nodal} involved the length of the normal lifts of the level sets of $f$. It would be interesting to compare a direct extension of this approach to higher dimensions with inequality \eqref{eq-CGG-vsp}.
\end{remark}

\subsection{Ideas of the proof}
Let us outline the proof of Theorem \ref{thm: main 2} for functions on a cube (see also Theorem \ref{CUBE}).
The general case is based on the same ideas. In this informal sketch we write $\lesssim$ for {\it less or equal} up to a multiplicative constant depending only on $k,n,p$, but not on the function $f$ and the real number $\delta$. The proof is based on two important facts from the theory of persistence modules.

\medskip
\noindent {\bf Fact 1.} By a fundamental {\it stability theorem} (see Theorem \ref{Stability theorem}), $\cl N_\delta(f)$ does not decrease if we perturb $f$ in the uniform norm and simultaneously slightly decrease $\delta$. Thus,  if $f$ is well approximated on an $n$-dimensional cube $Q$ (or more generally, on an $n$-dimensional box
$B= \prod_{i=1}^n [a_i,b_i]$) by a polynomial of degree $k$, the quantity $\cl N_\delta(f|_Q)$ is bounded from above by the number of critical points of this polynomial. By Milnor's celebrated bound {and Morse theory for manifolds with corners}, this yields $\cl N_{\delta}(f|_Q)=O( k^n)$.

\medskip
\noindent {\bf Fact 2.} We  repeatedly use that if $U \to V \to W$ is
an exact sequence of persistence modules, then
$$\cl N_{2\delta}(V) \leq \cl N_{\delta}(U) + \cl N_\delta(W), \;\; \forall \delta >0\;.$$

This fact appears to be new, and its proof is based on algebraic ideas, see Section \ref{sec: commutes}.

The argument goes as follows. Put $\alpha:= k/n - 1/p >0$. Fix a function $f\in W^{k,p}(Q)$ on a  unit cube $Q=[0,1]^n$ and divide it into $2^n$
equal cubes. A cube $Q_i$ of the partition is called {\it good} if
\begin{equation}\label{eq-cube-vsp}
\text{Vol}(Q_i)^\alpha \cdot ||D^k(f|_{Q_i})||_{L^p} \lesssim \delta\;,
\end{equation}
and {\it bad} otherwise. We subdivide each bad cube again, and continue the process using criterion
\eqref{eq-cube-vsp} until all the cubes are good; note that this will be achieved after a finite number of steps.
We get  {\it a multiscale dyadic partition}
$K$ of $Q$ consisting of $\kappa$ good cubes. The crux of the matter is that on each good cube
$f$ is well approximated by a polynomial of degree $k$. This readily follows from the Morrey-Sobolev inequality (see Theorem \ref{Morrey-Sobolev}) which we review in the Appendix. Hence, by Fact~1,
\begin{equation}\label{eq-cube0-vsp}
\cl N_\delta(f|_{Q_i}) =O( k^n)
\end{equation}
for every good cube $Q_i$.

The next task is to assemble estimates \eqref{eq-cube0-vsp} for individual cubes of the partition into a global estimate. Our argument echoes\footnote{We thank G. Binyamini and D. Novikov for pointing this out to us.} the one in \cite{GV}. First, we  use Lemma \ref{Bad_cubes} to prove  that
\begin{equation}\label{eq-cube1-vsp}
\kappa \lesssim  \left(\frac{||D^k f||_{L^p}}{\delta}\right)^{n/k} + 1.
\end{equation}
Second, using a  combinatorial argument (Lemma \ref{Minimal_Covering_Lemma}) we show that $Q$ can be represented as a union of
$n+1$ sets  $K_j$,  $j = 0, \dots, n,$ satisfying the following properties:
\begin{itemize}
\item[(i)]  For each $j$, the set  $K_j$ is a pairwise disjoint union of rectangular boxes $B_{ij}$;

\smallskip

\item[(ii)]  Each box $B_{ij} \subset K_j$, $j=0,\dots, n$,  is contained in a small neighborhood of a $j$-dimensional face of some cube belonging to the multiscale dyadic partition $K$ (in this notation a $0$-face is a vertex of a cube and $n$-face is a cube itself).

\end{itemize}
We refer to Figure \ref{Coloring} for an illustation of this construction.

Using additivity of the bar counting function over disjoint sets \eqref{eq:bardisjoint},  we obtain
\begin{equation}
\label{eq:partadd}
\cl N_\delta(f|_{K_j}) = \sum_{i=1}^{\beta_j}  \cl N_\delta(f|_{B_{ij}}) \lesssim \beta_j \cdot  k^n.
\end{equation}
Here $\beta_j$ denotes the number of connected components of $K_j$,  and we use  a version of  \eqref{eq-cube0-vsp} and property (ii) combined with Fact 1 to obtain the inequality on the right-hand side. Property (ii) implies that $\beta_j$ are bounded above by $C(n)\kappa$, where $C(n)$ is a constant  depending only on $n$. 
Furthermore, (i) and (ii) yield that
the number of tuples $\{(i_1j_1, \dots, i_pj_p)\}$ with
$B_{i_1j_1} \cap \cdots \cap B_{i_pj_p} \neq \emptyset$ is bounded from above by
$C(n)\kappa$ as well. With this in mind, apply the Mayer--Vietoris sequence
together  with Fact~2 to the cover of $Q$ by the sets $K_j$. It follows that
$$\cl N_{2^{n+1}\delta}(f) \lesssim \sum_{j=0}^n \cl N_{\delta}(f|_{K_j}) + \sum \cl N_{\delta}(f|_{B_{i_1j_1} \cap \cdots \cap B_{i_pj_p}}) \lesssim C(n)\cdot \kappa   \cdot k^n.$$
Absorbing $C(n)$ and $k^n$ into the constants and using  \eqref{eq-cube1-vsp} we get
$$\cl N_{2^{n+1}\delta}(f) \lesssim    \left(\frac{||D^k f||_{L^p}}{\delta}\right)^{n/k} + 1 ,$$
and after a  rescaling  in $\delta$  this concludes the proof of Theorem \ref{thm: main 2} for functions on a cube.
\subsection*{Plan of the paper}  The paper is organized as follows. In Section \ref{Section_Prelims} we state the main preliminary facts about persistence modules and barcodes that are used  in the paper.  In Section \ref{sec: commutes} we prove Theorem \ref{Exact_Pers}  providing  subadditivity of the bar counting function for persistence modules in a short exact sequence.  This is a key technical result that appears to be novel in the theory of persistence modules. {In Section \ref{Section_Semialgebraic} we discuss multiscale polynomial approximation of a function on a dyadic partition of the cube and estimate the bar counting function in terms of the number of sets in the partition. In Section \ref{Proof} we prove Theorem \ref{thm: main 2} in the case of the cube by constructing such an approximation with the number of sets controlled by a suitable Sobolev norm. Then we extend the argument to the general case by triangulation.}
The proofs of the  coarse Courant  and B\'ezout theorems are  presented in Section \ref{sec:applications}. In Section \ref{sec: sharpness} we prove 
Theorem \ref{thm-sharp mini} showing that our main results are essentially sharp.
In Section \ref{sec: MV bezout prod} we show that the  coarse nodal estimate for the product of two functions can be deduced from the coarse B\'ezout 
using the Mayer--Vietoris 
sequence.
Finally, in Appendix \ref{app: Morrey-Sobolev} the proof of a more precise version of the Morrey--Sobolev theorem (Theorem \ref{Morrey-Sobolev}) is provided for the convenience of the reader.


\section{Preliminaries on persistence modules and barcodes}\label{Section_Prelims}

\subsection{Persistence modules and barcodes}
We review the basics of the persistence theory which we use. For a detailed account see \cite{Oudot_2015,Chazal_et_al_2016,PRSZ}.
\begin{dfn} A persistence module $(V,\pi)$ over a field $\mathbb{K}$ consists of a family of vector spaces $V_t,t\in \R$ over $\mathbb{K}$ together with linear maps $\pi_{s,t}:V_s\rightarrow V_t$ defined for all $s\leq t$, called structure maps, which satisfy $\pi_{t,t}=\id_{V_t}$ for all $t\in \R$ as well as $\pi_{s,t}\circ \pi_{r,s}=\pi_{r,t}$ for all $r\leq s \leq t.$
\end{dfn}

We often abbreviate $(V,\pi)$ to $V.$ The example of most interest for us is the following. Let $f:X\rightarrow \R$ be a function on a Hausdorff topological space. Define $V_k(f)_t=H_k(\{f\leq t \})$ and $\pi_{s,t}=(i_{s,t})_*$, where $i_{s,t}:\{ f \leq s\} \rightarrow \{f\leq t\}$ are inclusions and $H_k$ denotes singular homology in degree $k$ with coefficients in a field $\mathbb{K}$.

\begin{dfn}A morphism of persistence modules $\phi:(V,\pi^V)\rightarrow (W,\pi^W)$ is a family of linear maps $\phi_t :V_t \rightarrow W_t,t\in \R$ such that for all $s\leq t$ it holds $\pi^W_{s,t}\circ \phi_s=\phi_t\circ \pi^V_{s,t}.$
\end{dfn}
Given a morphism of persistence modules $\phi$, we may define $\ker \phi$ and $\im \phi$ as persistence modules by taking kernels and images for each $t\in \R$. More precisely, $(\ker \phi)_t=\ker(\phi_t),\pi^{\ker \phi}_{s,t}=\pi^V_{s,t}|_{\ker \phi_s}$ and similarly for $\im \phi.$ We define persistence submodules, quotients and direct sums in a similar way, pointwise for each $t\in \R.$ In the above example of a function $f:X\rightarrow \R$, we denote $V(f)=\oplus_{k} V_k(f).$

In order to have a rich theory, additional conditions are often placed on persistence modules. To this end, a persistence module $V$ is called {\it pointwise finite-dimensional} if for all $t\in \R$, $\dim V_t<\infty.$ Going back to our main example, if we take $X$ to be a smooth, compact manifold and $f:X\rightarrow \R$ a smooth Morse function, basic results of Morse theory tell us that $V(f)$ is pointwise finite-dimensional. Pointwise finite-dimensional modules have simple structure, as we will now explain.  By an interval $I\subset \R$ we mean any connected subset.
\begin{dfn}
For an interval $I\subset \R$, define the interval persistence module $\mathbb{K}_I$ as
$$(\mathbb{K}_I)_t =\begin{cases}
      \mathbb{K}, & \text{if}\ t\in I \\
      0, & \text{otherwise} \\
      \end{cases} ~~~~~,~~~~~
   \pi^{\mathbb{K}_I}_{s,t}=\begin{cases}
      \id_\mathbb{K}, & \text{if}\ s,t\in I \\
      0, & \text{otherwise} \\
      \end{cases}.$$
\end{dfn}
\begin{dfn}
A barcode $\B$ is a multiset of intervals with finite multiplicities. 
\end{dfn}
\begin{theorem}[Structure theorem]
To every pointwise finite-dimensional persistence module $(V,\pi)$ corresponds a unique barcode $\B(V)$ such that
$$(\pi,V)\cong \oplus_{I\in \B(V)}(\mathbb{K}_I,\pi^{\mathbb{K}_I}).$$
\end{theorem}
Structure theorem in stated generality was proven in \cite{WCB}. In the modern theory of persistence, structure theorem first appeared in \cite{ELZ02,ZC05}. A version of the theorem was also proven in \cite{Barannikov94} using different language. However, as noticed in \cite{BMMS21}, the notion of a barcode can be traced back to the works of Morse. Namely, in \cite{Morse_RS} Morse defines notions of a {\it cap} and a {\it cap height} which is equivalent to the endpoint of a bar as well as a notion of a {\it cap span} which is equivalent to the length of the corresponding bar.

Given a persistence module $(V,\pi),$ it will be convenient to call a point $t\in \R$ {\em spectral for $V$} if $t$ is an endpoint of a bar in $\cl{B}(V).$ The spectrum $\mrm{Spec}(V)$ of $V$ is the set of the points $t \in \R,$ which are spectral for $V.$ 

One of the most important features of barcodes is the fact that they behave in a stable manner with respect to perturbations of persistence modules. This stability is a part of the metric theory which we now present.

We use $\langle a,b \rangle$ to denote any of the intervals $(a,b),(a,b],[a,b),[a,b].$ Two barcodes $\B_1$ and $\B_2$ are {\it $\varepsilon$-matched}, $\varepsilon>0$, if after erasing certain bars of length $<2\varepsilon$ from each of them, there exists a bijection $\Phi$ between remaining bars, which satisfies
$$\Phi(\langle a,b \rangle)=\langle c,d \rangle \Rightarrow |a-c|,|b-d|<\varepsilon.$$
Intuitively, an erased bar is matched with an empty bar at its center. Thus, $\varepsilon$-matching can be thought of as a matching up to an error $\varepsilon$ at the endpoints. {\it The bottleneck distance} between barcodes is defined as
$$d_{bottle}(\B_1,\B_2)=\inf \{ \varepsilon ~|~ \B_1,\B_2\textrm{ are }\varepsilon\textrm{-matched} \}.$$
It is not difficult to check that $d_{bottle}$ is a pseudometric. The persistence counterpart of this distance is defined as follows. For $\varepsilon>0$ and a persistence module $V$, denote by $V[\varepsilon]$ the persistence module given by $V[\varepsilon]_t=V_{t+\varepsilon}$, $\pi^{V[\varepsilon]}_{s,t}=\pi^V_{s+\varepsilon,t+\varepsilon}$. A pair of morphisms $\phi: V\rightarrow W[\varepsilon], \psi:W\rightarrow V[\varepsilon]$ is called an {\it $\varepsilon$-interleaving} if for all $t\in \R$, $\psi_{t+\varepsilon}\circ \phi_t=\pi^V_{t,t+2\varepsilon}, \phi_{t+\varepsilon}\circ \psi_t=\pi^W_{t,t+2\varepsilon}.$ If such a pair of morphisms exists $V$ and $W$ are said to be {\it $\varepsilon$-interleaved}. The {\it interleaving distance} between two persistence modules is defined as
$$d_{inter}(V,W)=\inf \{ \varepsilon ~|~ V,W\textrm{ are }\varepsilon\textrm{-interleaved} \}.$$
Again, it is not difficult to check that $d_{inter}$ is a pseudometric. The following result is one of the cornerstones of the theory of persistence modules and barcodes.
\begin{theorem}[Isometry theorem]\label{Isometry theorem}
For two pointwise finite-dimensional persistence modules $V$ and $W$ it holds
$$d_{inter}(V,W)=d_{bottle}(\B(V),\B(W)).$$
\end{theorem}
The isometry theorem is due to \cite{CSEH07,CCSGGO09,Lesnick15}, see \cite{BL15} for a detailed history. In the case of a persistence module coming from a function, we abbreviate $\B(V_k(f))$ to $\B_k(f)$ and $\B(V(f))$ to $\B(f).$ As an immediate corollary of the isometry theorem, we obtain the following statement \cite{CSEH07}. 
\begin{theorem}[Stability theorem]\label{Stability theorem} Assume that $f,g:X\rightarrow \R$ are such that $V_k(f),V_k(g)$ are pointwise finite-dimensional. Then
$$d_{bottle}(\B(V_k(f)),\B(V_k(g)))\leq d_{C^0}(f,g).$$
\end{theorem}
\begin{proof}
Inclusions $\{f\leq t \}\subset \{g\leq t+d_{C^0}(f,g)\}\subset \{f\leq t+2d_{C^0}(f,g) \}$ induce a $d_{C^0}(f,g)$-interleaving between $V_k(f)$ and $V_k(g)$, which together with Theorem \ref{Isometry theorem} finishes the proof.
\end{proof}
\begin{remark}
For convenience, we will sometimes use \v{C}ech homology instead of singular homology, see Proposition \ref{Sing_Cech} and the discussion preceeding it. Stability theorem continues to hold with the same proof.
\end{remark}
\subsection{Bar counting function}\label{SubSec_Bar_Count}
We say that a persistence module is a {\it finite barcode module} if it is pointwise finite-dimensional and its barcode is finite. Let $\delta>0$ and $V$ a finite barcode module. We define $\mathcal{N}_\delta(V)$ to be the number of bars, counting multiplicities, of length $>\delta$ in $\B(V).$ We also use $\mathcal{N}_\delta(\B)$ for an arbitrary barcode as well as $\mathcal{N}_{k,\delta}(f)=\mathcal{N}_\delta (\B_k(f))$ and $\mathcal{N}_{\delta}(f)=\mathcal{N}_\delta (\B(f)).$

Our results concern $\mathcal{N}_\delta$ of persistence modules which are not necessarily finite barcode modules. This is justified by the fact that we only consider continuous objects such as functions or sections, defined on fairly regular spaces, such as compact manifolds with corners. Indeed, for such a space $X$, the set of continuous functions $f$ such that $\B(f)$ is finite is dense in $(C^0(X),d_{C^0}). $ Hence, due to stability theorem, the 1-Lipschitz function $f\rightarrow \B(f)$ extends to $C^0(X)$, taking values in the completion of the space of finite barcodes with respect to $d_{bottle}.$ This completion consists exactly of all barcodes $\B$ such that for all $\delta>0,\mathcal{N}_\delta(\B)$ is finite, see \cite[Theorem 5.21]{Chazal_et_al_2016} and \cite[Proposition 22]{LeRSV21}. 

Alternatively, we may argue that on our spaces of interest, for each $f\in C^0(X), V(f)$ is a $q$-tame persistence module.
\begin{dfn}
A persistence module is called $q$-tame if for all $s<t,\pi_{s,t}$ has finite rank. 
\end{dfn}
The structure and isometry theorems carry over to this generality with minor modifications, see \cite{Obs16} and references therein.  If the set of functions whose associated persistence module is pointwise finite-dimensional is dense in {$(C^0(X),d_{C^0})$}, then $V(f)$ is $q$-tame for all $f\in C^0(X).$ This is for instance the case when $X$ is a compact manifold with corners. Indeed, for fixed $s<t$, we may find a $C^0$-small perturbation $g$ of $f$ such that $V(g)$ is pointwise finite-dimensional and for some $\varepsilon>0$, $\{f\leq s \}\subset \{g\leq s+\varepsilon \} \subset \{f \leq t \}.$ This implies that $\pi^{V(f)}_{s,t}$ factors through $V(g)_{s+\varepsilon}$ which is finite-dimensional and hence $V(f)$ is $q$-tame.  Moreover,  if $f$ is a continuous function on a compact Hausdorff space such that $V(f)$ is $q$-tame, then $\mathcal{N}_\delta(f)$ is finite, as explained in \cite{BMMS21}.

Let us mention that the finiteness of $\mathcal{N}_\delta$ has been studied already by Morse, see \cite[Theorem 7.5,Corollary 10.2]{Morse_RS}. Moreover, in the same work, Morse observed the relevance of the condition of $q$-tameness,  see \cite[Theorem 6.3]{Morse_RS}. We refer the reader to \cite{BMMS21} for further connections of Morse's works to the modern theory of persistence.

\begin{remark}\label{Remark_Observable}
There is a slight ambiguity in the two extensions of $\mathcal{N}_\delta(f)$ to continuous functions we just presented. Namely, $d_{bottle}$ is only a pseudometric, so in order to define the completion, we need to consider the quotient space of barcodes, with respect to a relation $\B_1\sim \B_2$ if and only if $d_{bottle}(\B_1,\B_2)=0.$ This amounts to ignoring bars of length zero as well as identifying bars with different conventions on endpoints (open, closed and half-open). Manifestly, for $\delta>0$, $\mathcal{N}_\delta$ is well-defined on this quotient and on the resulting completion. On the side of persistence modules, one should regard $q$-tame modules as objects in the {\it observable category}. Informally, this category ignores all the features which do not persist over non-zero time, see \cite{Obs16} for details. Again, for $\delta>0$, $\mathcal{N}_\delta$ is well-defined in the observable category. 
\end{remark}

\begin{remark}\label{Less_VS_Leq}
Defining $V_*(f)_t$ to be $H_*(\{ f\leq t \})$ instead of $H_*(\{f<t\})$ is a matter of convention which does not affect $\mathcal{N}_\delta (f).$ Namely, if we set $\mathring{V}_*(f)_t=H_*(\{ f<t \})$ it immediately follows that $d_{inter}(V(f),\mathring{V}(f))=0$ since for each $\varepsilon>0$, $\{f<t \} \subset \{f\leq t+\varepsilon \} \subset \{ f<t+2\varepsilon \}. $ By the isometry theorem $d_{bottle}(\B(\mathring{V}(f)),\B(V(f)))=0$ and hence $\mathcal{N}_\delta(\mathring{V}(f))=\mathcal{N}_\delta(V(f))$ for all $\delta>0.$
\end{remark}

It will be useful for us to work with homology theories other than singular homology. Namely, in Sections \ref{Section_Semialgebraic} and \ref{Proof} we use Mayer-Vietoris sequence for compact sets which exists in \v{C}ech homology (see \cite[Chapters IX, X and Theorem I.15.3]{EilenbergSteenrod} and \cite[Appendix A]{Landi}). Recall that \v{C}ech homology is the inverse limit of the homology of nerves of open covers, where the covers are partially ordered via refinement. This change of convention is justified as follows. Let $\check{V}_*(f)_t=\check{H}_*(\{f\leq t\})$ where $\check{H}_*$ denotes \v{C}ech homology with coefficients in $\mathbb{K}.$ From the discussion above it follows that in all cases we consider, for a continuous function $f$, $\check{V}_*(f)_t$ is $q$-tame and in fact $\mathcal{N}_\delta (\check{V}_*(f))$ is finite. Moreover, the following holds.
\begin{prop}\label{Sing_Cech}
Let $M$ be a compact manifold, possibly with boundary, and $f:M\rightarrow \R$ a continuous function. For all $\delta>0,k\in \Z$ it holds $\mathcal{N}_\delta(\check{V}_k(f))=\mathcal{N}_\delta(V_k(f)).$
\end{prop}
\begin{proof}
It is enough to prove the proposition for a smooth function. Indeed, due to stability theorem, for $f\in C^0(M)$ and $\delta>0$, $\mathcal{N}_\delta(\check{V}_k(f))=\lim_{n\to \infty} \mathcal{N}_\delta (\check{V}_k(f_n))$,  $\mathcal{N}_\delta(V_k(f))=\lim_{n\to \infty} \mathcal{N}_\delta (V_k(f_n))$ for a sequence of smooth functions $f_n\xrightarrow{C^0} f.$ Thus, let us assume that $f$ is smooth.

We will show that $d_{inter}(\check{V}_k(f),V_k(f))=0.$ Let $\varepsilon>0,t\in \R$ and $t'\in (t,t+\varepsilon)$ a regular value of $f.$ Then $\{ f\leq t' \}$ is a CW-complex and hence there is an isomorphism $I_{t'}:\check{H}_k(\{ f\leq t' \})\rightarrow H_k(\{f\leq t'\})$, see \cite[Theorem IX.9.3]{EilenbergSteenrod}  (see also \cite[Appendix A]{Landi}, \cite{Kelly}, \cite[Chapter 15.2]{May}, \cite{Milnor-axiomatic}). Define $\phi:\check{V}_k(f)\rightarrow V_k(f)[\varepsilon]$ as $\phi_t=\pi_{t', t+\varepsilon}\circ I_{t'} \circ \pi_{t,t'}$.  Due to naturality of $I$,  $\phi_t$ does not depend on the choice of $t'.$ We define $\psi:V_k(f)\rightarrow \check{V}_k(f)[\varepsilon]$ in the same way, by replacing $I$ with $I^{-1}.$ Naturality of $I$ implies that $\phi$ and $\psi$ define an $\varepsilon$-interleaving which finishes the proof.
\end{proof}

In the rest of the paper we will denote $\mathcal{N}_\delta(V(f))$, $\mathcal{N}_\delta(\mathring{V}(f))$ and $\mathcal{N}_\delta(\check{V}(f))$ all by  $\mathcal{N}_\delta(f)$, while specifying which conventions are used.

As expained in Remark \ref{Less_VS_Leq} and Proposition \ref{Sing_Cech}, different conventions for filtration or choices of homology theory do not influence the bar counting function $\mathcal{N}_\delta.$ However, when we discuss algebraic properties of persistence modules, it will be useful to fix certain conventions for simplicity and clarity. To this end we call a persistence module {\it bounded from the left} if there exists $t_0\in \R$ such that $V_t=0$ for all $t<t_0.$ A persistence module is called {\it upper semi-continuous} if the canonical map $V_t \rightarrow \lim_{s>t} V_s$ to the inverse limit of the system formed by the $V_s$ for $s>t$ (and the associated structure maps) is an isomorphism for all $t\in \R.$ A $q$-tame, bounded from the left, upper semi-continuous persistence module $(V,\pi)$ has a direct product decomposition
\begin{equation}\label{SEMI_BARCODE_PRODUCT}
(V,\pi)\cong \Pi_{I\in \B(V)} (\mathbb{K}_I,\pi^{\mathbb{K}_I}),
\end{equation}
which is a genuine isomorphism (not only an isomorphism in the observable category), see \cite{Schmahl22} for details. Moreover, all bars in the above barcode are of the form $[a,b)$ or $[a,+\infty)$, $a,b\in \R.$ We also note that for a continuous function $f:X\rightarrow \R$ on a compact Hausdorff space $X$, $\check{V}_*(f)$ is bounded from the left, upper semi-continuous, see \cite{Schmahl22}, and assuming it is $q$-tame, it also has bounded spectrum. Therefore, this generality would suffice for our considerations in Sections \ref{Section_Semialgebraic} and \ref{Proof}. However, we choose to work in slightly larger generality, which is more natural for our algebraic techniques.

\begin{dfn}
	A persistence module $V$ is called {\em moderate} if it is $q$-tame, upper semi-continuous, has no intervals of the form $I=(-\infty,c)$ in its direct product decomposition, and for all $\delta>0,$ $\cl{N}_{\delta}(V)$ is finite.  
\end{dfn}

The results \cite[Theorem 5.21]{Chazal_et_al_2016} and \cite[Proposition 22]{LeRSV21} imply that the space of moderate persistence modules is naturally isometric to the completion of the space of finite barcode upper semi-continuous persistence modules bounded from the left.

\subsection{Tameness and regularization}\label{subsec:reg}

We will use the following results in Section \ref{sec: commutes} below. First, we show that one can replace every exact sequence of q-tame or finite barcode modules by a new exact sequence of upper semi-continuous q-tame or finite barcode modules which are isomorphic to the given ones in the observable category.

We call the functor $P$ from the category of q-tame persistence modules to itself, given by $P(V) = V_+$ with \[(V_+)_t = \lim_{s>t} V_s\] the {\em regularization functor}. It is equipped with a natural transformation $q: I \to P$ from the identity functor, which is given at an object $V$ in the category by the natural morphism $q_V: V \to V_{+}$ induced by the persistence structure maps $\{ \pi^V_{s,t} \}$ of $V.$ This natural transformation becomes an isomorphism after passing to the observable category by \cite{Obs16}. In this language a q-tame persistence module $V$ is upper semi-continuous if and only if $q_V:V \to V_{+}$ is an isomorphism.

\begin{lemma}\label{lma: upper sc replacement}
The regularization functor $P$ is exact. If $V$ is a q-tame persistence module, then $P(V)$ is upper semi-continuous. If $V$ is a finite barcode module, then $P(V)$ is a finite barcode module. 
\end{lemma}

\begin{remark}\label{rmk: upper sc}
More concretely, let \[0 \to A \to B \to C \to 0\] be a short exact sequence of q-tame (respectively finite barcode) modules. Then there exists a new exact sequence  \[0 \to A_+ \to B_+ \to C_+ \to 0\] of upper semi-continuous q-tame (respectively finite barcode) modules, which fits into the commutative diagram \[\xymatrix@C+2pc@R+1pc{
		0 \ar[r] & A \ar[d]^{q_A}\ar[r] & B \ar[d]^{q_B}\ar[r] &  C\ar[d]^{q_C}\ar[r] & 0\\
		0 \ar[r] & A_+\ar[r] & B_+ \ar[r] &  C_+\ar[r] & 0,}\] where all the vertical arrows induce isomorphisms in the observable category.
\end{remark}

\begin{proof}
	We first note that given $t\in \R,$ and a q-tame persistence module $V$ we may compute $\lim_{s>t} V_s$ by restricting $s$ to lie in a countable cofinal directed subset of $(t,\infty),$ for instance $\{t+1/i\}_{i \geq 1}.$ 
	
	The fact that $P:V \mapsto V_{+}$ is a functor from the category of q-tame persistence modules to itself is an easy verification. Indeed if $V$ is q-tame, then so is $V_{+}$ by an argument involving composition of structure maps.	Furthermore, every morphism $f:V \to W$ of persistence modules induces a natural morphism $P(f) = f_{+}: V_{+} \to W_{+},$ since for every $t \in \R$ it yields a morphism of inverse systems $\{V_s\}_{s>t}$ and $\{W_s\}_{s>t}$ (with suitable structure maps). Moreover it is an easy computation with inverse limits that $V_{+}$ is always upper semi-continuous. Now observe that given $t \in \R,$ and a q-tame persistence module $V,$ the inverse system $\{V_s\}_{s>t}$ satisfies the {\em Mittag-Leffler condition}, see \cite[p. 62]{Chazal_et_al_2016}. Therefore if \[0 \to A \to B \to C \to 0\] is an exact sequence of q-tame persistence modules, then \[0 \to A_s \to B_s \to C_s \to 0\] is an exact sequence of inverse systems (indexed by $s\in (t,\infty)$), and the inverse limits of these systems still form an exact sequence. The exactness on the left is automatic \cite[Section 02MY]{Stacks}, while the exactness on the right follows from the Mittag-Leffler condition \cite[Section 0594]{Stacks}. This exact sequence is \[0 \to (A_+)_t \to (B_+)_t \to (C_+)_t \to 0,\] from which it is easy to conclude that we obtained the exact sequence \[0 \to A_+ \to B_+\to C_+ \to 0,\] of q-tame persistence modules. In other words, $P$ is an exact functor.

	Finally, if $V$ is a finite barcode module, then so is $V_{+}$ since for every $t$ which is not spectral for $V,$ $V_t \to (V_{+})_t$ is a natural isomorphism, so $t$ is not spectral for $V_{+}.$ The proof is now finished by observing that $\mrm{rank}(\pi^V_{s,s'}) = \mrm{rank}(\pi^{V_{+}}_{s,s'})$ for all $s,s'$ not spectral for $V,$ which implies that $\cl{N}_0(V) = \cl{N}_0(V_{+}).$ In fact, the barcodes of $V$ and $V_{+}$ are related as follows: the bars are in bijection such that every bar $\left<a,b\right>$ for $V$ corresponds to the bar $[a,b)$ for $V_{+}.$

\end{proof}

We will also need the following lemma.

\begin{lemma}\label{Middle_q-tame} Let $U\rightarrow V \rightarrow W$ be an exact sequence of persistence modules. If $U$ and $W$ are $q$-tame then $V$ is $q$-tame as well.
\end{lemma}
\begin{proof} Let us fix $s<t$ and show that $\pi_{s,t}^V$ has finite rank. We pick an arbitrary $s<s'<t$. The following diagram commutes
\[\xymatrix@C+2pc@R+1pc{U_t \ar[r] & V_t  & \\
		U_{s'} \ar[r]\ar[u]^{\pi^U_{s',t}} & V_{s'}  \ar[u]^{\pi^V_{s',t}}\ar[r] & W_{s'} \\
		& V_s \ar[r]\ar[u]^{\pi^V_{s,s'}} & W_s \ar[u]^{\pi^W_{s,s'}}}\]
		where horizontal maps are the maps of the exact sequence. Thus, the middle row is exact and since $U$ and $W$ are $q$-tame, $\pi^U_{s',t}$ and $\pi^W_{s,s'}$ have finite rank. Now \cite[Lemma II.17.3]{Bredon97} implies that $\pi^V_{s,t}=\pi^V_{s',t}\circ \pi^V_{s,s'}$ has finite rank as well.
\end{proof}

\subsection{K\"unneth formula and duality}\label{subsec-Kunneth}

We describe K\"unneth formula for persistence modules associated to continuous functions, slightly extending its version from \cite{PSS}, see \cite{HitPer19,CarFil20, BubMil21} for subsequent works.

Let $M$ be a compact manifold without boundary and $f:M\rightarrow \R$ a continuous function. As before, denote by $\mathring{V}_*(f)_t=H_*(\{ f<t \}).$ Due to lower semi-continuity\footnote{Similarly to the upper semi-continuity, a persistence module is called lower semi-continuous if the canonical map $\colim_{s<t} V_s \rightarrow V_t$ is an isomorphism for all $t\in \R.$ } of $\mathring{V}_*(f)$, the bars in $\B(\mathring{V}_*(f))$ are of the form $(a,b]$ or $(a,+\infty)$ for $a,b\in \R$ and moreover $\mathring{V}_*(f) \cong \oplus_{I\in \B(\mathring{V}_*(f))} \mathbb{K}_I$, see \cite{Schmahl22} for details.  Note that this is a genuine isomorphism of persistence modules, while without the lower semi-continuity assumption we would only have an isomorphism in the observable category, as explained in Remark \ref{Remark_Observable}. 

For a function $f\in C^0(M)$ on a closed manifold $M$ set $\mathring{\cl{B}}_r(f) = \cl{B}(\mathring{V}_r(f)).$ Let $\mathring{\cl{B}}^{\mrm{fin}}_r(f)$ denote the sub-barcode of $\mathring{\cl{B}}_r(f)$ consisting of all its finite bars. Similarly $\mathring{\cl{B}}^{\mrm{inf}}_r(f)$ is the sub-barcode of $\mathring{\cl{B}}_r(f)$ consisting of all its infinite bars.

\begin{theorem}[K\"unneth formula] Let $M_1,M_2$ be two closed manifolds and  $f_1\in C^0(M_1),  f_2\in C^0(M_2)$. The barcode of $f_1+f_2 \in C^0(M_1 \times M_2)$ can be computed from $\mathring{\B}_*(f_1)$ and $\mathring{\B}_*(f_2)$ as follows. For each pair of bars $(a,b]\in \mathring{\B}_{k_1}(f_1)$ and $(c,d]\in \mathring{\B}_{k_2}(f_2),$ there exists a pair of bars $(a+c,\min \{a+d,b+c \}]\in \mathring{\B}_{k_1+k_2}(f_1+f_2),$ $(\max \{a+d,b+c \},b+d]\in \mathring{\B}_{k_1+k_2+1}(f_1+f_2).$ If $b=+\infty$ or $d=+\infty$ only the first bar exists in $\mathring{\B}(f_1+f_2).$
\end{theorem}
\begin{proof}
The theorem has been proven in \cite{PSS} for Morse functions. To extend the proof to continuous functions, it is enough to find $C^0$-approximating sequences of Morse functions and apply the stability theorem.
\end{proof}

We will also require the following duality statement for functions. 

\begin{prop}\label{prop: duality}
Let $M$ be a closed manifold of dimension $n$ and $f \in C^0(M).$ For every integer $0 \leq r < n,$ the barcode $\mathring{\cl{B}}^{\mrm{fin}}_{n-r-1}(-f) = \{(I'_j,m'_j)\}$ of $-f$ in degree $n-r-1$ and the barcode $\mathring{\cl{B}}^{\mrm{fin}}_r(f) = \{(I_j,m_j)\}$ of $f$ in degree $r$ are related as follows: the two indexing sets agree, $m'_j = m_j$ for all $j,$ and if $I_j=(a_j,b_j]$ then $I'_j = (-b_j,-a_j].$ Similarly, if $\mathring{\cl{B}}^{\mrm{inf}}_{r}(f) = \{ ((c_k,\infty),m_k)\}$ then $\mathring{\cl{B}}^{\mrm{inf}}_{n-r}(-f) = \{ ((-c_k,\infty),m_k)\}.$
\end{prop}

For convenience, we denote the situation described by this proposition by $\mathring{\cl{B}}^{\mrm{fin}}_{n-r-1}(-f) = - \mathring{\cl{B}}^{\mrm{fin}}_r(f),$ $\mathring{\cl{B}}^{\mrm{inf}}_{n-r}(-f) = - \mathring{\cl{B}}^{\mrm{inf}}_r(f).$

\begin{proof}
For $f$ a smooth Morse function this is well known. For instance, it is an immediate application of \cite[Proposition 6.7]{UsherZhang} for $\Gamma = 0$ applied to the Morse complexes (see for example \cite{Schwarz-book}) of $f$ and $-f$ with respect to the same Riemannian metric $\rho$ on $M,$ such that $(f,\rho)$ is Morse-Smale. For a general continuous function $f$ we pick a sequence $f_i$ of smooth Morse functions $C^0$-converging to $f.$ Then by the isometry theorem we have the convergences $\mathring{\cl{B}}^{\mrm{fin}}_{n-r-1}(-f_i) \to \mathring{\cl{B}}^{\mrm{fin}}_{n-r-1}(-f),$ $\mathring{\cl{B}}^{\mrm{fin}}_r(f_i) \to \mathring{\cl{B}}^{\mrm{fin}}_r(f)$ in the bottleneck distance. However, by the Morse case $\mathring{\cl{B}}^{\mrm{fin}}_{n-r-1}(-f_i) = - \mathring{\cl{B}}^{\mrm{fin}}_r(f_i)$ for all $i,$ whence the result follows for finite bars. A similar argument applies in the case of infinite bars.
\end{proof}

%


\section{Subadditivity of the bar counting function}\label{sec: commutes}

\subsection{Subadditivity theorem}

A crucial property of the bar counting function which we prove and use in this paper is its subadditivity for persistence modules in exact sequences. More precisely, the following theorem holds.

\begin{theorem}\label{Exact_Pers}
	Let $U\rightarrow V \rightarrow W$ be an exact sequence of moderate persistence modules. Then for every $\delta> 0$ the following inequality holds:
	$$\mathcal{N}_{2 \delta}(V) \leq  \mathcal{N}_\delta(U) + \mathcal{N}_\delta(W).$$
\end{theorem}

\begin{remark}
	In particular, Theorem \ref{Exact_Pers} applies to finite barcode modules which are upper semi-continuous and bounded from the left (upper semi-continuity can in fact be dropped by an application of Lemma 
	\ref{lma: upper sc replacement}).
	
	On a different note, we expect that the same statement should hold for $U, V, W$ being arbitrary q-tame persistence
	modules. However, this generality is not necessary for us in this paper.
\end{remark}
	
In this section we present a proof of Theorem \ref{Exact_Pers} as well as its extension which takes into account the positions of the starting points of bars. This is a key technical tool from the theory of persistence modules and barcodes. It allows us to make local-to-global estimates which are crucial for the multiscale argument in the proof of the main technical result, Theorem \ref{thm: main 2}.

\subsection{Proof of Theorem \ref{Exact_Pers}}\label{subsec:subadd}
The main technical result we will need is the following proposition.
\begin{prop}\label{prop: Ndelta}
Let \[ 0 \to A \to B \to C \to 0 \] be a short exact sequence of finite barcode modules bounded from the left. Then for every $\delta \geq 0,$ \[ \cl N_{2\delta}(B) \leq \cl N_{\delta}(A) + \cl N_{\delta}(C).\] Moreover, \[ \cl{N}_{\delta}(A) \leq \cl{N}_{\delta}(B),\] \[\cl{N}_{\delta}(C) \leq \cl{N}_{\delta}(B).\]
\end{prop}

We defer proving Proposition \ref{prop: Ndelta} and first show how it implies Theorem \ref{Exact_Pers}.

Recall that for a persistence module $V$ and a real number $a \in \R,$ the shift $V[a]$ of $V$ by $a$ is defined as \[ V[a]_t = V_{a+t}.\] If $a \geq 0,$ there is a canonical shift morphism \[sh_{a,V}: V \to V[a]\] given by \[(sh_{a,V})_t = \pi^V_{t, a+t}: V_t \to V_{a+t} = V[a]_t\] for $\pi^V_{s,t}: V_s \to V_t,$ $s \leq t$ the structure maps of the persistence module $V.$ Denote by $V^{(a)}=\im(sh_{a,V}).$

\begin{lemma}\label{Shift_counting_lemma}
Let $V$ be a moderate persistence module. For all $\delta>0$, $V^{(\delta)}$ is a finite barcode module and $\cl N_\delta(V)=\cl N_0 (V^{(\delta)}).$
\end{lemma}
\begin{proof}
For $I=[a,b)$ we have that if $b-a>\delta$, $\mathbb{K}_I^{(\delta)}=\mathbb{K}_{I^{(\delta)}}$ where $I^{(\delta)}=[a,b-\delta)$ and $\mathbb{K}_I^{(\delta)}=0$ otherwise. Due to barcode decomposition \eqref{SEMI_BARCODE_PRODUCT}, we have that $V^{(\delta)}\cong \Pi_I\mathbb{K}_I^{(\delta)}$, the product going over all $I\in \B(V)$ of length greater than $\delta.$ Since $\cl N_\delta$ of a moderate persistence module is finite,  this product is finite and the claim follows. 
\end{proof}

We will also require the following auxiliary results. Recall that for a q-tame persistence module $U,$ we denote by $U_{+}$ its upper semi-continuous regularization, defined in Section \ref{subsec:reg}.

\begin{lemma}\label{lma: sub}
	Let $i:U \to V$ be an injection of q-tame persistence modules, such that $V$ is upper semi-continuous. Then the natural map $q_U: U \to U_+$ is injective.
\end{lemma}

\begin{proof}
	The maps $i,q_U,$ the induced map $i_+: U_+ \to V_+,$ and the natural map $q_V: V \to V_+,$ which is an isomorphism, fit into the commutative diagram: \[\xymatrix@C+2pc@R+1pc{
		U  \ar[d]^{q_U}\ar[r]^{i} & V \ar[d]^{q_V}\\
		 U_+ \ar[r]^{i_+} & V_+.}\] Now $q_V \circ i = i_+ \circ q_U$ is injective, and therefore $q_U$ is injective.
\end{proof}

\begin{lemma}\label{lma: im ker}
Let $f:A \to B$ be a morphism of moderate persistence modules. Then $\ker(f)$ and $\im(f)$ are moderate persistence modules.
\end{lemma}

\begin{remark}
We can complete the proof of Theorem \ref{Exact_Pers} using either one of $\ker(f)$ or $\im(f)$ being moderate. We include both statements in the lemma for the sake of completeness, and opt to use the latter one in our exposition.
\end{remark}

\begin{proof}
Let $K = \ker(f)$ and $J = \im(f).$ These are q-tame persistence modules, as submodules of q-tame persistence modules. They fit into the exact sequence \[ 0 \to K \to A \to J \to 0.\] Let us prove that $K, J$ are upper semi-continuous.

First, $K, J$ are submodules of upper semi-continuous q-tame modules. By Lemma \ref{lma: sub} these two facts imply that the natural maps $q_K: K \to K_+$ and $q_J: J \to J_+$ are injective. It remains to show that they are surjective. By Lemma \ref{lma: upper sc replacement} or Remark \ref{rmk: upper sc} we have an induced short exact sequence \[ 0 \to K_+ \to A_+ \to J_+ \to 0 \] of q-tame persistence modules which fits into the commutative diagram \[\xymatrix@C+2pc@R+1pc{
	0 \ar[r] & K \ar[d]^{q_K}\ar[r] & A \ar[d]^{q_A}\ar[r] &  J\ar[d]^{q_J}\ar[r] & 0\\
	0 \ar[r] & K_+\ar[r] & A_+ \ar[r] &  J_+\ar[r] & 0.}\]

Now the surjectivity of both $q_K$ and $q_J$ is a quick diagram chase. For instance, let $t \in \R$ and $y_+ \in (J_{+})_t.$ Then there exists $x_+ \in (A_+)_t$ which maps to $y_+.$ Then, since $q_A$ is an isomorphism, $x_+ = (q_A)_t(x)$ for some $x \in A_t.$ Let $x$ map to $y \in J_t.$ Then $q_J(y) = y_+$ by the commutativity of the right square.

Moreover it is easy to show that $\cl N_{\delta}(K), \cl N_{\delta}(J)$ are finite for all $\delta > 0$ for instance by the same argument as for Lemma \ref{Shift_counting_lemma}. Finally, their barcodes do not contain negative rays since those of $A, B$ do not. This finishes the proof.
\end{proof}

\begin{lemma}\label{Q-tame_extension}
Proposition \ref{prop: Ndelta} remains true for $\delta>0$ if we only assume that $A, B$ and $C$ are moderate. 
\end{lemma}
\begin{proof}
We first prove the moreover part of the proposition. Denote the maps in the exact sequence by 
$$0\to A \xrightarrow{\phi} B \xrightarrow{\psi} C \to 0$$
One readily checks that $\phi^{(\delta)}=\phi[\delta]|_{A^{(\delta)}}:A^{(\delta)}\rightarrow B^{(\delta)}$ is injective, while $\psi^{(\delta)}=\psi[\delta]|_{B^{(\delta)}}:B^{(\delta)}\rightarrow C^{(\delta)}$ is surjective.  We may complete these maps to short exact sequences
$$0\to A^{(\delta)}\xrightarrow{\phi^{(\delta)}} B^{(\delta)} \to \coker \phi^{(\delta)} \to 0 ,~0\to \ker \psi^{(\delta)} \to B^{(\delta)}\xrightarrow{\psi^{(\delta)}} C^{(\delta)} \to  0.$$
Persistence modules in these sequences are finite barcode modules bounded from the left and hence Proposition \ref{prop: Ndelta} implies that $\cl N_0(A^{(\delta)}) \leq \cl N_0(B^{(\delta)})$ and $\cl N_0(C^{(\delta)}) \leq \cl N_0(B^{(\delta)}).$ These inequalities together with Lemma \ref{Shift_counting_lemma} finish the proof of the moreover part.

Now, let us fix a decomposition of $B$ as in \eqref{SEMI_BARCODE_PRODUCT} and let $B'$ be a submodule of $B$ obtained by taking only summands corresponding to bars of length greater than $2\delta.$ Then $B'$ is a finite barcode module such that $\cl N_{2\delta} (B')=\cl N_{2\delta} (B)$ and we consider a short exact sequence
$$0 \to X \xrightarrow{f} B' \xrightarrow{g} Y \to 0,$$
where $X=\phi^{-1}({B'}),$ $Y=\psi({B'}),$ $f = \phi|_X,$ $g = \psi|_{B'}.$ Note that $B'$ is a finite barcode module bounded from the left, and hence so are $X$ and $Y$ as its submodule and quotient module respectively. Indeed, $X$ and $Y$ are pointwise finite dimensional, and therefore have barcode normal forms. Then their barcodes are finite by a local calculation and evidently bounded from the left.

Hence, we may apply Proposition \ref{prop: Ndelta} to obtain $$\cl N_{2\delta} (B)=\cl N_{2\delta} (B') \leq \cl N_\delta (X)+\cl N_\delta (Y).$$ Furthermore, since $X$ is a submodule of $A$ and $Y$ is a submodule of $C$, the moreover part which we already proved implies that $\cl N_\delta (X)\leq \cl N_\delta (A), \cl N_\delta (Y)\leq \cl N_\delta (C)$ which finishes the proof. \end{proof}

\begin{proof}[Proof of Theorem \ref{Exact_Pers}]
Denote the maps in the exact sequence by
$$U \xrightarrow{f} V \xrightarrow{g} W$$
and consider the induced exact sequence \[0 \to \mrm{im}(f) \to V \to \mrm{im}(g) \to 0.\] By Lemma \ref{lma: im ker} $\im(f), \im(g)$ are moderate. Now by Proposition \ref{prop: Ndelta} and Lemma \ref{Q-tame_extension} we obtain that \[\cl{N}_{\delta}(\mrm{im}(f)) \leq \cl{N}_{\delta}(U),\] \[\cl{N}_{\delta}(\mrm{im}(g)) \leq \cl{N}_{\delta}(W),\] since $f: U \to \mrm{im}(f)$ is surjective and the inclusion $\mrm{im}(g) \to W$ is injective. Now by Proposition \ref{prop: Ndelta} and Lemma \ref{Q-tame_extension} again, we obtain \[ \cl{N}_{2\delta}(V) \leq \cl{N}_{\delta}(\mrm{im}(f)) + \cl{N}_{\delta}(\mrm{im}(g)) \leq \cl{N}_{\delta}(U) + \cl{N}_{\delta}(W). \]
\end{proof}

Before we proceed with the proof of Proposition \ref{prop: Ndelta}, we require a few preparatory notions and results.  We start with the following key definition.	
	
\bs

\begin{dfn}\label{def: close maps}	
	
For two morphisms \[f:X \to Y, \; f':X' \to Y'\] of persistence modules, we say that $f, f'$ are $(\de_1,\de_2; \de'_1,\de'_2)$-close if there are $(\de_1,\de_2)$-interleavings \[p_X: X \to X'[\de_1], q_X: X' \to X[\de_2],\]  $q_X[\de_1] \circ p_X = sh_{\de_1+\de_2, X},$ $p_X[\de_2] \circ q_X = sh_{\de_1+\de_2, X'},$ and
\[p_Y: Y \to Y'[\de_1], q_Y: Y' \to Y[\de_2],\] $q_Y[\de_1] \circ p_Y = sh_{\de_1+\de_2, Y},$ $p_Y[\de_2] \circ q_Y = sh_{\de_1+\de_2, Y'},$ such that the following condition holds: \begin{align}\label{eq:closeness} sh_{\de'_1,Y'[\de_1]} \circ (p_Y \circ f - f'[\de_1] \circ p_X) = 0,\\ \nonumber
sh_{\de'_2,Y[\de_2]} \circ (q_Y \circ f' - f[\de_2] \circ q_X) = 0.\end{align} In other words $\mrm{im}(p_Y \circ f - f'[\de_1] \circ p_X)$ and $\mrm{im}(q_Y \circ f' - f[\de_2]\circ q_X)$ are respectively ${\de'_1}/2$ and ${\de'_2}/2$-close to $0$ in the bottleneck distance. \end{dfn}

We prove the following lemma.

\bs

\begin{lemma}\label{lma: close maps}
If $f, f'$ are $(\de_1,\de_2;\de'_1,\de'_2)$-close then their cokernels $\mrm{coker}(f),$ $\mrm{coker}(f')$ are $(\de_1+\de'_1, \de_2+\de'_2)$-interleaved.
\end{lemma}

\begin{proof}
Set $C = \mrm{coker}(f),$ $C' = \mrm{coker}(f')$ and let $\pi = \pi_C: Y \to C, \pi'=\pi_{C'}:Y \to C'$ be the natural projections.	
	
We will first construct the interleavings \[p_C: C \to C'[\de_1+\de'_1], q_C: C' \to C[\de_2+\de'_2],\] and then show that they satisfy the interleaving identities \[q_C[\de_1+\de'_1] \circ p_C = sh_{\delta, C},\] \[p_C[\de_2+\de'_2] \circ q_C = sh_{\delta, C'}\] for \[\delta = \de_1+\de'_1 + \de_2+\de'_2.\]

Note that to construct $p_C: C \to C'[\de_1+\de'_1],$ it is enough to construct $\til{p}_C: Y \to C'[\de_1+\de'_1]$ such that $\til{p}_C \circ f = 0.$ Set $\til{p}_C = \pi'[\de_1+\de'_1] \circ sh_{\de'_1,Y'[\de_1]} \circ p_Y.$

Then by condition \eqref{eq:closeness} \[\til{p} \circ f = \pi'[\de_1+\de'_1] \circ (sh_{\de'_1,Y'[\de_1]} \circ p_Y \circ f) = \pi'[\de_1+\de'_1]  \circ  (sh_{\de'_1,Y'[\de_1]} \circ f'[\de_1] \circ p_X ) = \] \[ = (\pi'[\de_1+\de'_1]  \circ  sh_{\de'_1,Y'[\de_1]}) \circ f'[\de_1] \circ p_X  =  sh_{\de'_1,C'[\de_1]}  \circ \pi'[\de_1]  \circ f'[\de_1] \circ p_X = 0,\] since $\pi'[\de_1] \circ f'[\de_1] = 0$ by definition of cokernel. This yields our desired map $p_C.$ The map $q_C$ is constructed similarly.

Now let us check that $q_C[\de_1+\de'_1] \circ p_C = sh_{\delta,C},$ where $\de = \de_1+\de'_1+\de_2+\de'_2.$ It is enough to check that ${q}_C[\de_1+\de'_1] \circ \til{p}_C = \pi_C[\delta] \circ sh_{\delta,Y} = sh_{\delta,C} \circ \pi_C.$ Indeed, note that $\til{p}_C = p_C \circ \pi_C$ so we would get that the desired identity $q_C[\de_1+\de'_1] \circ p_C = sh_{\delta,C}$ holds on the image of $\pi_C,$ which is surjective, so it holds in general.

In turn, it is now enough to calculate that \[\pi_C[\delta] \circ sh_{\de'_2, Y[\de_1+\de'_1+\de_2]} \circ (q_Y[\de_1+\de'_1] \circ sh_{\de'_1,Y'[\de_1]}) \circ p_Y =\] \[= \pi_C[\delta] \circ (sh_{\de'_2, Y[\de_1+\de'_1+\de_2]} \circ sh_{\de'_1,Y[\de_1+\de_2]}) \circ (q_Y[\de_1]  \circ p_Y) = \] \[ = \pi_C[\delta] \circ sh_{\de'_1+\de'_2,Y[\de_1+\de_2]} \circ sh_{\de_1+\de_2,Y} = \pi_C[\delta] \circ sh_{\delta,Y}.\]

\end{proof}

For a finite barcode module $V,$ let $N(V) = \cl N_0(V)$ denote the total number of bars of positive length in the barcode of $V.$ If $V$ is in addition upper semi-continuous, $N(V)$ is equal to the total number of bars in its barcode, since there are no bars of length zero. 

\bs

\begin{lemma}\label{lma: exact sequences}
Let \[ 0 \to A \to B \to C \to 0 \] be a short exact sequence of finite barcode modules bounded from the left. Then \[ N(B) \leq N(A) + N(C).\] Moreover, \[ N(A) \leq N(B),\] \[N(C) \leq N(B).\]

\end{lemma}

\begin{proof}
	
We apply Lemma \ref{lma: upper sc replacement} to assume that $A,B,C$ are upper semi-continuous.	
	
Observe that for an upper semi-continuous persistence module $V$ bounded from the left the number $N(V)$ of bars in the barcode of $V$ is equal to the number of left endpoints of bars for $V.$ The number $K(V)$ of finite bars in the barcode of $V$ is equal to the number of (finite) right endpoints of bars for $V.$ Finally, set $I(V)$ for the number of infinite bars in the barcode of $V.$

If $\eps > 0$ is smaller than the minimal gap in the spectrum of $V,$ then for every spectral point $x$ of $V$ the number $N(V,x)$ of bars starting at $x$ satisfies: \[ N(V,x) = \dim L(V,x),\] \[ L(V,x) = \mrm{coker}(\pi^V_{x-\eps,x+\eps}: V_{x-\eps} \to V_{x+\eps}),\] where $\pi^V_{s,t}: V_s \to V_t$ for $s \leq t$ are the structure maps of the persistence module $V.$ Similarly, the number of bars $K(V,x)$ ending at $x$ satisfies: \[K(V,x) = \dim R(V,x),\] \[ R(V,x) = \mrm{ker}(\pi^V_{x-\eps,x+\eps}: V_{x-\eps} \to V_{x+\eps}).\] 
	
Now, in the setting of our short exact sequence, let $\eps > 0$ be smaller than the minimal gap in the union of the spectra of $A, B, C.$ Let $x$ be a spectral point for $A, B,$ or $C.$ Then applying the snake lemma to the following commutative diagram \[
\xymatrix@C+1pc@R+1pc{
	0 \ar[r] & A_{x-\eps}  \ar[d]^{\pi^A_{x-\eps,x+\eps}} \ar[r] & B_{x-\eps} \ar[d]^{\pi^B_{x-\eps,x+\eps}} \ar[r] & C_{x-\eps} \ar[d]^{\pi^C_{x-\eps,x+\eps}} \ar[r] & 0 \\
	0 \ar[r] & A_{x+\eps} \ar[r] & B_{x+\eps} \ar[r] & C_{x+\eps} \ar[r] & 0} \]
yields the exact sequences of cokernels \[ L(A,x) \to L(B,x) \to L(C,x) \to 0\] and kernels \[ 0 \to R(A,x) \to R(B,x) \to R(C,x).\] To prove the first statement, we let $x$ be spectral for $B$ and  calculate dimensions for the cokernel exact sequence. This yields \[ N(B,x) \leq N(A,x) + N(C,x).\] Summing over all spectral points $x$ for $B,$ we obtain \[  N(B) \leq N(A) + N(C),\] as desired. 

To prove the moreover part we first suppose that $x$ is spectral for $C$ and compute dimensions for cokernels to obtain \[ N(B,x) \geq N(C,x)\] and sum up over all such $x$ to get $ N(B) \geq N(C).$

Then we suppose that $x$ is spectral for $A$ and compute dimensions for kernels to get \[ K(A,x) \leq K(B,x).\] Summing up over all such $x$ we obtain that $K(A) \leq K(B).$ However, the numbers of infinite bars in $A, B, C$ satisfy \[ I(B) = I(A) + I(C) \geq I(A),\] hence \[ N(A) \leq N(B).\] This finishes the proof. 
\end{proof}

\bs

Now we are ready to proceed to the proof of the main proposition.

\begin{proof}[Proof of Proposition \ref{prop: Ndelta}]
	
We first apply Lemma \ref{lma: upper sc replacement} to assume that $A,B,C$ are upper semi-continuous.	
	
To prove the moreover part it suffices to notice that for a persistence module $V$ and $\delta \geq 0,$ \[\cl{N}_{\delta}(V) = N(V^{(\delta)})\] for $V^{(\delta)} = \mrm{im}(sh_{\delta,V}).$ Now in our situation $A^{(\delta)} \to B^{(\delta)}$ is injective and $B^{(\delta)} \to C^{(\delta)}$ is surjective, hence by the moreover part of Lemma \ref{lma: exact sequences} we obtain the desired inequality. 

This motivates our approach to the main part of the proposition: we reduce it to Lemmas \ref{lma: exact sequences} and \ref{lma: close maps} by a suitable key construction.
	
Inspired by \cite[Section 8]{Skraba-Turner} we let \[0 \to R \xrightarrow{j} G \to C \to 0\] be a projective resolution of $C$ given by resolving every finite elementary module $\bK[a,b)$ in a normal form decomposition of $C$ by \[ 0 \to \bK[b,\infty) \to \bK[a,\infty) \to \bK[a,b) \to 0.\]

Observe that in view of the theory of extension groups, $B$ considered as an extension of $C$ by $A$ is obtained from a homomorphism \[ g:R \to A.\] Namely \[ B \cong \mrm{coker}(j \oplus g),\] for the monomorphism \[j \oplus g: R \to G \oplus A\] of persistence modules. 

The key construction in this proof reduces Proposition \ref{prop: Ndelta} to Lemmas \ref{lma: exact sequences} and \ref{lma: close maps}. We proceed as follows.

Let $A',C'$ be the submodules of $A, C$ obtained by erasing all direct summands in the normal form decompositions of $A, C$ corresponding to bars $[a,b)$ of length $b-a \leq \delta.$ Let \[ p: A \to A',\; i: A' \to A\] be the natural projection and injection.

Observe that $A, A'[\delta]$ are $(0,\delta)$-interleaved. Indeed \[\rho = p[\delta] \circ sh_{\delta, A} : A \to A'[\delta]\] \[ \sigma = i[\delta]: A'[\delta] \to A[\delta]\] provides a $(0,\delta)$-interleaving.

Consider the projective resolution \[ 0 \to R' \xrightarrow{j'} G \to C' \to 0 \] of $C',$ where $R' \xrightarrow{j'} G$ is obtained from $R \xrightarrow{j} G$ by keeping every direct summand $\bK[b,\infty) \to \bK[a,\infty)$ corresponding to a bar $[a,b)$ of length $> \delta,$ and changing every summand $\bK[b,\infty) \to \bK[a,\infty)$ corresponding to a bar of length $\leq \delta$ to $\bK[b',\infty) \to \bK[a,\infty)$ for $b' = a.$

Note that there are natural maps \[\mu: R \to R',\; \nu: R' \to R[\delta].\] They provide a $(0,\delta)$-interleaving.

Let us now construct an extension \[0 \to A'[\delta] \to B' \to C' \to 0\] of $C'$ by $A'[\delta]$ by considering homomorphism \[ g': R' \to A'[\delta]\] defined as the composition \[ R' \xrightarrow{\nu} R[\delta] \xrightarrow{g[\delta]} A[\delta] \xrightarrow{p[\delta]} A'[\delta]\] and setting \[ B' =\mrm{coker}(j' \oplus g'),\] for the map $j' \oplus g': R' \to G \oplus A'[\delta].$ 

By Lemma \ref{lma: exact sequences} we obtain \[ N(B') \leq N(A'[\delta]) + N(C') = \cl{N}_{\delta}(A) + \cl{N}_{\delta}(C).\]

It is therefore sufficient to prove that \[ N(B') \geq \cl{N}_{2\delta}(B),\] which would follow directly from the isometry theorem if $B'$ and a shift $B[a]$ of $B$ for suitable $a \in \R$ are $\delta$-interleaved.

This indeed holds by Lemma \ref{lma: close maps} combined with the following statement.

\bs

\begin{lemma}\label{lma: close check}
The maps $j \oplus g: R \to G \oplus A$ and $j' \oplus g': R' \to G \oplus A'[\delta]$ are $(0,\delta; 0, \delta)$-close. 
\end{lemma}

\begin{proof}

Indeed, let us first prove that the following diagram is commutative:

\[\xymatrix@C+2pc@R+1pc{
	  R  \ar[d]^{\mu} \ar[r]^{j \oplus g} & G \oplus A \ar[d]^{\id \oplus \rho} \\
	R' \ar[r]^{j' \oplus g'} & G \oplus A'[\delta]}\]

Indeed \[(\id \oplus \rho) \circ (j \oplus g) = j \oplus p[\delta]\circ sh_{\delta,A} \circ g\] and \[(j' \oplus g') \circ \mu = j'\circ \mu \oplus g'\circ \mu = j \oplus p[\delta] \circ g[\delta] \circ \nu \circ \mu,\] however $g[\delta] \circ \nu \circ \mu = g[\delta] \circ sh_{\delta,R} = sh_{\delta,A}\circ g,$ which finishes the first part of the proof. 

Now consider the diagram:

\[\xymatrix@C+2pc@R+1pc{
	R'  \ar[d]^{\nu} \ar[r]^{j' \oplus g'} & G \oplus A'[\delta] \ar[d]^{sh_{\delta,G} \oplus \sigma} \\
	R[\delta] \ar[r]^{j[\delta] \oplus g[\delta]} & G[\delta] \oplus A[\delta]}\]

Let us prove that it commutes up to $\delta$ in the sense that \[ sh_{\delta, G[\delta] \oplus A[\delta]}\circ (sh_{\delta,G} \oplus \sigma) \circ (j' \oplus g')  = sh_{\delta, G[\delta] \oplus A[\delta]} \circ (j[\delta] \oplus g[\delta]) \circ \nu.\] Let us establish this component-wise. The first components coincide since \[sh_{\delta,G[\delta]} \circ j[\delta] \circ \nu = sh_{\delta,G[\delta]} \circ j'[\delta] \circ \mu[\delta] \circ \nu = sh_{\delta,G[\delta]} \circ j'[\delta] \circ sh_{\delta, R'} = sh_{\delta, G[\delta]} \circ sh_{\delta,G} \circ j'.\] The second components coincide since \[ sh_{\delta, A[\delta]} \circ \sigma \circ g' = sh_{\delta, A[\delta]} \circ \sigma \circ p[\delta] \circ g[\delta] \circ \nu =  (\sigma \circ \rho)[\delta] \circ g[\delta] \circ \nu = sh_{\delta, A[\delta]} \circ g[\delta] \circ \nu.\] This finishes the proof of the lemma.

\end{proof}

Now by Lemma \ref{lma: close maps}, $B = \mrm{coker}(j \oplus g)$ and $B' = \mrm{coker}(j' \oplus g')$ are $(0,2\delta)$-interleaved and hence $B$ and $B'[-\delta]$ are $\delta$-interleaved. This finishes the proof.
\end{proof}

\subsection{Subadditivity with controlled endpoints}

We will later require the following sharpening of Proposition \ref{prop: Ndelta} and Theorem \ref{Exact_Pers}, which is proven using similar methods. For a persistence module $V,$ $\delta \geq 0,$ and a subset $X \subset \R$ denote by \[ \cl N_{\delta}(V,X)\] the number of bars of length $>\delta$ in the barcode of $V,$ {\em which start at a point of $X.$} Recall that for another subset $Y \subset \R,$ one denotes $X+Y = \{ x+y|\; x\in X, \, y\in Y \}.$ 

\bs

\begin{theorem}\label{prop: Ndelta2}
Let \[ 0 \to A \to B \to C \to 0 \] be a short exact sequence of moderate persistence modules.  Then for every $Z \subset \R,$ $\delta > 0,$ \[ \cl{N}_{2\delta}(B,Z) \leq \cl{N}_{\delta}(A,Z+[-\delta,\delta])+ \cl{N}_{\delta}(C,Z+[-2\delta,0]),\] and moreover \[ \cl N_{\de}(C,Z) \leq \cl N_{\de}(B,Z).\]
\end{theorem}

The following consequence shall be of use in Section \ref{sec: MV bezout prod}. Set \begin{equation}\label{eq: N delta 0 def} \cl{N}_{\delta}^0(V) = \cl{N}_{\delta}(V,\{0\}).\end{equation} Call a persistence module $V$ {\em non-negatively supported} if $V_t = 0$ for all $t<0.$

\begin{cor}\label{cor: subadd 0}
Suppose that $U \xrightarrow{f} V \xrightarrow{g} W$ is an exact sequence of non-negatively supported moderate persistence modules. Then \[\cl{N}_{2\delta}^0(V) \leq \cl{N}_{\delta}(U,[0,\delta])+ \cl{N}_{\delta}^0(W).\]
\end{cor}

Let us now prove Theorem \ref{prop: Ndelta2} and Corollary \ref{cor: subadd 0} by a couple extra arguments similar to those in Section \ref{subsec:subadd}. 

\begin{proof}[Proof of Theorem \ref{prop: Ndelta2}]
	
Let $\de>0, Z\subset \R.$ We first prove the moreover part. We proceed like in the proof of Proposition \ref{prop: Ndelta}, the only difference being the additional observation that \[ \cl N_{\de}(V,Z) = N(V^{(\de)}[\de],Z)\] and that the moreover part holds for finite barcode modules. The latter statement holds by summing up the local inequality $N(B,x) \geq N(C,x)$ over all $x \in Z$ which are spectral for $C.$ 
	
To prove the main inequality, as in the proof of Proposition \ref{prop: Ndelta}, we first suppose that $A,B,C$ are finite barcode modules and observe that for all $Z \subset \R,$ \[\cl{N}_{0}(B,Z) \leq \cl{N}_{0}(A,Z) + \cl{N}_{0}(C,Z).\] This follows by summing up the local inequality $N(B,x) \leq N(A,x) + N(C,x)$ over all $x \in Z$ which are spectral for $B.$

Let $0 \to A'[\delta] \to B' \to C' \to 0$ be the exact sequence introduced in the proof of Proposition \ref{prop: Ndelta}, where we showed that the modules $B$ and $B'$ are $(0,2\delta)$-interleaved. Since this is equivalent to $B$ and $B'[-\delta]$ being $\delta$-interleaved, this means that after erasing certain bars of length $<2\delta$ from the barcodes $\cl{B}(B),$ $\cl{B}(B')$, there is a bijection $\Phi: \cl{B}^{\delta}(B) \to \cl{B}^{\delta}(B')$ between the resulting barcodes $\cl{B}^{\delta}(B), \cl{B}^{\delta}(B')$, such that $\Phi(\langle a,b \rangle) = \langle c,d \rangle$ implies $c \in \{a\}+[-2\delta,0],$ $d \in \{b\} + [-2\delta,0].$ This yields \[ \cl{N}_{2\delta}(B,Z) \leq \cl{N}_{0}(B',Z+[-2\delta,0]).\] In turn \[ \cl{N}_{0}(B',Z+[-2\delta,0]) \leq \cl{N}_{0}(A'[\delta],Z+[-2\delta,0]) + \cl{N}_{0}(C',Z+[-2\delta,0]) =\] \[= \cl{N}_{\delta}(A,Z+[-\delta,\delta]) + \cl{N}_{\delta}(C,Z+[-2\delta,0]).\]

Now for $A,B,C$ moderate, we pass to the short exact sequence \begin{equation}\label{eq: quotient replacement} 0 \to X' \to B' \to Y' \to 0\end{equation} 
where $B'$ is defined as in the proof of Lemma \ref{Q-tame_extension}. In particular $\cl{N}_{2\delta}(B,Z) = \cl{N}_{2\delta}(B',Z).$ We define $X', Y'$ as follows. Observe first that there is a natural map $p_B: B \to B'.$ Let $L' = \ker(p_B).$ This is a submodule of $B.$ Let $M' = \psi(L')$ and $K' = \phi^{-1}(L').$ These are submodules of $C$ and $A$ respectively. In total we obtain the diagram of short exact sequences \[
\xymatrix@C+1pc@R+1pc{
	0 \ar[r] & K'  \ar[d]^{i_A} \ar[r] & L' \ar[d]^{i_B} \ar[r] & M' \ar[d]^{i_C} \ar[r] & 0 \\
	0 \ar[r] & A \ar[r] & B \ar[r] & C\ar[r] & 0,}\] where the vertical maps $i_A, i_B, i_C,$ are the natural inclusions. Setting $X' =\coker(i_A),$ $Y' = \coker(i_C)$ and noting that $B' \cong \coker(i_B)$ by construction, the snake lemma produces the short exact sequence \eqref{eq: quotient replacement}, as desired, since $i_C$ is injective.

Now by the finite module case: \[ \cl{N}_{2\delta}(B,Z) = \cl{N}_{2\delta}(B',Z)  \leq \cl{N}_{\delta}(X',Z+[-\delta,\delta]) + \cl{N}_{\delta}(Y',Z+[-2\delta,0]) \leq\] \[\leq \cl{N}_{\delta}(A,Z+[-\delta,\delta]) + \cl{N}_{\delta}(C,Z+[-2\delta,0]).\] In the last step we used the fact that $X', Y'$ are quotient modules of $A, C$ and the moreover part of the theorem.  
\end{proof}

\begin{remark}
In the proof of Theorem \ref{prop: Ndelta2} we could not use the same finite barcode replacement $0 \to X \to B' \to Y \to 0$ as in the proof of Lemma \ref{Q-tame_extension}, since the moreover part of the theorem {\em does not} hold for $C$ a submodule of $B$ instead of a quotient module. We expect that this replacement would allow one to prove an analogue of Theorem \ref{prop: Ndelta2} where the control is on the right endpoints of the bars instead of their left endpoints. We do not require such an analogue in this paper.
\end{remark}

\begin{proof}[Proof of Corollary \ref{cor: subadd 0}]
As in the proof of Theorem \ref{Exact_Pers}, we replace the exact sequence by the short exact sequence $0 \to A \to B \to C \to 0$ where $A = \im(f),$ $B=V,$ $C = \im(g)$ are still moderate. As in the proof of Proposition \ref{prop: Ndelta} we see that for every $Z \subset \R,$ $\cl{N}_{\delta}(A,Z) \leq \cl{N}_{\delta}(U,Z).$

Now, by Theorem \ref{prop: Ndelta2} and non-negative support, \[\cl N_{2\delta} (V,\{0\}) \leq \cl{N}_{\delta}(A,[-\delta,\delta]) + \cl{N}_{\delta}(C,[-2\delta,0]) \leq \] \[\leq \cl{N}_{\delta}(U,[0,\delta]) + \cl{N}_{\delta}(C,\{0\}).\] 

We claim that $\cl{N}_{\delta}(C,\{0\}) \leq \cl{N}_{\delta}(W,\{0\}).$ This would imply \[\cl N_{2\delta}^0(V) \leq \cl N_{\delta}(U,[0,\delta]) + \cl N_{\delta}^0(W)\] as required. To prove the claim, note that for an upper semi-continuous non-negatively supported persistence module $Q,$ \[ \cl{N}^0_{\delta}(Q) = \mrm{rank} (\pi_{0,\delta}^Q).\] Applying this identity to $C$ and $W,$ it remains to show that $\mrm{rank} (\pi_{0,\delta}^C) \leq \mrm{rank} (\pi_{0,\delta}^W),$ which is evident because $\pi_{0,\delta}^C = \pi_{0,\delta}^W|_{C_0}.$

\end{proof}


\section{Multiscale polynomial approximation and cube counting}\label{Section_Semialgebraic}

The goal of this section is to prove a polynomial, multiscale version of the simplex counting method from \cite{CSEHM}, see also \cite{PRSZ}. It is given as Theorem \ref{MDP_Simplex_counting}.

\subsection{The result}
We start by introducing a notion of a {\it multiscale dyadic partition} of $[0,1]^n$, which will be central in our arguments.

\begin{dfn} \label{dyadic_cube}
	Let $l$ be a positive integer. A set $\sigma \subset \R^n$ given by $\sigma= \Big[ \frac{m_1}{2^l}, \frac{m_1+1}{2^l} \Big] \times \ldots \times \Big[ \frac{m_n}{2^l}, \frac{m_n+1}{2^l} \Big] $ for some $m_1, \ldots, m_n \in \Z$ is called a standard dyadic cube of size $\frac{1}{2^l}.$
\end{dfn}

\begin{dfn} \label{MDP}
	A multiscale dyadic partition of $[0,1]^n$ is a finite set $K=\{\sigma_1 ,  \ldots , \sigma_{|K|}  \}$ of standard dyadic cubes such that $\bigcup\limits_{i=1}^{|K|} \sigma_i = [0,1]^n$ and $\interior(\sigma_i) \cap \interior(\sigma_j) =\emptyset$ for $i\neq j.$ We abbreviate multiscale dyadic partition to MDP.
\end{dfn}

\begin{remark}
	By convention, we consider dyadic cubes to be closed. Hence,  an MDP is not a genuine partition of $[0,1]^n$, since dyadic cubes may intersect along faces of positive codimension.  Nevertheless, the interiors of dyadic cubes form a genuine partition of a subset of $[0,1]^n$ of full measure.
\end{remark}

One may construct an MDP of $[0,1]^n$ as follows. Firstly, we divide $[0,1]^n$ into $2^n$ standard dyadic cubes of size $\frac{1}{2}$ by median hyperplanes. Then, we choose a subset of these $2^n$ cubes and further divide each cube in this subset into $2^n$ cubes of size $\frac{1}{2^2}$ by median hyperplanes. We proceed to divide certain cubes of size $\frac{1}{2^2}$ into $2^n$ cubes of size $\frac{1}{2^3}$ and repeat this procedure finitely many times. The set of all cubes we obtain in the end is an MDP of $[0,1]^n.$ One may check that each MDP of $[0,1]^n$ can be obtained using the described algorithm. In other words, the set of MDPs is in bijection with the set of ordered, full, $2^n$-ary trees, see Figure \ref{MDP_Tree}.

\begin{figure}[ht]
	\begin{center}
		\includegraphics[scale=0.6]{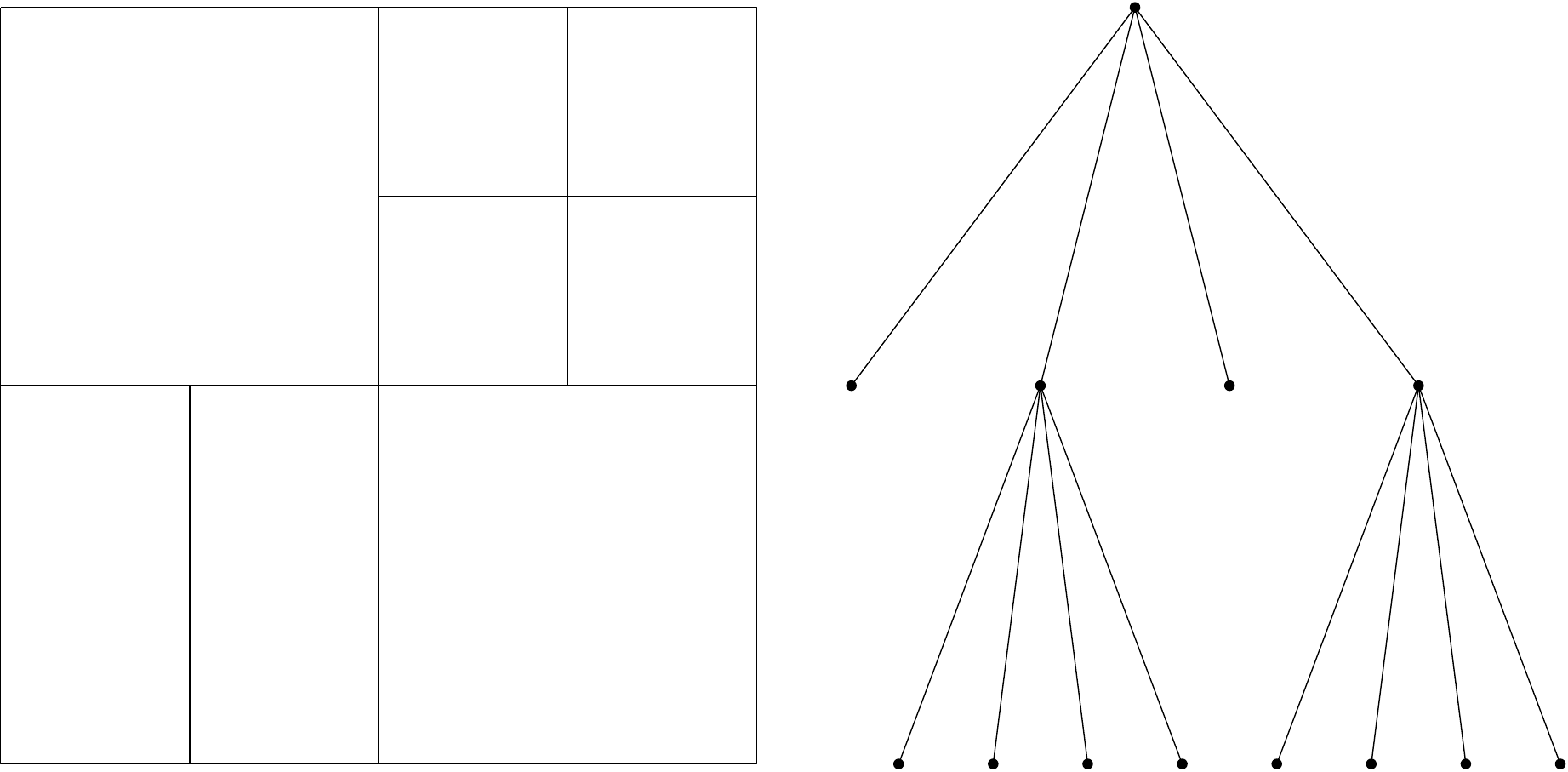}
		\caption{An MDP and a corresponding $2^n$-ary tree}
		\label{MDP_Tree}
	\end{center}
\end{figure}

Recall that $\mathcal{N}_\delta$ denotes the number of bars of length greater than $\delta$ in $\B(f)$, which is finite in all the cases we consider, see Subsection \ref{SubSec_Bar_Count}. Since we wish to use Mayer-Vietoris sequence for compact sets, in this section and Section \ref{Proof}, we consider $\mathcal{N}_\delta(f)$ to be defined using \v{C}ech homology of sublevel sets, i.e. $\mathcal{N}_\delta(f)=\mathcal{N}_\delta(\check{V}(f))$ in the notations from Section \ref{Section_Prelims}. This will not make a difference in the end result, see Proposition \ref{Sing_Cech}.

By a polynomial on a subset $U\subset \R^n$ we mean a restriction of a polynomial on $\R^n$ to $U.$ For a non-negative integer $k$, denote by $\P_k(U)$ the set of all real polynomial on $U$ of degree less than or equal to $k.$ Let $\S_k(\R^n)=\{ \sqrt{p}~|~ p\in \P_{2k}(\R^n),p\geq 0\}$ be the set of square roots of nonnegative polynomials of degree less than or equal to $2k.$ For a subset $U\subset \R^n$, denote by $\S_k(U)$ the set of restrictions of functions from $\S_k(\R^n)$ to $U.$

\begin{theorem} \label{MDP_Simplex_counting}
	Let $K$ be an MDP of $[0,1]^n$ and $f:[0,1]^n\rightarrow \R$ a continuous function. If for every $\sigma \in K$, $d_{C^0}(f|_\sigma, \P_k(\sigma))<\frac{\delta}{2}$ or $d_{C^0}(f|_\sigma, \S_k(\sigma))<\frac{\delta}{2}$,  then
	$$\mathcal{N}_{2^{n+1} \delta}  (f) \leq C_{n,k} |K|.$$
\end{theorem}

The proof of Theorem \ref{MDP_Simplex_counting} occupies the rest of the section. It has two main ingredients. The first one is a method of calculating $\mathcal{N}_\delta(f)$ from restrictions of $f$ to subsets covering its domain.  This method is explained in Subsection \ref{Subsection_Covers}. The second one is an estimate from above on  $\mathcal{N}_\delta$ of a polynomial on a box, as well as a square root of a polynomial on a box, see Proposition \ref{prop: barcode cube}. These two ingredients are combined using the stability theorem.

\begin{remark}
When considering barcodes in degree $0$ only, the proof of Theorem \ref{MDP_Simplex_counting} simplifies significantly, see Remark \ref{rmk: degree zero}.
\end{remark}

\begin{remark} Theorem \ref{MDP_Simplex_counting} can be considered a polynomial, multiscale version of the simplex counting method from \cite{CSEHM}, see also \cite{PRSZ}. To obtain the standard simplex counting one should set $k=0$ and notice that $d_{C^0}(f|_\sigma,\P_0(\sigma))=osc(f|_\sigma).$ To go from simplices to cubes, it is enough to divide a standard $n$-simplex into $n+1$ cubes by median hyperplanes, as we do in the proof of Proposition \ref{prop: cubulation}.
\end{remark}

\subsection{Barcode calculus on covers}\label{Subsection_Covers}
In this subsection, we work with barcodes of continuous functions on compact Hausdorff spaces. We wish to ensure that the corresponding persistence modules are moderate, so that results from Section \ref{sec: commutes} can be applied to them.  As explained in Section \ref{Section_Prelims}, in this situation all conditions in the definition of a moderate persistence module will be automatically satisfied, except for $q$-tameness.  To this end, we introduce the following notion.

\begin{dfn} Let $X$ be a Hausdorff topological space. A finite collection $\{A_i \}_{1\leq i \leq m}$ of compact subsets of $X$ is called tame if for every continuous function $f:X\rightarrow \R$ and any set of indices $1\leq i_1<\ldots <i_l \leq m$, $\check{V}(f|_{A_{i_1} \cap \ldots \cap A_{i_l}})$ is $q$-tame.
\end{dfn}

There are two examples of tame collections which will play important roles in the proofs of our main results. The first one is given by any finite collection of boxes in $\R^n.$ By a box we mean a product of closed intervals $[a_1,b_1]\times \ldots \times [a_n,b_n]\subset \R^n$ (here we allow also $a_i=b_i$). Persistence modules associated to continuous functions on boxes are $q$-tame, see Section 2. Hence, to see that such a collection is indeed tame, it is enough to notice that an intersection of boxes is again a box.

The second example is a collection of subsets of a manifold obtained as homeomorphic images of cubes from a fixed cubulation, see Proposition \ref{prop: cubulation}. Since two such subsets intersect along an image of face of a cube, all intersections will be homeomorphic to boxes and hence continuous functions on them will have $q$-tame persistence modules.

We wish to prove the following.
\begin{prop}\label{Barcode_Covers}
	Let $\{A_i \}_{1\leq i \leq m}$ be a tame collection of subsets of a Hausdorff topological space $X$ and $f:X\rightarrow \R$ a continuous function. Then $\check{V}(f|_{A_{i_1} \cup \ldots \cup A_{i_l}})$ is $q$-tame and for each $\delta > 0$, it holds
	$$\mathcal{N}_{2^m \delta} (f|_{A_1 \cup \ldots \cup A_m}) \leq \sum\limits_{1\leq i_1<\ldots <i_l \leq m} \mathcal{N}_ \delta (f|_{A_{i_1} \cap \ldots \cap A_{i_l}}).$$
\end{prop}
Using induction on $m$, one readily checks that Proposition \ref{Barcode_Covers} follows from the special case of two compact sets, i.e. $m=2.$ Thus, we are left to prove the following statement.
\begin{lemma}\label{Two_Sets}
	Let $\{A_1,A_2 \}$ be a tame collection of subsets of a Hausdorff topological space $X$ and $f:X\rightarrow \R$ a continuous function.  Then $\check{V}(f|_{A_1 \cup A_2})$ is $q$-tame and for each $\delta > 0$, it holds
	$$\mathcal{N}_{2 \delta} (f|_{A_1 \cup A_2}) \leq  \mathcal{N}_\delta (f|_{A_1}) + \mathcal{N}_\delta (f|_{A_2}) + \mathcal{N}_ \delta (f|_{A_1 \cap A_2}).$$
\end{lemma}
\begin{proof}
	Since $f$ is continuous, for every $t\in \R$, $\{f|_{A_1}\leq t \}$, $\{f|_{A_2}\leq t \}$ are compact and we may apply Mayer-Vietoris sequence to obtain a long exact sequence
	$$\ldots \rightarrow \check{H}_*(\{ f|_{A_1}\leq t \})\oplus \check{H}_*(\{ f|_{A_2}\leq t \})  \rightarrow \check{H}_* (\{ f|_{A_1 \cup A_2}\leq t \}) \rightarrow \check{H}_{*-1}(\{ f|_{A_1 \cap A_2}\leq t \}) \rightarrow \ldots$$
	Naturality of the Mayer-Vietoris sequence implies that in each degree $d$ there exists the following exact sequence of persistence modules
	$$\check{V}_d(f|_{A_1}) \oplus \check{V}_d(f|_{A_2}) \rightarrow \check{V}_d (f|_{A_1 \cup A_2}) \rightarrow \check{V}_{d-1} (f|_{A_1 \cap A_2}),$$
	which after summing over all degrees $d$ gives an exact sequence
	$$\check{V}(f|_{A_1}) \oplus \check{V}(f|_{A_2}) \rightarrow \check{V} (f|_{A_1 \cup A_2}) \rightarrow \check{V} (f|_{A_1 \cap A_2}).$$
	Thus, Lemma \ref{Middle_q-tame} implies that $\check{V}_d (f|_{A_1 \cup A_2})$ is $q$-tame and we may apply Theorem \ref{Exact_Pers} to obtain the desired inequality.\end{proof}

\begin{remark}\label{rmk: degree zero}
In the case where we consider barcodes in degree $0$ only, the proof of Proposition \ref{Barcode_Covers} becomes simpler and yields the following inequality with sharper dependence on $\delta:$ \begin{equation}\label{ineq: degree zero direct} \mathcal{N}_{0,\delta} (f|_{A_1 \cup \ldots \cup A_m}) \leq \sum\limits_{1\leq i \leq m} \mathcal{N}_{0,\delta} (f|_{A_{i}}).\end{equation} Indeed, for two sets, the relevant part of the Mayer-Vietoris sequence now takes the form:	$$\ldots \rightarrow \check{H}_0(\{ f|_{A_1}\leq t \})\oplus \check{H}_0(\{ f|_{A_2}\leq t \})  \rightarrow \check{H}_0 (\{ f|_{A_1 \cup A_2}\leq t \}) \rightarrow 0.$$ It now suffices to apply the monotonicity of the bar-counting function under surjections, see Proposition \ref{prop: Ndelta} and Lemma \ref{Q-tame_extension}.

Consequently, Theorem \ref{MDP_Simplex_counting} follows directly from Equation \eqref{ineq: degree zero direct}, Proposition \ref{prop: barcode cube}, and the stability theorem for barcodes. This bypasses the use of Lemmas \ref{Barcodes_Colored}, \ref{Nested_faces}, and \ref{Minimal_Covering_Lemma} below.

\end{remark}

By a {\it compact cover} we mean a family of compact subsets of a space whose union is the whole space. Let $\U=\{ U_i \}$ be a finite compact cover of a Hausdorff topological space $X$ and $f:X\rightarrow \R$ a continuous function. If $\U$ is tame, Proposition \ref{Barcode_Covers} gives the following estimate
$$
\mathcal{N}_{2^{|\U|}\delta} (f) \leq \sum\limits_{1\leq i_1<\ldots <i_l \leq |\U|} \mathcal{N}_ \delta (f|_{U_{i_1} \cap \ldots \cap U_{i_l}}). 
$$
Under certain assumptions, the coefficient $2^{|\U|}$ in this inequality can be improved. To this end, recall that a compact cover $\U$ is called {\it $m$-colorable} if it can be partitioned into $m$ subsets ({\it colors}) $\U_1, \ldots , \U_m \subset \U$ such that each $\U_i$ consists of disjoint sets.

\begin{lemma}\label{Barcodes_Colored}
	Assume that $\U$ is $m$-colorable and tame. For all $\delta > 0$ it holds
	$$\mathcal{N}_{2^m \delta} (f) \leq \sum\limits_{1\leq i_1<\ldots <i_l \leq |\U|} \mathcal{N}_ \delta (f|_{U_{i_1} \cap \ldots \cap U_{i_l}}). $$
\end{lemma}
\begin{proof}
	Let $\U_1,\ldots , \U_m \subset \U$ be a partitioning of $\U$ into $m$ colors. Denote by $A_i=\cup_{U\in \U_i} U$ for $1\leq i \leq m.$ $\{ A_i \}$ is a compact cover of $X.$ Since sets $A_i$ are unions of sets in a tame collection $\U$, Proposition \ref{Barcode_Covers} implies that $\{ A_i \}$ is also tame. We may now apply Proposition \ref{Barcode_Covers} again to obtain
	\begin{equation}\label{Cover_A}
		\mathcal{N}_{2^m \delta} (f) \leq \sum\limits_{1\leq i_1<\ldots <i_l \leq m} \mathcal{N}_ \delta (f|_{A_{i_1} \cap \ldots \cap A_{i_l}}). 
	\end{equation}
	We have that
	\begin{equation}\label{Intersection_A}
		A_{i_1} \cap \ldots \cap A_{i_l}=\bigcup_{(U_{j_1} , \ldots , U_{j_l}) \in \U_{i_1} \times \ldots \times \U_{i_l}} U_{j_1} \cap \ldots \cap U_{j_l},
	\end{equation}
	and due to the coloring condition, sets $U_{j_1} \cap \ldots \cap U_{j_l}$ for $(U_{j_1} , \ldots , U_{j_l}) \in \U_{i_1} \times \ldots \times \U_{i_l}$ are disjoint. Now notice that given two disjoint sets $X_1,X_2 \subset X$, it holds $\check{V}(f|_{X_1\cup X_2})=\check{V}(f|_{X_1})\oplus \check{V}(f|_{X_2})$ and thus 
\begin{equation}
\label{eq:bardisjoint}
\mathcal{N}_\delta (f|_{X_1\cup X_2})= \mathcal{N}_\delta (f|_{X_1}) + \mathcal{N}_\delta (f|_{X_2}).
\end{equation}
 This property combined with (\ref{Intersection_A}) gives us
	$$\mathcal{N}_\delta (f|_{A_{i_1} \cap \ldots \cap A_{i_l}})=\sum_{(U_{j_1} , \ldots , U_{j_l}) \in \U_{i_1} \times \ldots \times \U_{i_l}} \mathcal{N}_\delta( f|_{U_{j_1} \cap \ldots \cap U_{j_l}}),$$
	which together with (\ref{Cover_A}) proves the claim.
\end{proof}

\subsection{Barcode of a polynomial on a box}

By an {\it $n$-dimensional box} we mean a subset $Q\subset \R^n$ of the form $Q=[a_1,b_1]\times \ldots \times [a_n,b_n].$ For $0\leq i \leq n$, an {\it $i$-dimensional face} or an {\it $i$-face} of $Q$ is defined by setting $n-i$ coordinates in $Q$ to be equal to either $a_j$ or $b_j$, i.e. via conditions $(x_{j_1},\ldots , x_{j_{n-i}}) \in \{ a_{j_1},b_{j_1} \} \times \ldots \times \{a_{j_{n-i}},b_{j_{n-i}} \}$ and $(x_{j_{n-i+1}} , \ldots , x_{j_n})\in [a_{j_{n-i+1}},b_{j_{n-i+1}}]\times \ldots \times [a_{j_n},b_{j_n}].$ An {\it open $i$-dimensional face} is given via conditions $(x_{j_1},\ldots , x_{j_{n-i}}) \in \{ a_{j_1},b_{j_1} \} \times \ldots \times \{a_{j_{n-i}},b_{j_{n-i}} \}$ and $(x_{j_{n-i+1}} , \ldots , x_{j_n})\in (a_{j_{n-i+1}},b_{j_{n-i+1}})\times \ldots \times (a_{j_n},b_{j_n}).$ There are exactly $\binom{n}{i}2^{n-i}$ $i$-faces of an $n$-dimensional box. An $n$-dimensional cube is an $n$-dimensional box which satisfies $b_1-a_1=\ldots = b_n-a_n.$ 

We prove the following result which provides necessary bounds on the number of bars in the barcode of a polynomial or a square root of a polynomial on a box. 

\begin{prop}\label{prop: barcode cube}
	Let $Q\subset \R^n$ be an $n$-dimensional box and $p\in \P_k(Q)$ or $p\in \S_k(Q)$, $k \geq 1$. Then there exists a constant $C_n$ depending on $n$ only, such that for every $\delta > 0,$ \[ \cl N_{\delta}(p) \leq C_n k^n.\]
	Moreover, $\B(p)$ is finite and the total number of bars satisfies $\mathcal{N}_0(p)\leq C_n k^n.$
\end{prop}

\begin{remark}
In fact, we obtain the bound $ \cl N_{\delta}(p) \leq \frac{1}{2}(k+1)^n+\frac{1}{2}$ for $p\in \P_k(Q)$ and $ \cl N_{\delta}(p) \leq \frac{1}{2}(2k+1)^n+\frac{1}{2}$ for $p\in \S_k(Q).$
\end{remark}

\begin{proof}
Firstly, we notice that $\mathcal{N}_{\delta}(p) \leq C_n k^n$ for all $\delta>0$ implies the finiteness of $\B(p)$ with the desired bound. Indeed, due to upper semi-continuity of $\check{V}(p)$, there are no bars of length zero in $\B(p)$, see \cite{Schmahl22} for details.  Since the bound does not depend on $\delta$ the claim follows. Hence we are left to prove the inequality for a fixed $\delta>0$.
	
	Let us first prove the case $p\in \P_k(Q). $ Having fixed $\delta,$ consider a small perturbation $g$ of $p,$ satisfying $|p-g|_{C^0(Q)} < \delta/2,$ that is a Morse polynomial of degree at most $k$ on the box $Q$ in the sense of manifolds with corners \cite[Definitions 4,6]{Handron}. In particular, we can assume that it is Morse on every open $i$-dimensional face of $Q,$ for $0 \leq i \leq n$ and each of its critical points contributes at most one endpoint of a bar to the barcode of $g$ on $Q.$ This is a consequence of the first and second Morse theorems for manifolds with corners \cite[Theorems 7,8]{Handron}. Furthermore, $\cl N_{\delta}(p) \leq \cl N_{0}(g)$ by the stability theorem. Now the number of bars in the barcode of $g$ is bounded in terms of the total number $C(g,Q)$ of the critical points of its restrictions to the open $i$-dimensional faces of $Q,$ for $0 \leq i \leq n.$ Let $F$ be such an open $i$-dimensional face. Then $g|_F$ is identified with a Morse polynomial $h=h_{F^i}$ of degree at most $k$ on $F^i \subset \R^i$, $F^i$ being the interior of an $i$-dimensional box. The number $C(h,F^i)$ of critical points of $h$ is the number of common solutions of the $i$ polynomial equations $\del_1 h = 0, \ldots, \del_i h = 0,$ of degree at most $k-1.$ Furthermore the gradients of these polynomials are everywhere linearly independent on the common zero set. Therefore, by Milnor \cite[Lemma 1]{Milnor}, $C(h_{F^i},F^i) \leq (k-1)^i \leq (k-1)^n.$ Hence \[\mathcal{N}_0(g) \leq 1+ (C(g,Q)-1)/2,\] while \[C(g,Q) = \sum_{i=0}^n \sum_{F^i} C(h_{F^i},F^i) \leq \sum_{i=0}^n 2^{n-i} {n \choose i} (k-1)^i = (k+1)^n.\] This finishes the proof for $p\in \P_k(Q). $	
	
	To prove the case $p\in \S_k(Q)$, it is enough to notice that since $p\geq 0$, it holds $\{p\leq t\} = \{p^2 \leq t^2\}$ and hence $\B(p^2)=\{[a^2,b^2)~|~[a,b)\in \B(p)\}$, where $(+\infty)^2=+\infty$ by convention.  Now $\cl N_\delta(p)\leq \cl N_0(p)=\cl N_0(p^2)$ and since $p^2\in \P_{2k}(Q)$ the proof follows from the first case.
\end{proof}

\subsection{Proof of Theorem \ref{MDP_Simplex_counting}}

Let $K$ be an MDP of $[0,1]^n.$ For $0\leq i \leq n$, an $i$-face of a dyadic cube in $K$ is called {\it minimal} if it does not contain any other $i$-face of any other dyadic cube in $K.$ We denote by $K^{(i)}$ the union of all minimal faces of cubes in $K$ of dimension at most $i$ and call $K^{(i)}$ the {\it $i$-skeleton} of $K.$ This terminology comes from the fact that minimal faces constitute cells in the "obvious" CW-decomposition of $[0,1]^n$ induced by $K.$

We call an $l$-tuple $(\eta_1, \ldots , \eta_l)$ of minimal faces of cubes in $K$ {\it nested} if $\eta_1 \subset \ldots \subset \eta_l,$ the inclusions being strict. We will need the following lemma.
\begin{lemma}\label{Nested_faces}
There exists a constant $C_n$, which depends only on $n$, such that for every MDP of $[0,1]^n$, $K$, the total number of nested tuples does not exceed $C_n |K|.$
\end{lemma}
\begin{proof}
Every nested $l$-tuple $(\eta_1,\ldots , \eta_l)$ is a subtuple of a non-unique nested $(n+1)$-tuple. More precisely, there exists a non-unique $(n+1)$-tuple $(\nu_0,\ldots , \nu_n)$ such that $\eta_1=\nu_{i_1},\ldots , \eta_l=\nu_{i_l}$ for certain indices $0\leq i_1 <\ldots <i_l\leq n.$ Manifestly, every $\nu_i$ is an $i$-face of a dyadic cube. The total number of subtuples of a fixed $(n+1)$-tuple is $2^{n+1}$ and hence
\begin{equation}\label{NESTED_TUPLES_INEQ}
\# \text{ nested tuples} \leq 2^{n+1} \cdot (\# \text{ nested } (n+1) \text{-tuples}).
\end{equation}
To estimate the number of nested $(n+1)$-tuples $(\nu_0,\ldots, \nu_n)$ we first notice that the number of choices for $\nu_0$ is not greater than $2^n |K|$ because every dyadic cube has $2^n$ vertices. A chosen $\nu_0$ is contained in at most $2n$ minimal 1-faces and hence the number of pairs $\nu_0\subset \nu_1$ is at most $2^n |K|\cdot 2n.$ Similarly, if we have chosen $\nu_0\subset \nu_1 \subset \ldots \subset \nu_m$ the number of minimal $(m+1)$-faces which contain $\nu_m$ is at most $2(n-m).$ Thus
$$ \# \text{ nested } (n+1) \text{-tuples} \leq 2^n|K|\cdot 2^n \cdot n!,$$
which together with \eqref{NESTED_TUPLES_INEQ} finishes the proof.
\end{proof}

\begin{lemma}\label{Minimal_Covering_Lemma}
Let $K$ be an MDP of $[0,1]^n$ and $f:[0,1]^n\rightarrow \R$ a continuous function such that for every $\sigma \in K$, $d_{C^0}(f|_\sigma, \P_k(\sigma))<\frac{\delta}{2}$ or $d_{C^0}(f|_\sigma, \S_k(\sigma))<\frac{\delta}{2}$. There exists a compact cover $\U$ of $[0,1]^n$ which satisfies the following properties.
	\begin{enumerate}
		\item Sets in $\U$ are labelled by minimal faces in $K$, i.e. $\U=\{ U_\eta ~|~\eta \text{ a minimal face} \}$;
		\item Each $U_\eta$ is a box;
		\item $U_\eta\cap U_\nu \neq \emptyset \Leftrightarrow \eta \subset \nu$ or $\nu \subset \eta$;
		\item There exists a constant $C_{n,k}$ which depends only on $n$ and $k$ such that for each nested tuple $(\eta_1,\ldots , \eta_l)$ it holds $\mathcal{N}_\delta(f|_{U_{\eta_1}\cap \ldots \cap U_{\eta_l}})\leq C_{n,k}$.
	\end{enumerate}
\end{lemma}

We will first prove Theorem \ref{MDP_Simplex_counting} assuming Lemma \ref{Minimal_Covering_Lemma} and then prove Lemma \ref{Minimal_Covering_Lemma}.

\begin{proof}[Proof of Theorem \ref{MDP_Simplex_counting}]
	Let $\U$ be a compact cover of $[0,1]^n$ given by Lemma \ref{Minimal_Covering_Lemma}. 
	
	Now, by property (3) we have that for two minimal faces $\eta$ and $\nu$ of the same dimension it holds $U_\eta\cap U_\nu = \emptyset.$ Thus, sets
$$\U_i=\{ U_\eta ~|~\eta \text{ a minimal } i\text{-face} \},i=0,\ldots , n,$$
constitute a coloring of $\U$ by $n+1$ colors.  On the other hand, $U_{\eta_1} \cap \ldots \cap U_{\eta_l}\neq \emptyset$ implies that $\eta_1,\ldots , \eta_l$, in appropriate order, form a nested tuple, again due to property (3).  Moreover, since by (2) each set in $\U$ is a box, $\U$ is tame as explained in Subsection \ref{Subsection_Covers} and Lemma \ref{Barcodes_Colored} implies that
$$\mathcal{N}_{2^{n+1}\delta}(f)\leq \sum\limits_{(\eta_1,\ldots , \eta_l) \text{ nested}} \mathcal{N}_ \delta (f|_{U_{\eta_1} \cap \ldots \cap U_{\eta_l}}). $$
Property (4) gives us
$$\mathcal{N}_{2^{n+1}\delta}(f)\leq C_{n,k} \cdot (\text{the total number of nested tuples}),$$
which together with Lemma \ref{Nested_faces} proves Theorem \ref{MDP_Simplex_counting}. 
\end{proof}

\begin{proof}[Proof of Lemma \ref{Minimal_Covering_Lemma}]
	We will define $U_\eta$ as a box which approximates $\eta.$ More precisely, given $\varepsilon,\tau\geq 0$ and a minimal $m$-face $\eta= [a_{i_1},b_{i_1}]\times \ldots \times [a_{i_m},b_{i_m}] \times (x_{i_{m+1}},\ldots , x_{i_n}),$ we define an {\it $(\varepsilon,\tau)$-approximation of $\eta$} as
$$\eta^{(\varepsilon,\tau)}=[a_{i_1}+\varepsilon,b_{i_1}-\varepsilon]\times \ldots \times [a_{i_m}+\varepsilon,b_{i_m}-\varepsilon] \times [x_{i_{m+1}}-\tau, x_{i_{m+1}}+\tau]\times \ldots \times [x_{i_n}-\tau, x_{i_n}+\tau].$$
Our goal is to choose pairs $(\varepsilon_0,\tau_0),\ldots , (\varepsilon_n,\tau_n)$ in such a way that
$$\U=\bigcup_{i=0}^n \U_i,~\U_i=\{ U_\eta=\eta^{(\varepsilon_i,\tau_i)}\cap [0,1]^n ~|~\eta \text{ a minimal } i \text{-face} \}$$
satisfy (1)-(4). Manifestly, $\U$ satisfies properties (1) and (2) for any choice of $(\varepsilon_i,\tau_i).$ In order for sets in $\U$ to cover $[0,1]^n$ it is enough that
$$\tau_{n-1}<\ldots <\tau_1<\tau_0 \text{ and } \varepsilon_i<\tau_{i-1} \text{ for } i=1,\ldots , n.$$
Indeed, for any choice of $\tau_0$ sets in $\U_0$ cover the 0-skeleton $K^{(0)}.$ Condition $\varepsilon_1<\tau_0$ implies that sets in $\U_0 \cup \U_1$ cover the 1-skeleton $K^{(1)}.$ Similarly, $\varepsilon_i<\tau_{i-1}<\ldots <\tau_0$ implies that $\U_0\cup \ldots \cup \U_i$ covers the $i$-skeleton $K^{(i)}$ for all $0\leq i \leq n.$ Hence, $\U=\U_0 \cup \ldots \cup \U_n$ is a covering of $K^{(n)}=[0,1]^n.$ Figure \ref{Coloring} shows such a covering of $[0,1]^2$ with approximations of minimal faces colored in 3 colors.

\begin{figure}[ht]
	\begin{center}
		\includegraphics[scale=0.4]{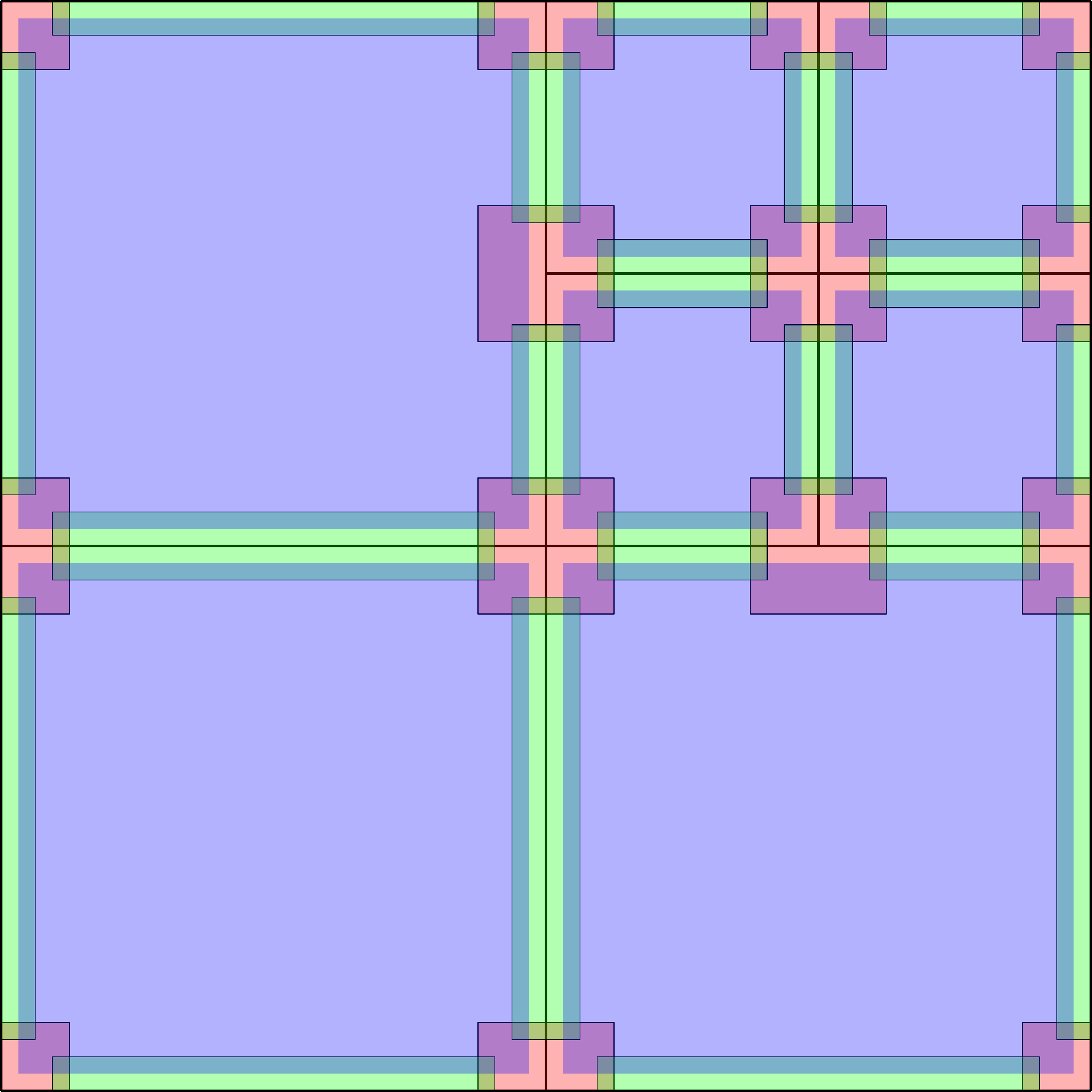}
		\caption{A cover of $[0,1]^2$ corresponding to an MDP}
		\label{Coloring}
	\end{center}
\end{figure}

What is left is to arrange for properties (3) and (4) to hold. To guarantee property (3) we choose $(\varepsilon_i,\tau_i)$ inductively in such a way that each $U_\nu\in \U_i$ intersects $U_\eta \in \U_0 \cup \ldots \cup \U_{i-1}$ if and only if $\eta \subset \nu$ and no two sets in $\U_i$ intersect. More precisely, we start by choosing $\tau_0$ small enough, so that sets in $\U_0$ are disjoint. Assume now that $(\varepsilon_0,\tau_0),\ldots , (\varepsilon_{i-1},\tau_{i-1})$ are given and let us choose $(\varepsilon_i,\tau_i).$ We first pick $\varepsilon_i$ to be an arbitrary number which satisfies $0<\varepsilon_i<\tau_{i-1}<\ldots < \tau_0.$ Notice that for each minimal $i$-face $\nu$ and all $U_\eta \in \U_0\cup \ldots \cup \U_{i-1}$ it holds
$$\nu^{(\varepsilon_i,0)}\cap U_\eta\neq \emptyset \text{ if and only if } \eta\subset \nu.$$
Since all the above sets are compact, for small enough $\tau_i'$ we have that still for each minimal $i$-face $\nu$ and all $U_\eta \in \U_0\cup \ldots \cup \U_{i-1}$ it holds
$$\nu^{(\varepsilon_i,\tau_i')}\cap U_\eta\neq \emptyset \text{ if and only if } \eta\subset \nu.$$
Similarly, notice that for any pair of minimal $i$-faces $(\nu_1,\nu_2)$ we have that $\nu_1^{(\varepsilon_i,0)} \cap \nu_2^{(\varepsilon_i,0)}=\emptyset$ and hence for small enough $\tau_i''$, $\nu_1^{(\varepsilon_i,\tau_i'')} \cap \nu_2^{(\varepsilon_i,\tau_i'')} =\emptyset$ holds as well. Taking $\tau_i<\min(\tau_i',\tau_i'')$ guarantees property (3).

Finally to arrange for property (4) to hold,notice that for a nested tuple $(\eta_1,\ldots,\eta_l)$, set $U_{\eta_1}\cap \ldots \cap U_{\eta_l}$ belongs to the $\tau_0$-neighbourhood of $\eta_1.$ By the assumption of Theorem \ref{MDP_Simplex_counting},  $d_{C^0}(f|_{\eta_1}, \P_k(\eta_1))<\frac{\delta}{2}$ or $d_{C^0}(f|_{\eta_1}, \S_k(\eta_1))<\frac{\delta}{2}$ and hence for small enough $\tau_0$ we have that
\begin{equation}\label{CZERO_FACE}
d_{C^0}(f|_{U_{\eta_1}\cap \ldots \cap U_{\eta_l}}, \P_k(U_{\eta_1}\cap \ldots \cap U_{\eta_l}))< \frac{\delta}{2},
\end{equation}
or
\begin{equation}\label{CZERO_FACE2}
d_{C^0}(f|_{U_{\eta_1}\cap \ldots \cap U_{\eta_l}}, \S_k(U_{\eta_1}\cap \ldots \cap U_{\eta_l}))<\frac{\delta}{2}.
\end{equation}
On the other hand, since $U_{\eta_1},\ldots , U_{\eta_l}$ are boxes, $U_{\eta_1}\cap \ldots \cap U_{\eta_l}$ is a box as well and hence Proposition \ref{prop: barcode cube} implies that $\mathcal{N}_{\delta'}(p)\leq C_{n,k}$ for any $\delta'>0$ and any $p\in \P_k (U_{\eta_1} \cap \ldots \cap U_{\eta_l})$ or $p\in \S_k (U_{\eta_1} \cap \ldots \cap U_{\eta_l})$. We choose $\delta'$ and $p$ such that
$$d_{C^0}(f|_{U_{\eta_1}\cap \ldots \cap U_{\eta_l}},p)<\frac{\delta}{2}-\frac{\delta'}{2}.$$
This inequality together with the stability theorem implies
$$\mathcal{N}_\delta ( f|_{U_{\eta_1}\cap \ldots \cap U_{\eta_l}})\leq \mathcal{N}_{\delta'}(p)\leq C_{n,k}.$$
Taking $\tau_0$ (and hence also all $\varepsilon_i,\tau_i$) small enough so that (\ref{CZERO_FACE}) or (\ref{CZERO_FACE2}) hold for all nested tuples of minimal faces guarantees property (4) and finishes the proof.
\end{proof}

\section{The proof of the main result}\label{Proof}

The goal of this section is to prove Theorem \ref{thm: main 2}. We first present a few preliminaries on Sobolev spaces, then we prove the local result on a cube, and finally prove the general case.

\subsection{Sobolev spaces}
\label{subsec:sobolev}
The goal of this subsection is to fix the definitions and notation for Sobolev norms that are used throughout the paper.  First, let $\Omega$ be a domain in 
$\mathbb{R}^n$. Given an integer $k\ge 0$ and a real number $p\ge 1$, we define a Sobolev space $W^{k,p}(\Omega)$ as the closure of $C^\infty(\Omega)$ with respect to the norm
\begin{equation}
	\label{eq:sobolevnorm}
	\| f\|_{W^{k,p}(\Omega)} = \left( \sum_{|\alpha| \le k} \int_\Omega \left|D^\alpha f(x)\right|^p dx \right)^{\frac{1}{p}},
\end{equation}
where the sum is taken over all multi-indices $\alpha=(\alpha_1, \dots, \alpha_n)$, $\alpha_i \in \mathbb{Z}_{\ge 0}$, such that $\alpha_1+\dots +\alpha_n \le k$, and $D^\alpha=D_{x_1}^{\alpha_1}\dots D_{x_n}^{\alpha_n}$ denotes the partial derivatives.  Similarly, the space $W_0^{k,p}(\Omega)$ is defined as the completion of the space $C_0^\infty(\Omega)$ of smooth functions with compact support with respect to the norm \eqref{eq:sobolevnorm}.

The notion of the Sobolev space together with the norm \eqref{eq:sobolevnorm} can be extended to functions on compact Riemannian manifolds and to sections of vector bundles. There exist several ways to do it yielding equivalent Sobolev norms. In the present paper we use the definition  via the partition of unity (see, for instance,  \cite[Appendix 1]{Shubin92}), and we briefly recall this construction.

Consider a finite atlas $\mathcal V = \{ (V_i,\phi_i) \}$ for a compact Riemannian manifold $M$ where $V _i \subset \R^n$ is an open set and $\phi_i:V_i \to U_i \subset M$ is a diffeomorphism and let $\{ \chi_i  \}$ be a subordinate partition of unity. Set $K_i = \mrm{supp}(\chi_i) \subset U_i.$ Then for $f \in C^{\infty}(M)$ we set $f_i = \chi_i f$ and define 
$$||f||_{W^{k,p}(M)} =  \left(\sum ||f_i \circ \phi_i||^p_{W^{k,p}(K_i)}\right)^{1/p}$$

Note that the  norm depends on the choice of the atlas and the partition of unity, however its equivalence class does not. 
This  definition extends in a straightforward way to sections of a vector bundle $E \to M$ with an inner product.

For functions on Euclidean domains $\Omega\subset \R^n$ we will also use the notation
\begin{equation}
	\label{eq:sobolevk}
	\|D^k f\|_{L^p(\Omega)} = \left( \sum_{|\alpha|=k} \int_\Omega \left|D^\alpha f(x)\right|^p dx \right)^{\frac{1}{p}}.
\end{equation}

This generalizes as follows to vector-valued functions. Given a positive integer $k$ and $s:\Omega\rightarrow \R^l$, $s=(f_1,\ldots , f_l)$, we denote
$$\| D^k s \|_{L^p}= \Bigg( \sum_{|\alpha|=k} \int_{\Omega} \Big(\sum_{i=1}^l |D^\alpha f_i(x) |^2 \Big)^{p/2} dx  \Bigg)^\frac{1}{p}.$$

\subsection{The case of a cube}

Recall that $\mathcal{N}_\delta(|s|)$ denotes the number of bars of length greater than $\delta$ in $\B(|s|)$ defined using \v{C}ech homology.  The following is the main analytic ingredient of the proof. 

\begin{prop}\label{MDP_Construction} Let $n,l$ and $k$ be positive integers and $p\geq 1$ a real number such that $kp>n.$ There exist a constant $\Cnkp$, which depends on $n,k,p$, such that for each smooth map $s:[0,1]^n\rightarrow \R^l $ and for all $\delta>0$ there exists an MDP of $[0,1]^n$, $K$, such that
\begin{enumerate}
\item $(\forall \sigma \in K)~ d_{C^0}(|s|_\sigma |, \S_{k-1})<\frac{\delta}{2}$ 
\item $|K| \leq 1 + \Cnkp \Big( \frac{\| D^k s \|_{L^p}}{\delta} \Big)^\frac{n}{k}.$
\end{enumerate}

\end{prop}

As an immediate corollary of Proposition \ref{MDP_Construction}, we obtain the local version of our main result, Theorem \ref{thm: main 2}.
\begin{theorem}\label{CUBE} Under the assumptions of Proposition \ref{MDP_Construction} it holds
$$\mathcal{N}_\delta (|s|) \leq C_{n,k} + \Cnkp \Bigg( \frac{\| D^k s \|_{L^p}}{\delta} \Bigg)^\frac{n}{k},$$
for certain constants $C_{n,k}$ and $\Cnkp$ which respectively depend on $n,k$ and $n,k,p.$
\end{theorem}	
\begin{proof}
Let $K$ be an MDP given by Proposition \ref{MDP_Construction}. Property (1) allows us to apply Theorem \ref{MDP_Simplex_counting} which together with property (2) proves the theorem. 
\end{proof}

The proof of Proposition \ref{MDP_Construction} occupies the rest of the subsection. Our goal will be to construct $K$ using a subdivision algorithm with a criterion for subdividing a dyadic cube $\sigma$ based on a Morrey-Sobolev type estimate for $d_{C^0}(|s|_\sigma|, \S_{k-1}(\sigma)).$ We first recall the relevant estimate. For a subset $Q \subset \R^n$ let $\cl{P}^l_{k-1}(Q)$ denote the space of mappings $s: Q \to \R^l$ all of whose coordinates are polynomials of degree at most $k-1.$ Endow $\R^l$ with the standard Euclidean metric.

\begin{theorem}[Morrey-Sobolev] \label{Morrey-Sobolev} Let $n,k$ be positive integers and $p\geq 1$ a real number such that $kp>n.$ There exists a constant $C_{n,k,p}'$ which depends on $n,k,p$ such that for every $n$-dimensional cube $Q\subset \R^n$ and every smooth function $s:Q\rightarrow \R^l$ it holds
	$$d_{C^0}(s,\P^l_{k-1}(Q))\leq C_{n,k,p}' (\Vol Q)^{\frac{k}{n}-\frac{1}{p}} \| D^k s\|_{L^p}.$$
\end{theorem}
We include a proof of Theorem \ref{Morrey-Sobolev} following \cite{DS80} in Appendix \ref{app: Morrey-Sobolev}. As an immediate corollary of Theorem \ref{Morrey-Sobolev}, we obtain that for every smooth $s:Q\rightarrow \R^l$ 
\begin{equation}\label{Criterion_Estimate}
d_{C^0}(|s|,\S_{k-1}(Q))\leq C_{n,k,p}' (\Vol Q)^{\frac{k}{n}-\frac{1}{p}} \| D^k s\|_{L^p}.
\end{equation}
Indeed, if $s=(f_1,\ldots,f_l)$ is approximated by $\tilde{s}=(p_1,\ldots, p_l) \in \cl P^l_{k-1}(Q)$ via Theorem \ref{Morrey-Sobolev}, we obtain 
$$||s|-|\tilde{s}||\leq |s-\tilde{s}|=\Big(\sum_{i=1}^l (f_i - p_i)^2\Big)^\frac{1}{2}\leq C_{n,k,p}' (\Vol Q)^{\frac{k}{n}-\frac{1}{p}} \| D^k s\|_{L^p}.$$

Let us now fix positive integers $n,l,k$, a real $p\geq 1$ such that $kp>n$, $\delta >0$ and a smooth map $s:[0,1]^n\rightarrow \R^l.$ We call a cube $Q\subset [0,1]^n$ {\it bad} if
$$(\Vol Q)^{\frac{k}{n}-\frac{1}{p}} \| D^k (s|_Q)\|_{L^p}\geq \frac{\delta}{2 \Cnkp'},$$
and otherwise we call it {\it good}. Notice that by \eqref{Criterion_Estimate} if $Q$ is good then \[d_{C^0}(|s|_Q|, \S_{k-1}(Q))<\frac{\delta}{2}.\] We will need the following lemma.
\begin{lemma} \label{Bad_cubes}
	Let $K$ be an MDP of $[0,1]^n$ and assume that $\sigma_1,\ldots, \sigma_N \in K$ are bad. Denote $B=\cup_{i=1}^N \sigma_i.$ It holds
	$$N\leq (2\Cnkp')^\frac{n}{k} (\Vol B)^{1-\frac{n}{kp}} \Bigg( \frac{\| D^k (s|_B) \|_{L^p}}{\delta} \Bigg)^\frac{n}{k}.$$
\end{lemma}

\begin{proof}
	Since all $\sigma_i$ are bad we have that for $i=1,\ldots,  N$ it holds
	$$(\Vol \sigma_i)^{\frac{1}{p}-\frac{k}{n}} \leq 2\Cnkp' \frac{\| D^k (s|_{\sigma_i}) \|_{L^p}}{\delta}.$$
	Raising both sides of the inequality to the power $p$ and summing over $i$ gives us
	\begin{equation}\label{Sum_of_bad}
		\sum_{i=1}^N (\Vol \sigma_i)^{1-\frac{kp}{n}} \leq (2\Cnkp')^p \Bigg( \frac{\| D^k (s|_B) \|_{L^p}}{\delta} \Bigg)^p.
	\end{equation}
	One may check that if $\alpha<0$, $x_1,\ldots, x_N>0$ and $\sum_{i=1}^N x_i$ is fixed, $\sum_{i=1}^N x_i^\alpha$ attains minimum when all $x_i$ are equal. Thus, $1-\frac{kp}{n}<0$ and $\sum_{i=1}^N \Vol \sigma_i=\Vol B$ imply that
	$$\sum_{i=1}^N \Bigg( \frac{\Vol B}{N} \Bigg)^{1-\frac{kp}{n}} \leq \sum_{i=1}^N (\Vol \sigma_i)^{1-\frac{kp}{n}},$$
	which together with (\ref{Sum_of_bad}) yields
	$$N^\frac{kp}{n} (\Vol B)^{1-\frac{kp}{n}} \leq (2\Cnkp')^p \Bigg( \frac{\| D^k (s|_B) \|_{L^p}}{\delta} \Bigg)^p .$$
	Raising both sides of this inequality to the power $\frac{n}{kp}$ finishes the proof.
\end{proof}
We now have all the ingredients necessary to prove Proposition \ref{MDP_Construction}.
\begin{proof}[Proof of Proposition \ref{MDP_Construction}]
	We assume that $\| D^k s \|_{L^p}\neq 0.$ Otherwise $s=(f_1,\ldots ,f_l)$ with $f_i\in \P_{k-1}([0,1]^n)$,  which implies $|s|\in \S_{k-1}([0,1]^n)$ and hence Proposition \ref{MDP_Construction} holds for $K=\{ [0,1]^n \}$.  By convention, we consider any sum from 0 to -1 to be equal to zero.
	
	Let $K_0=\{[0,1]^n \}.$ We construct a finite sequence $K_l$ of MDPs of $[0,1]^n$ inductively, according to the following algorithm. If all $\sigma\in K_l$ are good the algorithm stops. If not, we subdivide all bad dyadic cubes in $K_l$ into $2^n$ smaller dyadic cubes. $K_{l+1}$ is the MDP obtained as a result of this subdivision, see Figure \ref{Subdivision}.
	\begin{figure}[ht]
		\begin{center}
			\includegraphics[scale=0.6]{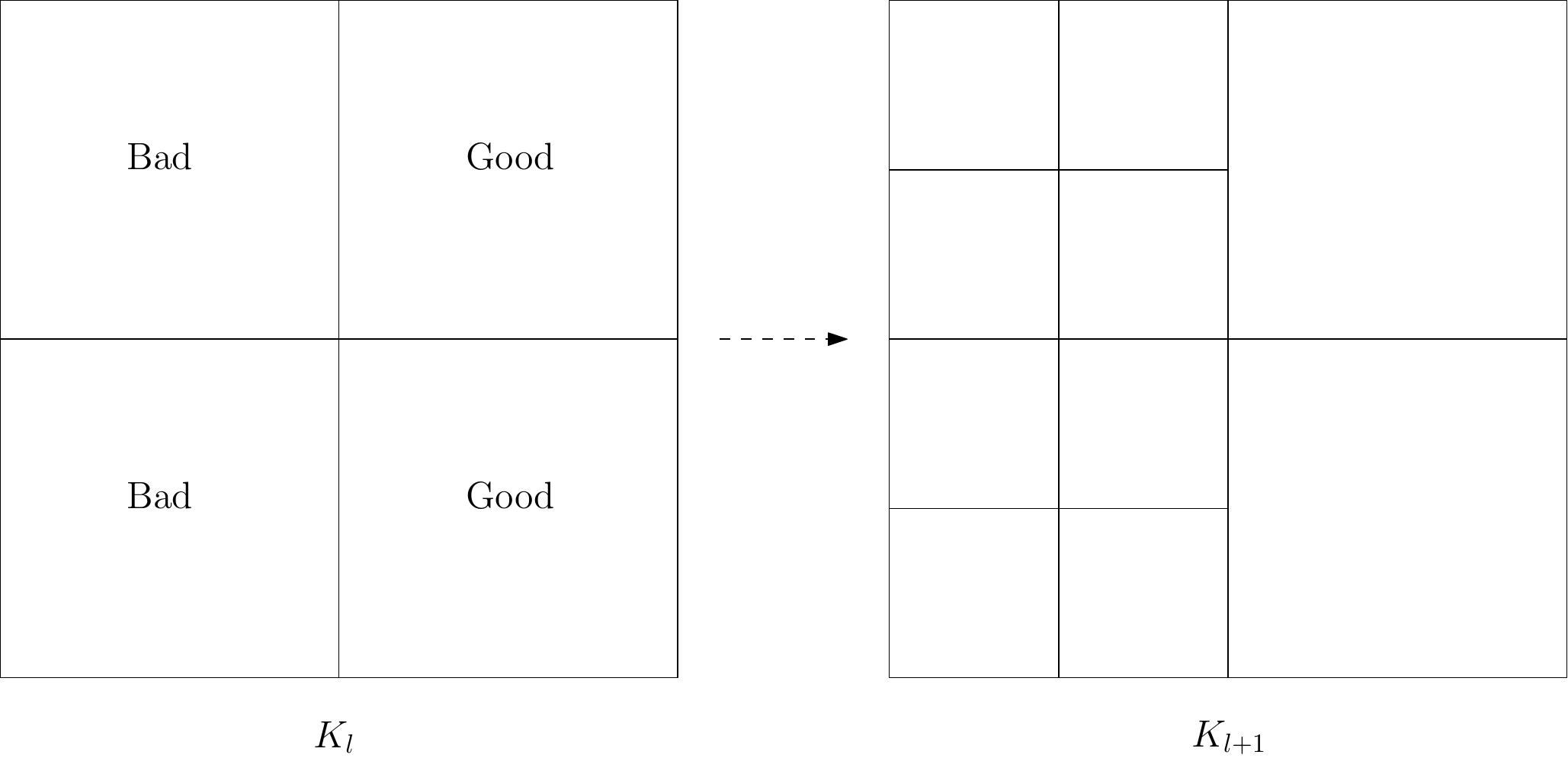}
			\caption{A step in the subdivision algorithm}
			\label{Subdivision}
		\end{center}
	\end{figure}
	
	Since $\frac{\delta}{2 \Cnkp'}$ is fixed and $\frac{k}{n}-\frac{1}{p}>0$ the algorithm stops after a finite number of steps $l_0.$ We define $K=K_{l_0}$ and proceed to prove that $K_{l_0}$ satisfies properties (1) and (2) of the proposition. All $\sigma \in K_{l_0}$ are good and hence by \eqref{Criterion_Estimate}
	$$(\forall \sigma \in K_{l_0})~d_{C^0}(|s|_\sigma|, \S_{k-1}(\sigma))<\frac{\delta}{2}.$$

	This proves property (1). 
	
	What is left to do is to estimate $|K_{l_0}|.$ Denote by $B_l\subset K_l$ the set of all bad dyadic cubes in $K_l.$ By construction
	\begin{equation}\label{INEQ_Bad}
		|K_{l_0}|\leq 1+2^n \sum_{l=0}^{l_0-1} |B_l|.
	\end{equation}
	We use two competing estimates for $|B_l|.$
	\\
	\\
	\noindent {\bf Estimate 1:} $|B_l|\leq 2^{nl}.$
	\\
	\\
	\noindent {\bf Estimate 2:} $|B_l|\leq \Cnkp'' 2^{-l(kp-n)} \Bigg( \frac{\| D^k s \|_{L^p}}{\delta} \Bigg)^p$, where $\Cnkp''=(2 \Cnkp')^p.$
	\\
	\\
	\noindent Estimate 1 follows from the construction since $|B_{l+1}|\leq 2^n |B_l|.$ Estimate 2 follows from Lemma \ref{Bad_cubes}. Indeed, this lemma gives us\footnote{We slightly abuse the notation and use $B_l$ both for the set of bad cubes and for $\cup_{\sigma \in B_l} \sigma.$}
	$$|B_l|\leq (2\Cnkp')^\frac{n}{k} (\Vol B_l)^{1-\frac{n}{kp}} \Bigg( \frac{\| D^k (s|_{B_l}) \|_{L^p}}{\delta} \Bigg)^\frac{n}{k}.$$
	Substituting $\Vol B_l=2^{-nl}|B_l|$ into this inequality and using $\| D^k (s|_{B_l}) \|_{L^p} \leq \| D^k s \|_{L^p}$ yields Estimate 2.
	
	To complete the proof, notice that Estimate 1 gets worse, while Estimate 2 improves as $l$ grows. Hence, there exists an optimal value, $l_{opt}$, starting from which Estimate 2 becomes better than Estimate 1. This $l_{opt}$ can be computed from the two conditions
	$$2^{nl_{opt}}\leq \Cnkp'' 2^{-l_{opt}(kp-n)} \Bigg( \frac{\| D^k s \|_{L^p}}{\delta} \Bigg)^p \textrm{ and } \Cnkp'' 2^{-(l_{opt}+1)(kp-n)} \Bigg( \frac{\| D^k s \|_{L^p}}{\delta} \Bigg)^p < 2^{n(l_{opt}+1)},$$
	which are equivalent to
	\begin{equation}\label{LOPT}
		2^{l_{opt}}\leq \Cnkp'''\Bigg( \frac{\| D^k s \|_{L^p}}{\delta} \Bigg)^\frac{1}{k} <2^{l_{opt}+1},
	\end{equation}
	where $\Cnkp'''=(\Cnkp'')^\frac{1}{kp}=(2 \Cnkp')^\frac{1}{k}.$ In case $l_{opt}<0$, i.e. $\Cnkp''' \big( \frac{\| D^k s \|_{L^p}}{\delta} \big)^\frac{1}{k}<1$, we set $l_{opt}=-1.$ Applying Estimates 1 and 2 to (\ref{INEQ_Bad}) yields
	\begin{equation}\label{ZERO}
		|K_{l_0}|\leq 1+2^n \sum_{j=0}^{l_{opt}} 2^{nj} + 2^n \sum_{j=l_{opt}+1}^\infty \Cnkp'' 2^{-j(kp-n)}\Bigg(\frac{\| D^k s \|_{L^p}}{\delta} \Bigg)^p.
	\end{equation}
	First inequality in (\ref{LOPT}) gives us (when $l_{opt}\geq 0$)
	\begin{equation}\label{ONE}
		\sum_{j=0}^{l_{opt}} 2^{nj}=\frac{1}{2^n-1} (2^{n(l_{opt}+1)}-1) \leq \frac{1}{2^n-1} \Bigg((2\Cnkp''')^n \Bigg( \frac{\| D^k s \|_{L^p}}{\delta} \Bigg)^\frac{n}{k} -1\Bigg) .
	\end{equation}
	On the other hand, since $kp-n>0$, we have
	$$\sum_{j=l_{opt}+1}^\infty \Cnkp'' 2^{-j(kp-n)}\Bigg(\frac{\| D^k s \|_{L^p}}{\delta} \Bigg)^p = \frac{\Cnkp''}{1-2^{n-kp}} 2^{(n-kp)(l_{opt}+1)}\Bigg(\frac{\| D^k s \|_{L^p}}{\delta} \Bigg)^p.$$
	Second inequality in (\ref{LOPT}) gives us\footnote{Here we use the assumption that $\| D^k s \|_{L^p}\neq 0.$}
	$$2^{(n-kp)(l_{opt}+1)}< (\Cnkp''')^{n-kp} \Bigg(\frac{\| D^k s \|_{L^p}}{\delta} \Bigg)^{\frac{n}{k}-p},$$
	and hence
	\begin{equation}\label{TWO}
		\sum_{j=l_{opt}+1}^\infty \Cnkp'' 2^{-j(kp-n)}\Bigg(\frac{\| D^k s \|_{L^p}}{\delta} \Bigg)^p < \frac{\Cnkp''(\Cnkp''')^{n-kp}}{1-2^{n-kp}} \Bigg(\frac{\| D^k s \|_{L^p}}{\delta} \Bigg)^{\frac{n}{k}} .
	\end{equation}
	Substituting (\ref{ONE}) and (\ref{TWO}) in (\ref{ZERO}) finishes the proof.
\end{proof}

\subsection{General case}\label{Section_Manifold}

In this subsection we prove Theorem \ref{thm: main 2} using Proposition \ref{MDP_Construction}.  We start with a consequence of a theorem of Whitney \cite[Section IV.12, Theorem 12A]{Whitney-book} regarding triangulations of manifolds. 

\begin{prop}\label{prop: cubulation}
Let $M$ be a compact manifold of dimension $n.$ There exists a finite collection of smooth embeddings $\theta_i: Q \to M,$ $1 \leq i \leq N$ where $Q=[0,1]^n$ is the standard cube in $\R^n,$ with the following properties: \begin{enumerate}
\item $ M =  \bigcup_{1 \leq i \leq N} \theta_i({Q})$
\item $\{\theta_i(\odot{Q})\}_{1 \leq i \leq N}$ are disjoint, where $\odot{Q} = (0,1)^n,$
\item Given $1 \leq i_1 < \ldots < i_l \leq N,$  $I_{i_1,\ldots,i_l} = \bigcap_{1 \leq j \leq l} \theta_{i_j}(Q) \neq \emptyset$ implies that there exist faces $F_{i_1},\ldots, F_{i_l}$ of $Q$ of the same dimension such that $I_{i_1,\ldots,i_l} = \theta_{i_j}(F_{i_j})$ for all $j.$ 
\item \label{property:affine} For all $1 \leq j < j' \leq l,$ $(\theta_{i_{j'}}|_{I_{i_1,\ldots,i_l}})^{-1} \circ \theta_{i_j}: F_{i_j} \to F_{i_{j'}}$ is an affine diffeomorphism of cubes.
\end{enumerate}
\end{prop}

\begin{proof}
Indeed Whitney's theorem produces a triangulation with similar properties, that is $g_j: \Delta^n \to M$ where $\Delta^n = \{(t_0,\ldots, t_n)\;|\; \sum t_j = 1, t_j \geq 0\}$ is the standard $n$-simplex, satisfying the properties above with $Q$ replaced by $\Delta$ and $\odot{Q}$ replaced by $\odot{\Delta}^n = \{(t_0,\ldots, t_n)\;|\; \sum t_j = 1, t_j > 0\}.$ It remains to divide the $n$-simplex into $(n+1)$ topological cubes $\{ Q_0,\ldots, Q_n \},$ where $Q_j = \{t_j \geq t_i\;, i\neq j \} \subset \Delta^n.$ Note that $Q_j$ is parametrized by $Q$ as follows: $\phi_j = (\psi_j)^{-1}: Q \to Q_j,$ where $\psi_j: Q_j \to Q$ is $\psi_k(t_0,\ldots, t_n) = \pi_{j}((t_0,\ldots,t_n)/t_j),$ where $\pi_j$ is the projection to the coordinate plane $H_j = \{t_j = 0\}$ composed with an evident isomorphism $H_j \to \R^{[n]\setminus \{j\}},$ where $[n] = \{0,1,\ldots,n\},$ which we further identify $\R^{[n]\setminus \{j\}}$ with $\R^n$ artificially by listing coordinates in increasing order. However, it is convenient to work directly in $\R^{[n]\setminus \{j\}}$ and the cube $Q^{(j)} =\{(x_k)_{k \neq j}\;|\; 0 \leq x_k \leq 1\}.$ Note that $\phi_j((x_k)_{k \neq j}) = (t_0,\ldots, t_n)$ where $t_j = \frac{1}{1+\sum {x_k}}$ and $t_i = \frac{x_i}{1+\sum {x_k}}$ for $i \neq j.$

We claim that the resulting maps $\theta_{jk} = g_j \circ \phi_k: Q \to M$ suitably reindexed satisfy the required properties. Indeed, in view of the analogue of property \eqref{property:affine} from Whitney's theorem and the definition of $Q_j$, it is enough to check that the intersection condition holds for the $Q_j$ themselves. This is a direct verification, which we illustrate in the case $l=2.$ In this case, for $i< j,$ $J_{ij} = \phi_i(Q) \cap \phi_j(Q)$ satisfies $J_{ij} = \phi_i(F_i) = \phi_j(F_j),$ $F_i = \{x_j = 1\},$  $F_j = \{x_i = 1\}.$ Moreover $(\phi_j|_{J_{ij}})^{-1} \circ \phi_i: F_i \to F_j$ sends the vector $(x_k)_{k \neq i}, x_j = 1$ to the vector $(x'_k)_{k \neq j}, x'_i=1$ where $x'_k = x_k,$ for $k \notin \{i,j\},$ which is an affine isomorphism of the cubes $F_i, F_j.$
\end{proof}

We will use the following auxiliary result.

\begin{lemma}\label{lma: replace C2}
Suppose that Theorem \ref{thm: main 2} holds with $C'_1, C'_2$ both depending on $M, E, k, p$ instead of $C_1, C_2$ as in its formulation. Then there exists $C_1$ depending on $M, E, k, p$ such that Theorem \ref{thm: main 2} holds with $C_2 = \dim H_*(M)$ as stated.
\end{lemma}

\begin{proof}[Proof of Lemma \ref{lma: replace C2}]
We have \[ \cl N_{\delta}(|s|) \leq \frac{C'_1}{\delta^{n/k}}
||s||_{W^{k,p}}^{n/k}+C'_2,\] $C'_1, C'_2$ depending only on $M, E, k, p.$ By the Sobolev inequality \[ ||s||_{L^\infty} = \max |s| \leq C ||s||_{W^{k,p}},\] where $C$ depends on $M, k, p$ only. In particular if $||s||_{W^{k,p}} \leq \delta/C$ then $\cl N_{\delta}(|s|) = C_2 = \dim H_*(M),$ since only the infinite bars would contribute to $\cl N_{\delta}(|s|).$ If $||s||_{W^{k,p}} \geq \delta/C,$ setting $C_1 = C'_1 + C^{n/k} C'_2$ we have \[ \frac{C_1}{\delta^{n/k}}
||s||_{W^{k,p}}^{n/k} \geq \frac{C'_1}{\delta^{n/k}} ||s||_{W^{k,p}}^{n/k} + C^{n/k} C'_2 C^{-n/k} \geq \cl N_{\delta}(|s|).\] This finishes the proof.
\end{proof}

\begin{proof}[Proof of Theorem \ref{thm: main 2}] Without loss of generality, we may assume that $s$ is a smooth section. Let $l$ be the rank of $E.$  For $\{\theta_i\}_{1 \leq i \leq N}$ from Proposition \ref{prop: cubulation}, consider orthogonal trivializations $\Psi_i:\theta_i^*E\rightarrow Q\times \R^l.$ Viewing $s\circ \theta_i$ as a section of $\theta_i^*E$, we have that $s_i:=\Psi_i\circ s\circ \theta_i:Q\rightarrow Q\times \R^l$ is a section of a trivial bundle which we identify with a map $s_i:Q\rightarrow \R^l.$

Proposition \ref{MDP_Construction} shows that for all $1 \leq i \leq N$ there is an MDP $K_i$ of $Q$ such that for all $\sigma \in K_i,$ $d_{C^0}(|s_i|_{\sigma}|, \cl S_{k-1}(\sigma)) < \delta/2$ and
\[|K_i| \leq C_i(s,\delta) = 1 + \Cnkp \left( \frac{||D^k s_i||_{L^p}}{\delta} \right)^{\frac{n}{k}}.\]
Consider a face $F$ of $Q$ of dimension $m.$ It can be canonically identified with $[0,1]^m$ and the MDP $K_i$ induces an MDP $K_i^F$ of $[0,1]^m$ such that $|K_i^F| \leq |K_i|$ and still for all $\sigma \in K_i^F,$  $d_{C^0}(|s_i|_{\sigma}|, \cl S_{k-1}(\sigma)) < \delta/2.$ Theorem \ref{MDP_Simplex_counting} implies that
\begin{equation}\label{eq: N face} \cl N_{2^{n+1} \delta}(|s|_{\theta_i(F)}|) =\cl N_{2^{n+1} \delta}(|s_i|_F|) \leq C_{n,k} C_i(s,\delta),\end{equation} 
for all $i$ and every face $F$ of $Q$ of dimension $0\leq m \leq n.$ Set $C(s,\delta) = \max_{1 \leq i \leq N} C_i(s,\delta).$ Note that
\begin{equation}\label{eq: cube to manifold bound} C(s,\delta) \leq 1 + C_{M,E,k,p} \left( \frac{||s||_{W^{k,p}(M;E)}}{\delta} \right)^{\frac{n}{k}}, \end{equation}
where the constant does depend on the choice of the maps $\{\theta_i\},\{\Psi_i\}$ but this choice has been fixed given $M$ and $E.$ By (3) in Proposition \ref{prop: cubulation} $\{A_i=\theta_i(Q)\}$ is tame, as explained in Subsection \ref{Subsection_Covers}. Applying Proposition \ref{Barcode_Covers} to $\{A_i\}$ and using  \eqref{eq: N face}, we obtain that \[ \cl N_{2^{N+n+1} \delta}(|s|) \leq \sum_{1 \leq i_1 < \ldots < i_l\leq N} \cl N_{2^{n+1}\delta}(|s|_{A_{i_1}\cap \ldots \cap A_{i_l}}|) \leq C_N C_{n,k} C(s,\delta).\] By \eqref{eq: cube to manifold bound} and Lemma \ref{lma: replace C2} this finishes the proof.
\end{proof}

\begin{remark}\label{rmk: negative s detail}
In order to obtain Theorem \ref{thm: main 2} for $|s|$ replaced by $-|s|$ as in Remark \ref{rmk: negative s}, we notice that $s:Q \to \R^l$ satisfies \[d_{C^0}(|s|,\cl{S}_{k-1}(Q)) = d_{C^0}(-|s|,-\cl{S}_{k-1}(Q)),\] where $-\cl{S}_{k-1}(Q) = \{- q\;|\; q \in \cl{S}_{k-1}(Q)\},$ and Proposition \ref{prop: barcode cube} still holds for $p \in -\cl{S}_{k-1}(Q).$ The rest of the proof goes through entirely analogously.  
\end{remark}


\section{Proofs of Applications}
\label{sec:applications}

In this section we prove the applications of Theorem \ref{thm: main 2} and of Theorem \ref{CUBE}. 

We start with a general estimate of Sobolev norms of linear combination of eigenfunctions. 

\begin{prop}\label{prop: Sobolev of eigenchunk}
Let $M$ be a closed Riemannian manifold of dimension $n$ and let $D$ be a non-negative self-adjoint elliptic pseudo-differential operator of order $q$ on the sections of a vector bundle $E$ over $M$ with an inner product. Let $s = \sum_{j=1}^i a_j s_j$ be a linear combination of eigensections $s_j$ of $D$ with eigenvalues $\leq \la$, such that $||s||_{L^2} =1$. Then \[ ||s||_{W^{k,2}} \leq C_{M,E,D,k} (\la+1)^{k/q}. \]
\end{prop}

\begin{proof}
Without loss of generality we can assume that $||s_j||_{L^2} = 1$ for all $j,$ and are moreover orthogonal to each other as $D$ is self-adjoint. Moreover, by possibly adding to $D$ the identity operator $I$ and adjusting the constant $C_{M,E,D,k}$, we may assume that $D$ is positive. 

We may then consider the $q$-th root $D_1$ of $D$ which is a positive self-adjoint elliptic pseudo-differential operator of degree $1$ (\cite{Seeley,Shubin-book}). Note that $D_1$ has exactly the same eigenfunctions as $D,$ but its eigenvalues are $\la^{1/q}$ where $\la$ is an eigenvalue of $D.$

A  fundamental elliptic estimate (see, for example, \cite[Lemma 1.4, p.69]{Shubin92} or \cite[Chapter III, Theorem 5.2(iii), p. 193]{lawson2016spin}) states that \[||s||_{W^{k,2}} \leq C_{M,E,D,k} (||D_1^k s||_{L^2} + ||s||_{L^2}).\]

Now $D_1^k s = \sum \la^{k/q} a_j s_j,$ whence \[ ||D_1^k s||^2 = \sum \la_j^{2k/q} |a_j|^2 \leq \la^{2k/q},\] so \[ ||D_1^k s|| \leq \la^{k/q}.\] In turn we obtain \[ ||s||_{W^{k,2}} \leq C_{M,E,D,k} (\la+1)^{k/q},\] where we absorbed the term $||s||_{L^2} = 1$ into $C_{M,E,D,k} (\la+1)^{k/q}$ by increasing the constant $C_{M,E,D,k}$ suitably. \end{proof}

The case of a manifold with boundary is more complicated, because the boundary conditions play an important role. In particular, the argument via roots of elliptic operators does not apply.

\begin{prop}\label{prop: Sobolev of eigenchunk bdry}
Let $M$ be a compact Riemannian manifold of dimension $n$ with boundary and let $D$ be a non-negative self-adjoint elliptic differential operator of order $q$ with Dirichlet boundary conditions on the sections of a vector bundle $E$ over $M$ with an inner product. Let $s = \sum_{j=1}^i a_j s_j$ be a linear combination of eigensections $s_j$ of the Dirichlet boundary value problem for $D$ with eigenvalues $\leq \la$, such that $||s||_{L^2} =1$. Then for all integers $k \geq 0,$ \[ ||s||_{W^{k,2}} \leq C_{M,E,D,k} (\la+1)^{k/q}. \]
\end{prop}

\begin{proof}
	
First of all, by standard elliptic regularity for every integer $m \geq 0$ and $s \in W_0^{m+q},$ the following coercivity inequality is satisfied \[||s||_{W^{m+q,2}} \leq C_{M,E,D,m} (||D s||_{W^{m,2}} + ||s||_{L^2}).\] 
	
Let us start by proving the statement for an integer multiple $k = l q,$ $l\geq 1,$ of $q$ by induction on $l.$ The base case is $m=0$ in the coercivity estimate from the formulation. The inductive step from $k_0 = lq$ to $k =(l+1)q = k_0 + q$ is again an application of the coercivity estimate: first as $s = \sum a_j s_j$ is a linear combination of eigensections satisfying the homogeneous boundary conditions, so is $Ds = \sum \la_j a_j s_j.$ Note that \[||Ds||_{L^2} = \big(\sum |\la_j|^2|a_j|^2 \big)^{1/2} \leq \la.\]
	
	Therefore by the coercivity estimate and the inductive hypothesis we obtain: \[ ||s||_{W^{k,2}} \leq C_{M,E,D,k}(||Ds||_{W^{k_0,2}} + 1) \leq C'_{M,E,D,k}( \la (\la+1)^{k_0/q} + 1) \leq C'_{M,E,D,k}(\la+1)^{k/q},\] possibly for a different constant $C'_{M,E,D,k}.$
	
	Now, it remains to prove the desired estimate for all $0 \leq k < q.$ Indeed, the same argument as in the inductive step will then yield the estimate in full generality. For $k = 0$ the estimate is trivial. For $0<k<q$ we use the interpolation inequality in Sobolev spaces, which is easy to obtain from 
\cite[Theorem 3.70]{Aubin1}, and the condition $||s||_{L^2}=1$: \[ ||s||_{W^{k,2}} \leq C_{M,k,q} ||s||^{k/q}_{W^{q,2}} ||s||^{1-k/q}_{L^2} \leq C''_{M,E,D,k} (\la+1)^{k/q}.\] 
	
\end{proof}

\begin{remark}
The proof of Proposition \ref{prop: Sobolev of eigenchunk bdry} only relied on the boundary value problem being self-adjoint, non-negative, homogeneous, and satisfying a suitable analogue of the coercivity inequality. This condition appears to hold in more general settings: see e.g. \cite[Section 3.1.1.4]{rempelindex} for a discussion of the pseudo-differential setting. In particular, it holds for the Neumann Laplacian, see \cite[Chapter 5, Proposition 7.2]{Taylor}.
\end{remark}

We require the following basic lemma about persistence modules and their barcodes.

\begin{lemma}\label{lma: delta count}
Let $V_r(f)$ be a persistence module of a function $f:M \to \R$ for $r \in \Z$ and $\cl{B}_r(f)$ be its barcode. Then for all $\delta>0$ and $t \in \R,$  \[\dim \im(\pi_{t,t+\delta}:(V_r(f))_{t} \to (V_r(f))_{t+\delta}) \leq \cl{N}_{r,\delta}(f).\]
\end{lemma}

Indeed, the number on the left hand side counts bars which start before or at $t$ and end after $t+\delta,$ hence their lengths are all greater than $\delta.$ 

\subsection{Proofs of Theorems \ref{thm: main 2-vsp} and \ref{thm: coarse Courant}}

It suffices to observe that in view of Lemma \ref{lma: delta count}, $m_r(s,\delta) \leq \cl N_{\delta'}(-|s|)$ and $z_r(s,\delta) \leq \cl N_{\delta'}(|s|)$ for all $0 < \delta' <\delta.$ The estimate of Theorem \ref{thm: main 2-vsp} is then an immediate consequence of Theorem \ref{thm: main 2}, Remarks \ref{rmk: negative s} and \ref{rmk: 1.12 for r} and taking the limit as $\delta' \to \delta.$ Theorem \ref{thm: coarse Courant} is a direct consequence of Theorem \ref{thm: main 2-vsp} for $p=2$ and Proposition \ref{prop: Sobolev of eigenchunk}. 

\begin{remark}
In fact, the stronger inequality $m_r(s,\delta) \leq \cl N_{r,\delta}(-|s|) \leq \cl N_{\delta}(-|s|)$ holds. Moreover, a similar stronger inequality $z'_r(s,\delta) \leq \cl N_{r,\delta}(|s|) \leq \cl N_{\delta}(|s|)$ holds for the following modification $z'_r(s,\delta)$ of $z_r(s,\delta)$: \[ z'_r(s,\delta) = \dim \mrm{Im}(H_r(Z_s) \to H_r(\{|s| \leq \delta\}))\;.\] (Note the non-strict inequality on the right.) The second observation is not hard to deduce from the upper semi-continuity of the persistence module $V_r(|s|),$ which implies that all bars in $B_r(|s|)$ are closed on the left and all finite bars therein are open on the right, and the fact that both $Z_s$ and $\{|s| \leq \delta \}$ are closed sublevel sets of $|s|$. Similarly, the first observation follows from the lower semi-continuity of $\mathring{V}_r(-|s|).$ 
\end{remark}

\begin{remark}\label{rmk: no 113}
In the case of closed manifolds, we may replace Remark \ref{rmk: negative s} by an argument involving duality. Namely, observe that for any function $f$ (we will be interested in the cases $f=|s|$ and $f=-|s|$) we have \[\cl N_\delta(f) = \cl N_{\delta}^{\mrm{fin}}(f) + b_r(M),\] where $b_r(M) = \dim H_r(M)$ is the $r$-th Betti number of $M.$ Therefore it suffices to bound $\cl N_{\delta}^{\fin}(-|s|)$ and $\cl N_{\delta}^{\fin}(|s|.)$ Now Proposition \ref{prop: duality} implies that $\cl N^{\fin}_{n-r-1, \delta}(-|s|) = \cl N^{\fin}_{r,\delta}(|s|)$ for all $0\leq r < n.$ (In fact this identity is also true for $r<0$ and $r\geq n$ as in these cases it is easy to see that both sides vanish.) Hence $\cl N^{\fin}_{\delta}(-|s|) = \cl N^{\fin}_{\delta}(|s|),$ and therefore it is sufficient to bound only one of these values.
\end{remark}

\subsection{Proof of Theorem \ref{cor: product}}\label{subsec: prod proof}

We prove the following more general statement which readily yields Theorem \ref{cor: product}. Let $M$ be a manifold of dimension $n,$ and $\cl F_{\la}$ is the space of linear combinations of eigenfunctions of a non-negative self-adjoint elliptic pseudo-differential operator $D$ of order $q>0$ with eigenvalues $\leq \la.$  For a number $\la$ set $\ol{\la} = \la+1.$

\begin{theorem}[generalized coarse Courant for products]\label{thm: gen courant} Let $f_1,\ldots,f_l$ be $l$ smooth functions with $f_j \in \cl F_{\la_j}$ and $||f_j||_{L^2} = 1.$ Let $f = f_1\cdot \ldots \cdot f_l.$ Fix integers  $0 \leq r < n$ and $k > n/2$. Then for all $\delta > 0, \al > 0$ \[ m_r(f,\delta) \leq \frac{C_1}{\delta^{n/k}} \left(\sum_{j=1}^l \ol{\la}_j^{(k-n/2-\al)/q}\right)^{n/k}(\ol{\la}_1\cdot \ldots \cdot \ol{\la}_l)^{n(n/2+\al)/kq} +C_2,\] \[z_r(f,\delta) \leq \frac{C_1}{\delta^{n/k}} \left(\sum_{j=1}^l \ol{\la}_j^{(k-n/2-\al)/q}\right)^{n/k}(\ol{\la}_1\cdot \ldots \cdot \ol{\la}_l)^{n(n/2+\al)/kq} +C_2,\] where the constants $C_1,C_2,C_3$ depend only on $M,E,D,k, \al.$ 
\end{theorem}

We use the following fractional Leibniz rule for Sobolev spaces, which holds for instance for $f,g \in W^{s,2}$ where $s>n/2$ \[ ||fg||_{W^{s,2}} \leq C\left( ||f|| _{W^{s,2}} ||g||_{L^{\infty}} + ||f||_{L^{\infty}} ||g|| _{W^{s,2}}\right).\] This estimate is easily verified on $\R^n$ by means of the Fourier transform and then extended to a closed manifold using a partition of unity. For further generalizations and relation to the Kato-Ponce inequality see \cite{Naibo, Grafakos}. Combined with Sobolev's inequality, this yields the following estimate for $k > n/2:$ \begin{equation}\label{est: prod}||f||_{W^{k,2}} \leq C  \sum_{j=1}^l ||f_j||_{W^{k,2}} \prod_{i \neq j} ||f_i||_{W^{n/2+\al,2}},\end{equation} where $C$ depends on $M,n,l$ only. 

By Proposition \ref{prop: Sobolev of eigenchunk} and $f_j \in \cl F_{\la_j}$ this becomes \[||f||_{W^{k,2}} \leq C'  \sum_{j=1}^l (\la_j+1)^{k/q} \prod_{i \neq j} (\la_i+1)^{(n/2+\al)/q}\] for $C'$ depending on $M,E,D,k,\al.$

With this estimate, Theorem \ref{thm: gen courant} follows directly from Theorem \ref{thm: main 2} for $s=f,$ Remark \ref{rmk: 1.12 for r}, and the inequalities $m_r(s,\delta)\leq \cl N_{\delta'}(|s|),$ $z_r(s,\delta) \leq \cl N_{\delta'}(|s|)$ for all $\delta' < \delta.$ \qed

Theorem \ref{cor: product} then follows by replacing all  $\ol{\la}_j$ by $\ol{\la},$ so that $\ol{\la}$ enters with the exponent $n/q +  b/k$ for $b = (l-1)n(n/2+\al)/q$ and taking $k$ large enough so that $b/k < \varepsilon.$

\begin{remark}\label{rmk: trace}
Instead of the fractional Leibniz rule, we could have used the Sobolev trace theorem for restricting $F:M^l \to \R,$ $F(x_1,\ldots,x_l) = f_1(x_1) \cdot \ldots \cdot f_l(x_l)$ to the diagonal $M \cong \Delta \subset M^l$ consisting of points $(x_1,\ldots,x_l)$ with $x_i = x_j$ for all $i,j$ (see \cite[p. 121]{EgorovShubin}). It yields a weaker estimate than \eqref{est: prod}, which is, however, still sufficient to deduce Theorem \ref{cor: product}.
\end{remark}

\subsection{Proof of Theorem \ref{cor: Bezout}}

We prove the following more general result from which Theorem \ref{cor: Bezout} follows directly.

\begin{theorem}[general coarse B\'ezout]\label{thm: Bezout general}  Let $f_j \in \cl F_{\la_j},$ $1 \leq j \leq l$ be $l$ functions. Fix integers $0 \leq r < n$ and $k > n/2$. Then for all $\delta > 0$,
	\[ z_r(s,\delta) \leq \frac{C_1}{\delta^{n/k}} \left(\sum_{j=1}^l (\la_j+1)^{k/q}\right)^{n/k}+C_2,\]
	\[ m_r(s,\delta) \leq \frac{C_1}{\delta^{n/k}} \left(\sum_{j=1}^l (\la_j+1)^{k/q}\right)^{n/k}+C_2,\]
	where $C_1$ depends only on $M,D,k$ and $C_2 = \dim H_r(M).$ 
	
\end{theorem}

In view of Lemma \ref{lma: delta count}, $z_r(s,\delta) \leq \cl N_{r,\delta'}(|s|) \leq \cl N_{\delta'}(|s|)$ for all $\delta'<\delta.$ Furthermore, in view of Proposition \ref{prop: Sobolev of eigenchunk}, \[||s||_{W^{k,2}} \leq \sum_{j=1}^l ||f_j||_{W^{k,2}} \leq C_{M,D} \sum_{j=1}^l (\la_j+1)^{k/q}.\] Therefore this is now a direct consequence of Theorem \ref{thm: main 2} and Remark \ref{rmk: 1.12 for r}.

\subsection{Proofs of Theorems \ref{thm: barcode conj delta} and \ref{thm: barcode conj norm}}

Theorem \ref{thm: barcode conj delta} is a direct application of Theorem \ref{thm: main 2} for $p=2$ together with Proposition \ref{prop: Sobolev of eigenchunk}.

Theorem \ref{thm: barcode conj norm} is proven as follows. Set $p=2.$ Then by Theorem \ref{thm: barcode conj delta} applied once with $n/2<k_1<n$ and once with $k_2 > n,$ we obtain \begin{equation}\label{eq: dominated N delta} \cl{N}_{\delta}(|s|) \leq C_1 (\la+1)^{n/q} \min\{\delta^{-n/k_1},\delta^{-n/k_2}\} + C_2,\end{equation} 
where $C_1, C_2$ are suitable maxima of the constants for the two cases. Note that $\cl{N}_{\delta}(|s|)$ is a measurable function of $\delta$ on $[0,\infty)$ and the right hand side of \eqref{eq: dominated N delta} is integrable on every compact interval in $[0,\infty).$ Therefore by Lebesgue's dominated convergence theorem the function $\cl{N}_{\delta}(|s|)$ is integrable on $[0,\max(|s|)].$ Now \begin{equation}\label{eq:moment estimate} \left|\cl{B}(|s|)\right| \leq \int_{0}^{\max(|s|)} \cl{N}_{\delta}(|s|)\,d\delta.\end{equation} 
Indeed, every finite bar $[a,b)$ contributes $b-a$ to both sides (see \cite[Proof of Moment Lemma]{CSEHM}) and every infinite bar $[c,\infty),$ satisfies $0 \leq c\leq \max(|s|),$ and contributes $\max(|s|)-c$ to the left hand side and $\max(|s|)$ to the right hand side. Now \eqref{eq:moment estimate} and \eqref{eq: dominated N delta} imply that \[ \left|\cl{B}(|s|)\right| \leq C_1 B_{n,k_1,k_2} (\la+1)^{n/q} + C_2 \max(|s|)\] for $B_{n,k_1,k_2} = \int_0^{\infty} \min\{\delta^{-n/k_1},\delta^{-n/k_2}\}\,d\delta < \infty.$ Finally, in view of the Sobolev inequality, Proposition \ref{prop: Sobolev of eigenchunk}, and the choice $n/2 < k_1 < n,$ \[\max(|s|) \leq C_3 ||s||_{W^{k_1,2}} \leq C_4 (\la + 1)^{k_1/q} \leq C_4 (\la+1)^{n/q}.\] Hence we obtain \[ \left|\cl{B}(|s|)\right| \leq C(\la+1)^{n/q},\] with $C = C_1 B_{n,k_1,k_2} + C_2 C_4.$

\begin{remark}\label{rmk: Lp barcode proof}
Here we provide some details of the proof of the estimate from Remark \ref{rmk: Lp barcode}. For the $L^p$ norm, we modify  \eqref{eq:moment estimate} as follows: \[ |\cl{B}(|s|)|_p^p \leq p \int_{0}^{\max(|s|)} \delta^{p-1}\cl{N}_{\delta}(|s|) d\delta.\] Let $n/2 < k_1 < n.$ Then \eqref{eq: dominated N delta} implies that for $p \neq n/k_1,$ \[ |\cl{B}(|s|)|_p^p \leq (\la+1)^{n/q} B_{p,n,k_1,k_2} \max(1,\max(1,|s|)^{p-n/k_1}),\] where $\max(1,-) = \max(1, \max(-)).$ Now for $p-n/k_1 < 0$ \[\max(1,\max(1,|s|)^{p-n/k_1}) = 1,\] whence \[|\cl{B}(|s|)|_p \leq C (\la+1)^{n/pq} \leq C (\la+1)^{n/q},\] whereas for $p-n/k_1 > 0$ \[ \max(1,\max(1,|s|)^{p-n/k_1}) \leq C_5 \max(1,(\la+1)^{(k_1 p -n)/{q}}),\] whence \[ |\cl{B}(|s|)|_p \leq C (\la+1)^{k_1/q} \leq C (\la+1)^{n/q}.\]
\end{remark}

\subsection{Proof of Proposition \ref{prop: no}}

The first part regarding the existence of $f_{i_j}, \la_{i_j}$ for a metric $g_{BLS}$ on $T^2$ is a reformulation of the main result of \cite{BLS}.
The statement on $T^3 = T^2 \times S^1$ with $g_{BLS} \oplus g_{st}$ is a direct calculation. The only part which remains to be proven is the statement regarding $T^4 = T^2 \times T^2$ with $g_{BLS} \oplus g_{BLS}$ and $d_{i_j}(x,y) = f_{i_j}(x) - f_{i_j}(y).$ Clearly $d_{i_j}$ is an eigenfunction of the Laplacian on $T^4$ of eigenvalue $\la_{i_j}.$ Recall that for any function $f$, $\mathring{V}_*(f)_t=H_*(\{ f<t \}).$ By \cite{BLS} the barcode of $\mathring{V}_*(f_{i_j})$ has infinitely many bars $(a_k,b_k],$ $k \in \N,$ in degree $1.$ Respectively, by Proposition \ref{prop: duality}, the barcode of $\mathring{V}_*(-f_{i_j})$ has infinitely many bars $(-b_k,-a_k]$ in degree $0.$ Now, using K\"unneth formula for persistence modules proven in Section \ref{subsec-Kunneth}, we obtain that the barcode of $\mathring{V}_*(d_{i_j})$ contains the infinite family of bars $(a_k-b_k,0]$ in degree $1$ (and the infinite family of bars $(0,b_k-a_k]$ in degree $2,$ which we do not use).  In turn we obtain by definition that $\dim H_1(\{d<0\}) = + \infty.$

\section{Proof of Theorem \ref{thm-sharp mini}}\label{sec: sharpness}
\subsection{Construction}
Let  $ (M,g) $ be a closed Riemannian manifold,  and, as before,  let $ \cF_\lambda $ be the linear span of the Laplace eigenfunctions with eigenvalue $ \leq \lambda $.


The idea of the construction is as follows.  
Take  a smooth function $\phi$ which is supported in a unit ball, takes a positive value at the center, and the same negative value at any point of the sphere of 
radius $1/2$.  Consider  a collection of $\sim \lambda^{n/2}$ small disjoint balls on $M$, and let us transplant $\phi$ to each ball. 
Take the sum $F$ of all these transplanted functions and consider its $L^2$ projection $P$ on the space $\cF_\lambda$. We show that at least on a half of all the balls the remainder $F-P$ is small in the $L^\infty$ norm.  Therefore, on every such ball the function $P$ takes a positive value at the center and a negative value on a sphere in the middle. Taking into account additional control in $\delta$,  one can assure that after renormalization in $L^2$ these values are larger than $\delta$ in absolute value. This implies that at least half of all the balls contain a $\delta$-deep nodal component of the function 
$f=P/\|P\|_{L^2}$, which gives the desired lower bound for $m_0(f,\delta)$ and $z_0(f,\delta)$.

Let us now formalize this idea. Choose a local chart $ U \subset M $ which admits an extension to a slightly larger one. 
For convenience we will consider Euclidean distance $ d_{\eu} (\cdot,\cdot) $ on $ U $ as well as Euclidean balls $ B_\eu(x,\rho) $ for $ x \in U $ and $ \rho > 0 $ (we will always take $ \rho $ small enough so that the Euclidean ball sits in 
$ U $ and consequently can be considered as a subset of $ M $).  The following simple auxiliary lemma holds.

\begin{lemma} \label{lemma:Sobolev-local}
Let $ \eps > 0 $ be small enough. Then for every integer $ l > n/4 $, every $ x \in U $ such that $ B_\eu(x,2\eps) \subset U $, and every smooth function $ f : B_\eu(x,2\eps) \rightarrow \R $, we have
$$ \| f|_{B_\eu(x,\eps)} \|_{L^\infty} \leqslant C \eps^{-n/2} (\eps^{2l} \| \Delta^l f \|_{L^2} + \| f \|_{L^2}) .$$
\end{lemma}
\begin{proof}
The result follows from Sobolev's inequality and the fundamental elliptic estimate (cf. Section \ref{sec:applications}) applied to the rescaled function.  
We leave the details to the reader.
\end{proof}
%
	

Let us now fix some integer $ l > n/4 $ and  a smooth  function $ \phi : \R^n \rightarrow \R $ such that:
\begin{enumerate}
 \item $ \supp(\phi) \subset B(0,1) $.
 \item $ \phi(0) = 1 $.
 \item $ \phi(x) = -1 $ when $ |x| = 1/2 $.
\end{enumerate}

\begin{figure}[ht]
	\begin{center}
		\includegraphics[scale=0.4]{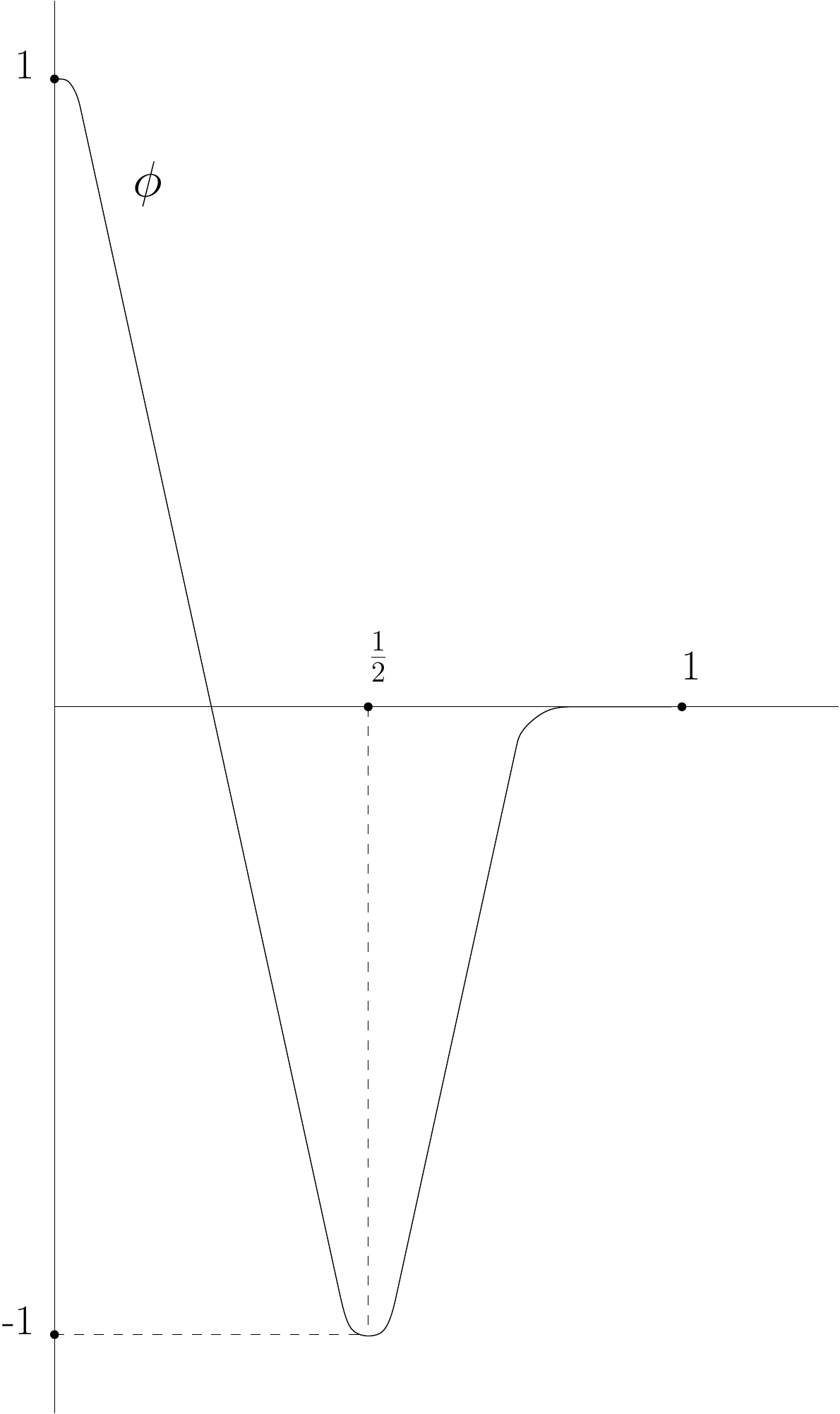}
		\caption{The function $\phi$ can be radial with this profile.}
		\label{bump}
	\end{center}
\end{figure}

Along the proof, constants $ c_i, C_i > 0 $ will depend only on $ (M,g,U,l,\phi) $. Suppose that $ \lambda = \lambda_m $ for a sufficiently large $ m $ so that 
\begin{equation} \label{eq:a-choice}
1 \leqslant \delta \leqslant a \lambda^{n/4} ,
\end{equation}
for some $ a > 0 $. We will later show that we may assume this for the choices of constants that we will make.

Denote $ \eps = (A/ \lambda)^{1/2} $ for some $ A > 1 $.

Consider a collection of disjoint balls   
\begin{equation} \label{eq:ldist-balls}
 B_j = B_\eu(x_j,2\eps) \subset U ,
\end{equation}
$ j = 1,\ldots,N $, such that 
\begin{equation}\label{eq: def N} N = \lfloor a_1 \delta^{-2} \eps^{-n} \rfloor = \lfloor a_1 A^{-n/2} \delta^{-2} \lambda^{n/2} \rfloor \geqslant 1, \end{equation}
for some $ a_1 > 0 $. We are able to do that when $ N \leqslant c\eps^{-n} $ (for $ c=c(M,g,U) $) which holds if
\begin{equation} \label{eq:choice-a1-cond}
0 < a_1 \leqslant a_1(M,g,U)
\end{equation}
(recall that $ \delta \geqslant 1 $), and at the same time when
\begin{equation} \label{eq:choice-a-cond}
a_1a^{-2}A^{-n/2} \geqslant 1   
\end{equation}
so that by $(\ref{eq:a-choice})$ we have $ N \geqslant 1 $. 
Constants $ a, a_1, A $ will be chosen in the course of the proof, and will eventually depend 
only on $ (M,g,U,l,\phi) $.

Define the smooth function $ F : M \rightarrow \R $ by
$$ F(x) = \sum_{j=1}^N \phi((x-x_j) / \eps) $$
for $ x \in U $, and $ F(x) = 0 $ when $ x \in M \setminus U $. The rest of the proof is devoted to showing that we can take the desired 
function $ f $ to be the $ L^2 $-normalized orthogonal $ L^2 $-projection of $ F $ onto $ \cF_\lambda $. 
Denote by $ P : M \rightarrow \R $ the function given by the orthogonal $ L^2 $-projection of $ F $ onto $ \cF_\lambda $, and then denote $ R:= F-P $. First we show that the remainder $ R $ is small in a certain sense.

\subsection{Estimating the remainder} 
Let us prove two technical lemmas.
\begin{lemma} \label{claim:DeltaF-est}
For any integer $ k \geqslant 0 $ we have 
$$ \| \Delta^k F \|_{L^2} \leqslant C N^{1/2} \eps^{-2k+n/2} ,$$
where $ C = C(M,g,U,k,\phi) $.
\end{lemma}
\begin{proof}
By a straightforward computation we have $ | \Delta^k F | \leqslant C\eps^{-2k} $ on each $ B_j $, and we have $ \Delta^k F = 0 $ on the complement of the union of the balls $ B_j $. 
\end{proof}

\begin{lemma} \label{claim:general-trunc-ineq}
Let $ H : M \rightarrow \R $ be a smooth function, denote by $ P_H $ the orthogonal $ L^2 $-projection of $ H $ onto $ \cF_\lambda $, and
then denote $ R_H := H-P_H $ (the remainder). Then
$$ \| R_H \|_{L^2} \leqslant \lambda^{-1} \| \Delta H \|_{L^2}. $$
\end{lemma}
\begin{proof}
Let $ f_0 \equiv 1, f_1,f_2, \ldots $ be an orthogonal basis of $ L^2 $ consisting of eigenfunctions of $ \Delta $, and let
$ \lambda_0 < \lambda_1 < \lambda_2 \leqslant \ldots $ be the corresponding eigenvalues. If we decompose 
$ H = \sum_{j=0}^\infty b_j f_j $, then $ \Delta H = \sum_{j=0}^\infty \lambda_j b_j f_j $, and now the claim follows from the Parseval identity.
\end{proof}

By Lemmas \ref{claim:DeltaF-est} and \ref{claim:general-trunc-ineq} we have 
\begin{equation} \label{eq:R-estimate}
\begin{gathered}
 \eps^{2l} \| \Delta^l R \|_{L^2} + \| R \|_{L^2} \leqslant \eps^{2l}\lambda^{-1} \| \Delta^{l+1} F \|_{L^2} + \lambda^{-1} \| \Delta F \|_{L^2} \\\leqslant C_1 N^{1/2} \eps^{n/2-2} \lambda^{-1} = C_1 A^{-1} N^{1/2} \eps^{n/2} .
\end{gathered}
\end{equation}
Hence 
$$  \int_{M} \left( \eps^{4l} (\Delta^l R(x))^2 +  (R(x))^2 \right) \, d\vol \leqslant  C_1^2 A^{-2} N \eps^{n} .$$
Therefore for at least $ N/2 $ of the $ B_j $'s we have 
$$  \int_{B_j} \left( \eps^{4l} (\Delta^l R(x))^2 +  (R(x))^2 \right) \, d\vol \leqslant  2C_1^2 A^{-2} \eps^{n} ,$$
hence
$$  \eps^{2l} \| \Delta^l R |_{B_j} \|_{L^2} +  \| R|_{B_j} \|_{L^2}  \leqslant  2C_1 A^{-1} \eps^{n/2} ,$$
and then Lemma \ref{lemma:Sobolev-local} implies 
$$ \| R|_{B_j'} \|_{L^\infty} \leqslant C_2 A^{-1} ,$$
where $ B_j' = B_\eu(x_j,\eps) $.

We are now in a position to  complete the proof of Theorem \ref{thm-sharp mini}.
If $ A $ is chosen to be greater than $ 2C_2 $, we conclude that for the function $ P=F-R $ and for 
at least $ N/2 $ of the $ x_j $'s we have 
$$ P(x_j) \geqslant 1/2 $$
and 
$$ P(x) \leqslant -1/2 $$
for $ d_\eu(x,x_j) = \eps /2 $. Also note that by $(\ref{eq:R-estimate}),$ by Lemma  \ref{claim:DeltaF-est} (used with $ k = 0 $), and by the choice of $A,$ we have 
$$ \| P \|_{L^2} = \| F-R \|_{L^2} \leqslant \|F \|_{L^2} + \|R\|_{L^2} \leqslant C_3 N^{1/2}\eps^{n/2} \leqslant C_3 a_1^{1/2} \delta^{-1}. $$
Hence the normalized function 
$$ f:= \frac{P}{\|P\|_{L^2}} $$
has the property that for at least $ N/2 $ of the $ x_j $'s we have 
$$ f(x_j) \geqslant c_1 a_1^{-1/2} \delta $$
and 
$$ f(x) \leqslant -c_1 a_1^{-1/2} \delta $$
for $ d_\eu(x,x_j) = \eps /2 $. Moreover, by \eqref{eq: def N} we have
$$ N/2 \geqslant \frac{1}{4} a_1 \delta^{-2} \eps^{-n} = \frac{1}{4} a_1 A^{-n/2} \delta^{-2} \lambda^{n/2}. $$
Now recall that we can first choose $A = 2C_2.$ Then choose $ a_1 > 0 $ small enough so that we have $ a_1 < c_1^2 $ and $(\ref{eq:choice-a1-cond})$ holds. Then choose $ a := a_1^{1/2} A^{-n/4} $ (according to $(\ref{eq:choice-a-cond})$). Note that these choices of $A, a_1$ can be done so that they depend only on $(M,g,U,l,\phi)$ and hence so does $a.$ As a result we get \[ N/2 \geqslant \frac{1}{4}  a^2 \delta^{-2} \la^{n/2},\] which implies $(\ref{eq:sthm-statement})$ with $c = \frac{1}{4} a^2$ and $\la$ instead of $\la+1$ (note that the right hand side being positive implies that \eqref{eq:a-choice} is satisfied). We can then replace $\la$ by $\la+1$ in $(\ref{eq:sthm-statement})$ by further decreasing $c.$ 

 \qed


We conclude this section by a few remarks.

\begin{remark}
To simplify exposition we stated Theorem \ref{thm-sharp mini} for the Laplace-Beltrami operator. Using similar ideas it is not 
hard to extend it  to arbitrary non-negative self-adjoint elliptic pseudo-differential operators on a closed manifold. 
\end{remark}

\begin{remark}\label{rmk: sharpness}
Recall that Theorem \ref{thm: coarse Courant} gives the bound
\begin{equation} \label{eq:upper-bound}
m_r(f,\delta) \leqslant \frac{C_1}{\delta^{n/k}} (\lambda+1)^{n/2} + C_2 
\end{equation}
for every $ f \in \cF_\lambda $ with $ \| f \|_{L^2} = 1 $ and any $ \delta > 0 $, where $ 0 \leqslant r < n $ and $ k > n/2 $. By Theorems \ref{thm: coarse Courant} and \ref{thm: barcode conj delta}, bounds of the same form hold also for $z_r(f,\delta)$ and $\cl N_\delta(f).$
Theorem \ref{thm-sharp mini} implies sharpness of \eqref{eq:upper-bound} if we fix $ \delta > 0 $ and send $ \lambda $ to infinity. (See Remark \ref{remark:general-r} for the case $r>0.$) Let us now discuss the sharpness in $\delta.$ To this end we compare $\max\{1,\delta^2\}^{-1}$ in Theorem \ref{thm-sharp mini} to $\delta^{-n/k}$ in \eqref{eq:upper-bound}. For the regime $ 0 < \delta \leqslant 1 $, since we can choose $ k $ arbitrarily large, we get that Theorem \ref{thm-sharp mini} implies  ``almost sharpness'' of \eqref{eq:upper-bound} in terms of $ \delta $. Namely for every $\eps>0$, we can obtain $\delta^{-\eps}$ in \eqref{eq:upper-bound}. Also, in some cases, such as the circle or more generally the flat torus $ \R^n / \mathbb{Z}^n $, one cannot improve the example given by the theorem: generic trigonometric polynomals of degree $ \leqslant k $ have no more than $ C k^n $ critical points. 

In the regime $ \delta \gg 1 $, Theorem \ref{thm-sharp mini} does not imply sharpness (or almost sharpness) of $(\ref{eq:upper-bound})$ as stated, since $k>n/2$ is an integer. However it  should not be hard to generalize our approach to non-integer $ k $ (see Remark \ref{rem:fractional}) and  obtain $(\ref{eq:upper-bound})$ for any real $ k $ with $ 2k > n$, which  by Theorem \ref{thm-sharp mini} is ``almost sharp" in this regime as well. Namely, for every $\eps>0$, we can obtain $\delta^{-2+\eps}$ in \eqref{eq:upper-bound}. 
\end{remark}

\begin{remark} \label{remark:general-r}
Note that the proof of Theorem \ref{thm-sharp mini} also provides the same lower bound for $ z_{n-1}(f,\delta) $.
Moreover, by making a different choice of the function $ \phi $ from the proof, we get lower bound for each $ m_r, z_r $ when $ 0 \leqslant r < n $. Namely, instead of taking a point and a sphere around it, we can take an $ r $-dimensional sphere, and require that $ \phi = 1 $ at any point on the $ r $-sphere, while $ \phi = -1 $ at any point from the boundary of its tubular neighbourhood. Instead of the $r$-sphere we can of course take another $ r $-dimensional closed submanifold. 
\end{remark}


\section{Coarse nodal estimates and the Mayer-Vietoris argument}\label{sec: MV bezout prod}

Here we derive an alternative coarse nodal estimate for the union of nodal sets
(see inequality \eqref{eq-improved-vsp} below) and compare it with the one coming
from the product of the eigenfunctions (see Section \ref{sec-prodrev}).

\subsection{Preliminaries}
Let $V,W$ be non-negatively supported persistence modules, that is $W_s=V_s=0$ for all $s < 0$, which
are multiplicatively $C$-interleaved with $C>1$. This means that for every
$s >0$ the compositions
$$V_s \to W_{Cs} \to V_{C^2s},\; W_s \to V_{Cs} \to W_{C^2s}$$
are the persistence morphisms in $V$ and $W$, respectively.

Recall from Section 3 that for $\delta > 0,$ $\cl N^0_{\delta}(V)$ denotes the number of bars of $V$ of length $>\delta$ 
starting at $0$ (see Equation \eqref{eq: N delta 0 def}), and $\cl N_{\delta}(V)$ is the number of all bars of $V$ of length $>\delta$.

\begin{prop}\label{prop-0-vsp}
	\begin{equation}
		\label{eq-N0-vsp}
		\cl N^0_{\delta}(V) \leq \cl N^0_{\delta/C}(W)\;.
	\end{equation}
\end{prop}
\begin{proof} Indeed, the bars starting at $0$ cannot be discarded
	under the multiplicative $C$-matching between the barcodes.
	The result follows from the isometry theorem.
\end{proof}

\subsection{Mayer-Vietoris for thickened nodal sets}
Let $E \to M$ be a vector bundle with an inner product
over a Riemannian manifold $M$. For the sake of simplicity, we are interested
in a pair (as opposed to an arbitrary tuple) of sections $f,g: M \to E$. Put $F(c) = \{||f|| < c\}$, $G(c) = \{||g|| < c\}$.
Consider persistence modules
$$U_f = H_*(F(c)), U_g = H_*(G(c)), V= H_*( F(c) \cup G(c)), W = H_{*-1} (F(c) \cap G(c))\;.$$ These modules are non-negatively supported and we have the piece
$$U_f \oplus U_g \to V \to W$$
of the Mayer-Vietoris sequence. By Corollary \ref{cor: subadd 0} we have
$$ \cl N^0_{2\delta}(V,2\delta) \leq \cl N^0_{\delta}(W) + \cl N_{\delta}(U_f)+ \cl N_{\delta}(U_g)\;.$$
At the same time $W$ is multiplicatively $\sqrt{2}$-interleaved with
$$W'= H_{*-1}(\sqrt{||f||^2+||g||^2} <c)\;.$$
It follows from  Proposition \ref{prop-0-vsp} that
$$ \cl N^0_{\delta}(W) \leq  \cl N^0_{\delta/\sqrt{2}}(W')\;.$$
Combining these inequalities we conclude that
\begin{equation}\label{eq-improved-vsp-0}
	\cl N^0_{2\delta}(V) \leq \cl N^0_{\delta/\sqrt{2}}(W') + \cl N_{\delta}(U_f)+ \cl N_{\delta}(U_g)\;.
\end{equation}

Assume now that $M$ is a surface ($n=2$) and $f,g$ are $L^2$-normalized linear combinations of
eigenfunctions of the Laplacian in $\cl F_{\lambda}.$ Then by the coarse Bezout theorem
$$ \cl N^0_{\delta/\sqrt{2}}(W') \leq C_1 \delta^{-\epsilon}(\lambda+1)+C_2$$
and by the coarse Courant theorem
$$\cl N_{\delta}(U_f)+ \cl N_{\delta}(U_g) \leq C_1 \delta^{-\epsilon}(\lambda+1)+C_2\;.$$
Thus, \eqref{eq-improved-vsp-0} yields the following coarse nodal estimate for the minimum of the absolute values of two eigenfunctions:
\begin{equation}\label{eq-improved-vsp}
	\cl N^0_{\delta}(V) \leq C_1 \delta^{-\epsilon}(\lambda+1)+C_2\;.
\end{equation}

\subsection{Product revisited}\label{sec-prodrev}
In the notations of the previous section, introduce the bundle
$E \otimes E$ with the inner product coming from $E$. Put
$$h(x): = ||f(x) \otimes g(x)|| =||f(x)|| \cdot ||g(x)||\;.$$
Consider the persistence module
$H:= H_*(\{||h|| < c\})$.
Put $$v(x) = \min(||f(x)||,||g(x)||),\; K= \max_{x \in M} (||f(x)||,||g(x)||)\;.$$
With this notation
$$V= H_*( F(c) \cup G(c)) = H_*(\{||v|| < c\})\;.$$

\begin{prop}\label{prop-prod-vsp}
	$\cl N^0_{\delta}(H) \leq \cl N^0_{\delta/K}(V)$.
\end{prop}

\begin{proof}
	We have
	$$v^2 \leq h \leq Kv\,$$
	and hence for $s < \delta^2$ holds
	$$\{h < s\} \subset \{v < \sqrt{s}\} \subset \{v < \delta\} \subset \{h < K\delta\}\;.$$
	Thus,  the persistence map
	$H_s \to V_{K\delta}$ factors through $V_{\sqrt{s}} \to V_\delta$. Taking $s \to 0$,
	we get the proposition.
\end{proof}

\begin{cor}\label{cor-prod} Let $M$ be a surface and $f,g$ be $L^2$-normalized linear combinations of 
	eigenfunctions of the Laplacian in $\cl{F}_{\lambda}$ and whose maximum
	does not exceed $K$,
	Then
	\begin{equation}
		\label{eq-prod-3-vsp}
		\cl N^0_{\delta}(H) \leq  C_1 \delta^{-\epsilon}K^\epsilon(\lambda+1)+C_2\;.
	\end{equation}
\end{cor}

This follows from \eqref{eq-improved-vsp} and Proposition \ref{prop-prod-vsp}.

Since by the Sobolev inequality and Proposition \ref{prop: Sobolev of eigenchunk} the upper bound $K$ can be taken as $\sim (\lambda+1)^{(1+\alpha)/2}$ for any fixed $\alpha>0$, taking $\epsilon$ small we get $(\lambda+1)^{1+\epsilon'}$ in the right hand side
of \eqref{eq-prod-3-vsp}. Thus the approach presented above recovers the bound
on $\cl N^0_{\delta}(H)$ from Theorem \ref{cor: product} but does not improve it. 

\begin{remark}
In this section we have discussed the case where $f,g$ are two sections and $M$ is a surface. It would be interesting to provide an argument along the same lines which works for an arbitrary tuple of sections and in arbitrary dimension. It is likely that the key new ingredient in this approach would be generalizing the coarse B\'{e}zout theorem for the coarse count $z_0$ measured by suitable fiberwise $L^{2m}$-type norms of sections instead of their Euclidean norms.
\end{remark}

\section*{Acknowledgements}

We are grateful to Misha Sodin for numerous helpful conversations, to Dmitry Faifman for useful discussions related to Proposition 1.11, as well as to Gleb Smirnov and Fabian Ziltener for a stimulating remark related to the last bullet in Section 1.2.


\appendix


\section{Morrey-Sobolev inequality on a cube}\label{app: Morrey-Sobolev}

Let $Q$ be a cube of side-length $r$ and diameter $d=\sqrt{n}r.$ Let $B \subset Q$ be the ball of radius $r_1 \leq r/2$ and diameter $d_1 = 2r_1.$ We will prove Theorem \ref{Morrey-Sobolev}, in fact a more precise version thereof, by more carefully calculating the constants in the results of Dupont-Scott \cite{DS80}, see also \cite{Hudson}. Recall that they first prove the following averaged Taylor formula, where the function $\phi$ plays the role of the mollifier with support $B.$ Endow $\R^l$ with the Euclidean metric.

\begin{prop}\label{prop: Taylor average}
Let $\phi \in C^{\infty}_c(B)$ be a smooth function with integral $1.$ Then every $f \in C^{\infty}(Q,\R^l)$ can be written as \[ f = T_k(f) + R_k(f),\] where $T_k(f) \in \cl{P}^l_{k-1}(Q)$ is a polynomial mapping of degree at most $k-1$ given by \[ T_k(f)(x) = \sum_{|\al|<k} \int_B \phi(y) \frac{\partial_{\alpha} f(y)}{\alpha !} (x-y)^{\alpha} dy,\] and the remainder term is given by \[ R_k(f)(x) = \sum_{|\al|=k} 
\int_Q r_{\al}(x,y) \frac{\partial_{\alpha} f(y)}{\alpha !} dy,\] \[r_{\al}(x,y) = k (x-y)^{\al} r(x,y)\] \[r(x,y) = \int_0^1 s^{-n-1} \phi(x+ s^{-1}(y-x)) ds.\]
\end{prop}

They also prove the following estimates: 

\begin{equation}\label{eq: estimate r}
|r(x,y)| \leq \frac{d_1^n}{n} ||\phi||_{L^{\infty}} |x-y|^{-n},
\end{equation}

\begin{equation}\label{eq: estimate r_al}
	|r_{\al}(x,y)| \leq \frac{k d_1^n}{n} ||\phi||_{L^{\infty}} |x-y|^{k-n},
\end{equation}
for all $\alpha$ with $|\al|=k.$

Set $I_k(x) = \chi_{B_0(d)} |x|^{k-n},$ $x\neq 0,$ where $\chi_{B_0(d)}$ is the characteristic function of the ball $B_0(d)$ of radius $d$ around $0.$ Now for a continuous function $g$ on $Q,$ extend it by $0$ to $\R^n$ and set \[I_k(g)(x) = I_k \ast g(x)  = \int_Q I_k(x-y) g(y) dy.\]

For us it is enough to estimate the uniform norm $|R_k (f)|_{L^{\infty}}$ of the remainder. We first estimate pointwise \[ |R_k (f)| \leq \sum_{|\al|=k} \frac{kd_1^n}{\al! n} ||\phi||_{L^{\infty}} I_k(|\del_{\al} f|).\] It is convenient to observe that we can take $\phi$ with $||\phi||_{L^{\infty}} \leq 2 C_n d_1^{-n}$ for $C_n = 2^n \om_n^{-1},$ where $\om_n = \pi^{\frac{n}{2}}/\Gamma(\frac{n}{2} + 1)$ is the volume of the unit $n$-ball, which yields \[ |R_k (f)| \leq 2 k \frac{C_n}{n} \sum_{|\al|=k} \frac{I_k(|\del_{\al} f|)}{\al!} .\] (We could improve the coefficient $2$ to $1 + \eps$ for arbitrary $\eps > 0.$)

It remains to use Young's convolution inequality to estimate $|I_k(g)|$ pointwise in terms of the $L^p$ norm of $g,$ where $k - n/p > 0,$ and apply this to $g = |\del_{\al} f|$ for multi-indices $\al$ with $|\al| = k.$ Indeed for $1/p+ 1/t = 1$ we get \[ |I_k(g)|_{L^{\infty}} \leq |I_k|_{L^t} |g|_{L^p}.\] We calculate $|I_k|_{L^t} = (n\omega_n \int_0^d r^{t(k-n)+n-1} dr)^{1/t} = \left(n\omega_n \frac{d^{t(k-n)+n}}{t(k-n)+n}\right)^{1/t},$ the integrability being ensured by $t(k-n)+n-1>-1 \Leftrightarrow k-n> -n/t   \Leftrightarrow  k-n > -n(1-1/p)   \Leftrightarrow k-n/p > 0.$ In total, we obtain \[ |R_k (f)| \leq 2k \frac{C_n}{n} \left(n\omega_n \frac{d^{t(k-n)+n}}{t(k-n)+n}\right)^{1/t}  \sum_{|\al|=k} \frac{|\del_{\al}f|_{L^p}}{\al! },\] for $1/p+1/t = 1.$ Estimating\footnote{Using H\"{o}lder's inequality at this point yields a more precise yet more complicated bound.} every $|\del_{\al}f|_{L^p}$ by $|D^k f|_{L^p},$ using the fact that $\sum_{|\al|=k} \frac{1}{{\al!}} = \frac{n^k}{k!},$ and calculating the power of $d = \sqrt{n}r,$ we get \begin{equation}\label{eq: estimate} |R_k (f)| \leq B_{n,k,p} r^{k-n/p} \frac{|D^k f|_{L^p}}{k!},\end{equation} 
for $B_{n,k,p} = 2 k C_n n^{k/2-1-n/2p} \left(\frac{n\omega_n}{t(k-n)+n}\right)^{1/t} n^k$ with $1/t = 1 - 1/p.$ This finishes the proof. \qed

\bibliographystyle{abbrv}
\bibliography{bibliographySobolev0126.bib}

\end{document}